\documentclass[11pt, reqno]{amsart}

\usepackage{amsmath, amsfonts, amssymb, amsbsy, bigstrut, graphicx, enumerate,  upref, color, comment, xcolor}

\usepackage{pstricks}

\usepackage[breaklinks]{hyperref}

\newcommand{\ep}{\epsilon}

\newcommand{\mr}{\mathcal{R}}

\newtheorem{thm}{Theorem}[section]
\newtheorem{lmm}[thm]{Lemma}
\newtheorem{cor}[thm]{Corollary}

\newcommand{\bigavg}[1]{\biggl\langle #1 \biggr\rangle}

\newcommand{\cc}{\mathbb{C}}

\newcommand{\ee}{\mathbb{E}}

\newcommand{\ma}{\mathcal{A}}

\newcommand{\cp}{\mathcal{P}}
\newcommand{\cs}{\mathcal{S}}

\newcommand{\mx}{\mathcal{X}}

\newcommand{\ra}{\rightarrow}
\newcommand{\rr}{\mathbb{R}}
\newcommand{\smallavg}[1]{\langle #1 \rangle}

\newcommand{\tr}{\operatorname{Tr}}

\newcommand{\zz}{\mathbb{Z}}

\newcommand{\fpar}[2]{\frac{\partial #1}{\partial #2}}
\newcommand{\spar}[2]{\frac{\partial^2 #1}{\partial #2^2}}
\newcommand{\mpar}[3]{\frac{\partial^2 #1}{\partial #2 \partial #3}}

\newcommand{\fs}{\mathbb{S}}
\newcommand{\fd}{\mathbb{D}}

\newcommand{\fst}{\mathbb{M}}
\newcommand{\fex}{\mathbb{E}}

\numberwithin{equation}{section}

\begin{document}
\title[$SU(N)$ Wilson loop expectations]{Wilson loop expectations in $SU(N)$ lattice gauge theory}
\author{Jafar Jafarov}
\address{\newline Department of Mathematics \newline Stanford University\newline 450 Serra Mall, Building 380 \newline Stanford, CA 94305\newline \textup{\tt jafarov@stanford.edu}}

\begin{abstract}
This article gives a rigorous formulation and proof of the $1/N$ expansion for Wilson loop expectations in strongly coupled $SU(N)$ lattice gauge theory in any dimension. The coefficients of the expansion are represented as absolutely convergent sums over trajectories in a string theory on the lattice, establishing a kind of gauge-string duality. Moreover, it is shown that in large $N$ limit, calculations in $SU(N)$ lattice gauge theory with coupling strength $2\beta$ corresponds to those in $SO(N)$ lattice gauge theory with coupling strength $\beta$ when $|\beta|$ is sufficiently small.
\end{abstract}
\maketitle

\tableofcontents

\section{Introduction}

Introduced by Wilson~\cite{wilson74} in $1974$, lattice gauge theories are  discrete versions of quantum Yang--Mills theories. They were initially investigated to gain a better understanding of quantum Yang--Mills theories, particularly with the aim of explaining the quark confinement phenomenon. The problem of taking a suitable scaling limit of lattice gauge theories, as the lattice spacing approaches zero, is one of the millennium prize problems posed by the Clay Mathematics Institute and remains unsolved to this date.

A lattice gauge theory involves  a lattice, which is usually $\zz^d$ for some $d\ge 2$, a compact Lie group that is called the gauge group of the theory, a matrix representation of the gauge group, and parameter $\beta$ that is called the inverse coupling strength. The main objects of interest in lattice gauge theories are Wilson loop variables and their expectations. The mathematical definitions of lattice gauge theories and Wilson loop variables are given in Section \ref{gauge}. 

When the gauge group is $SU(N)$, $U(N)$, $SO(N)$, $O(N)$ or $Sp(N)$ for some large $N$, the lattice gauge theory is called a ``large $N$ theory''.  The $1/N$ asymptotic expansion of Wilson loop expectations is an important theoretical tool in large $N$ lattice gauge theories. The earliest example of such an expansion was presented by 't Hooft \cite{thooft74} in $1974$. 't Hooft's perturbative expansion of weakly coupled large $N$ theories uses Feynman diagrams to represent the coefficients of the expansion. In a series of two recent papers, Chatterjee~\cite{chatterjee15} and Chatterjee and Jafarov~\cite{chatterjeejafarov16} gave a rigorous proof of a non-perturbative $1/N$ asymptotic expansion of the expectations of Wilson loop variables in $SO(N)$ lattice gauge theory at strong coupling. In these papers the coefficients of the expansion are represented as absolutely convergent sums over trajectories in a certain string theory on the lattice.

The main goal of this paper is to present the $1/N$ expansion of Wilson loop expectations in strongly coupled $SU(N)$ lattice gauge theory. As in \cite{chatterjee15} and \cite{chatterjeejafarov16}, the coefficients in the expansion will be represented as absolutely convergent sums over string trajectories in a kind of string theory on the lattice.

In the physics literature the connections between various lattice gauge theories in large $N$ limit was studied by Lovelace \cite{lovelace82}, who argued that $O(N)$, $U(N)$ and $Sp(N)$ lattice gauge theories have the same Wilson loop expectations in large $N$ limit. The relation between $SO(N)$ and $SU(N)$ gauge theories was also analyzed both theoretically and numerically in a number of papers~\cite{andreasetal15,blakecherman12,uy06}. In this manuscript we will rigorously prove that indeed in the large $N$ limit, $SU(N)$ lattice gauge theory with inverse coupling strength $2\beta$ corresponds to $SO(N)$ lattice gauge theory with inverse coupling strength $\beta$, provided that $|\beta|$ is sufficiently small. 

Although the general approach of this paper is similar to those in \cite{chatterjee15} and \cite{chatterjeejafarov16}, it includes many new results and equations. First of all, the Makeenko--Migdal equations for the $SU(N)$ group, computed here, are different than the equations for $SO(N)$ obtained in \cite{chatterjee15}. Secondly, the loop operations for lattice string theory that shows up in this paper are different in nature than the string theory that arises in the study of $SO(N)$ lattice gauge theory in \cite{chatterjee15, chatterjeejafarov16}. While an operation called  ``twisting'' from the earlier papers does not appear here, a new operation called ``expansion'' comes into play. In other words, whereas the general plan of attack is the same in this paper as in the earlier ones, the actual calculations are quite different and in many cases more complex. Lastly, this paper gives the first rigorous proof of the conjectured correspondence between $SO(N)$ and $SU(N)$ lattice gauge theories in the large $N$ limit.

The rest of the paper is organized as follows. Sections \ref{string} and~ \ref{gauge} introduce the basic definitions  and notations. The main results are presented in Section \ref{main}, together with some additional references to the literature (for extensive references, see \cite{chatterjee15, chatterjeejafarov16}). Section \ref{prelim} contains some preliminary lemmas. The Makeenko--Migdal equations for $SU(N)$ group are derived in Sections \ref{stein} and \ref{sd}. The rest of the paper is devoted to the proofs of the main results.

\section{A string theory on the lattice}\label{string}
Let $d\geq 2$ be an integer and $\zz^d$ be the $d$-dimensional integer lattice. Let $E$ be the set of all directed nearest neighbor edges of $\zz^d$. For any directed edge $e\in E$ let $u(e)$ and $v(e)$ be its starting and ending points, respectively, and represent this edge as $(u(e), v(e))$. If $e=(u, v)$, then denote the edge $(v, u)$ by $e^{-1}$. An edge $e$ is positively oriented if $u(e)$ is lexicographically smaller than $v(e)$ and negatively oriented otherwise. Let $E^+$ and $E^-$ respectively be the sets of all positively edges and the set of all negatively oriented edges in $\zz^d$. It is easy to see that if $e\in E^+$, then $e^{-1}\in E^-$ and vice versa. A path $\rho$ in the lattice $\zz^d$ is a sequence of edges $\rho=e_1e_2\cdots e_n$ such that $u(e_{i+1})=v(e_i)$ for all $1\leq i\leq n-1$. Such a path is called closed if $v(e_n)=u(e_1)$. We will say $\rho$ has length $n$ and denote this by $|\rho|=n$. The null path, $\emptyset$, is a path with no edges. By default the null path is assumed to be closed.

A plaquette $p$ is a closed path $e_1e_2e_3e_4$ of length four such that $e_i\neq e_j^{-1}$ for all $i$, $j$. Let $\cp(e)$ be the set of all plaquettes passing through $e$. If $p=e_1e_2e_3e_4$ is a plaquette, then denote $e_4^{-1}e_3^{-1}e_2^{-1}e_1^{-1}$ by $p^{-1}$. The plaquette $p=e_1e_2e_3e_4$ is called positively oriented if $u(e_1)$ is lexicographically the smallest and $u(e_2)$ is the second smallest among all starting points of its vertices. Otherwise $p$ is called negatively oriented. Let $\cp^+$ denote the set of all positively oriented plaquettes in $\zz^d$. It is easy to see that for any plaquette $p$ either $p$ or $p^{-1}$ belongs to $\cp^+$. Denote the set of all positively oriented plaquettes passing through edge $e$ or $e^{-1}$ by $\cp^+(e)$.

If $\rho=e_1e_2\cdots e_n$ is a closed path, then a path $\rho'$ is said to be cyclically equivalent to $\rho$ if
\[
\rho'=e_ie_{i+1}\cdots e_{n}e_1e_2\cdots e_{i-1}
\]
for some $1\leq i\leq n$. Each equivalence class will be called a cycle. Let the length of a cycle $l$ be the length of any path in this equivalence class and denote it by $|l|$. The equivalence class of the null path is called a null cycle and it has zero length.

We say a path $\rho=e_1e_2\cdots e_n$ has a backtrack at location $i$ if $e_{i+1}=e_i^{-1}$, where $e_{n+1}=e_1$. In this case the path 
$$\rho'=e_1e_2\cdots e_{i-1}e_{i+2}\cdots e_n$$
is said to be obtained from $\rho$ by a backtrack erasure at location $i$. It can be easily checked that $\rho'$ is also a path, and it is closed if $\rho$ is closed. A path without any backtrack is called a non-backtracking path. A loop is the cyclical equivalence class of a closed non-backtracking path. The equivalence class of the null path is called a null loop which is also denoted by $\emptyset$.

The path obtained by successfully deleting backtracks of a closed path $\rho$ until no backtrack is left is called a non-backtracking core of $\rho$ and it is denoted by $[\rho]$. It was proven in \cite{chatterjee15} that a path obtained by deleting backtracks of a closed path until there is no backtrack left does not depend on the order of backtrack erasures. Therefore, the non-backtracking core of a closed path is uniquely defined. Given a cycle $l$ define its non-backtracking core, $[l]$, to be the equivalence class of the non-backtracking core of any element of $l$.  Figure \ref{corefig} shows a closed path and its non-backtracking core.

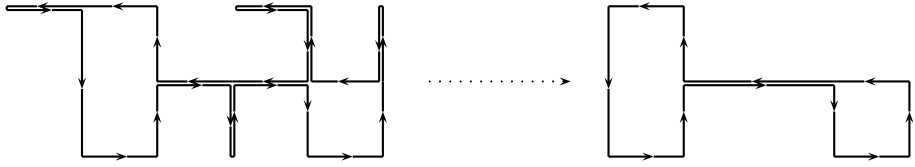
\begin{figure}[t]
\begin{pspicture}(-0.5,0.5)(12.5,3.55)
\psset{xunit=1cm,yunit=1cm}
\psline{->}(1,1)(1.6,1)
\psline{-}(1.6,1)(2,1)
\psline{->}(2,1)(2,1.6)
\psline{-}(2,1.6)(2,1.95)
\psline{->}(2,1.95)(2.6,1.95)
\psline{-}(2.6, 1.95)(2.975, 1.95)
\psline{->}(2.975, 1.95)(2.975, 1.4)
\psline{-}(2.975, 1.4)(2.975, 1)
\psline{-}(2.975, 1)(3.025, 1)
\psline{->}(3.025, 1)(3.025,1.6)
\psline{-}(3.025, 1.6)(3.025, 1.95)
\psline{->}(3.025, 1.95)(3.6,1.95)
\psline{-}(3.6, 1.95)(4, 1.95)
\psline{->}(4, 1.95)(4,1.6)
\psline{-}(4, 1.6)(4, 1)
\psline{->}(4, 1)(4.6,1)
\psline{-}(4.6, 1)(5, 1)
\psline{->}(5, 1)(5,1.6)
\psline{-}(5, 1.6)(5, 2)
\psline{->}(5, 2)(5,2.6)
\psline{-}(5, 2.6)(5, 3)
\psline{-}(5, 3)(4.95, 3)
\psline{->}(4.95, 3)(4.95,2.4)
\psline{-}(4.95, 2.4)(4.95, 2)
\psline{->}(4.95, 2)(4.4,2)
\psline{-}(4.4, 2)(4.05, 2)
\psline{->}(4.05, 2)(4.05, 2.6)
\psline{-}(4.05, 2.6)(4.05, 3)
\psline{->}(4.05, 3)(3.4, 3)
\psline{-}(3.4, 3)(3.05, 3)
\psline{-}(3.05, 3)(3.05, 2.95)
\psline{->}(3.05, 2.95)(3.6, 2.95)
\psline{-}(3.6, 2.95)(4, 2.95)
\psline{->}(4, 2.95)(4, 2.4)
\psline{-}(4, 2.4)(4, 2)
\psline{->}(4, 2)(3.4, 2)
\psline{-}(3.4, 2)(3, 2)
\psline{->}(3, 2)(2.4, 2)
\psline{-}(2.4, 2)(2, 2)
\psline{->}(2, 2)(2, 2.6)
\psline{-}(2, 2.6)(2, 3)
\psline{->}(2, 3)(1.4, 3)
\psline{-}(1.4, 3)(1, 3)
\psline{->}(1, 3)(0.4, 3)
\psline{-}(0.4, 3)(0, 3)
\psline{-}(0, 3)(0, 2.95)
\psline{->}(0, 2.95)(0.6, 2.95)
\psline{-}(0.6, 2.95)(1, 2.95)
\psline{->}(1, 2.95)(1, 1.9)
\psline{-}(1, 1.9)(1, 1)

\psline[linestyle = dotted]{->}(5.5, 2)(7.5, 2)

\psline{->}(8,1)(8.6,1)
\psline{-}(8.6,1)(9,1)
\psline{->}(9,1)(9,1.6)
\psline{-}(9,1.6)(9,1.95)
\psline{->}(9,1.95)(10.1,1.95)
\psline{-}(10.1, 1.95)(11, 1.95)
\psline{->}(11, 1.95)(11, 1.6)
\psline{-}(11, 1.6)(11, 1)
\psline{->}(11, 1)(11.6, 1)
\psline{-}(11.6, 1)(12, 1)
\psline{->}(12, 1)(12,1.6)
\psline{-}(12, 1.6)(12, 2)
\psline{->}(12, 2)(11.4,2)
\psline{-}(11.4, 2)(11, 2)
\psline{->}(11, 2)(9.9, 2)
\psline{-}(9.9, 2)(9, 2)
\psline{->}(9, 2)(9, 2.6)
\psline{-}(9, 2.6)(9, 3)
\psline{->}(9, 3)(8.4, 3)
\psline{-}(8.4, 3)(8, 3)
\psline{->}(8, 3)(8, 1.9)
\psline{-}(8, 1.9)(8, 1)
\end{pspicture}
\caption{A closed path and its nonbactracking core obtained by backtrack erasures.}
\label{corefig}
\end{figure}

Two finite sequences of loops $s=(l_1, l_2, \ldots, l_n)$ and $s'=(l_1', l_2', \ldots, l_m')$ are said to be equivalent if one can be obtained from another by inserting or deleting the null loop at various locations. The set of all equivalence classes is denoted by $\cs$, and an element of this set is called a loop sequence. Given a sequence of loops $s$, its equivalence class in $\cs$ has a representative without any null loop component. This representative is called the minimal representation of $s$. If $(l_1, l_2,\ldots, l_n)$ is the minimal representation of $s$ define its length as $$|s|:=|l_1|+|l_2|+\cdots+|l_n|\, ,$$
size as
$$\#s:=n\, ,$$
and index as
$$\iota(s):=|s|-\#s\, .$$
For a loop $l=e_1e_2\cdots e_n$ define the location of the edge $e_k$ in $l$ to be $k$. This is well defined if there is a fixed representation of each loop. This can be done by ordering the edges of $\zz^d$ in some predefined way, and then writing those edges in $l$ according to this order.

For any positively oriented edge $e$ let $A_r(e)$ be the set of locations in $l_r$ where $e$ occurs. Similarly, let $B_r(e)$ be the set of locations in $l_r$ where $e^{-1}$ occurs. Define $C_r(e):=A_r(e) \cup B_r(e)$. Note that one or more of the sets $A_r(e)$, $B_r(e)$, $C_r(e)$ can be empty. Let $m_r(e)$ be the size of the set $C_r(e)$ and $t_r(e)=|A_r(e)|-|B_r(e)|$. Define 
$$t(e):=t_1(e)+\cdots+t_n(e)\, .$$ 
In words, $t(e)$ denotes the difference between the number occurrences of $e$ and the number of occurrence of $e^{-1}$ in the loop components of $s$. We will simply write $A_r$, $B_r$, $C_r$, $m_r$, $t_r$ and $t$ whenever $e$ is clear from the context. Finally, let
$$\ell(s)=\sum_{e\in E^+} t(e)^2\, .$$
Note that since there are only finitely many edges in the loop components of $s$,  all but finitely many terms in above sum are zero.

Loop operations are ways of obtaining new loops from a given set of one or more loops. There are five types of operations in this paper, called merger, splitting, deformation, expansion and inaction, and the first four each have two subtypes, called positive and negative. Four of these types correspond to loop operations defined in \cite{chatterjee15} and \cite{chatterjeejafarov16}.

\textbf{Merger: } Let $l$ and $l'$ be two non-null loops. If an edge $e$ appears at location $x$ in $l$ and at location $y$ in $l'$, then write $l=aeb$ and $l'=ced$ where $a$, $b$, $c$ and $d$ are paths. Define positive merger of $l$ and $l'$ at locations $x$ and $y$ as
\[
l\oplus_{x,y} l' := [a e dceb]\,,
\]
where $[aedceb]$ is the nonbacktracking core of the cycle $aedceb$, and define the negative merger of $l$ and $l'$ at locations $x$ and $y$ as 
\[
l\ominus_{x,y} l' := [ac^{-1}d^{-1}b]\, .
\]
If an edge $e$ appears at location $x$ in $l$, and $e^{-1}$ appears at location $y$ in $l'$, then write $l=aeb$ and $l'=ce^{-1}d$. Define positive merger of $l$ and $l'$ at locations $x$ and $y$ as
 
\[
l\oplus_{x,y} l' := [ae c^{-1} d^{-1} e b]\, ,
\]
and define the negative merger of $l$ and $l'$ at locations $x$ and $y$ as
\[
l\ominus_{x,y} l' := [adcb]\, ,
\]
Positive merger is illustrated in Figure \ref{posstfig} and negative merger is illustrated in Figure \ref{negstfig}.

\begin{figure}[t]
\begin{pspicture}(0,0)(12,2.5)
\psset{xunit=1cm,yunit=1cm}
\psline{->}(0,0)(1.6,0)
\psline{-}(1.6,0)(3,0)
\psline{->}(3,0)(3,.6)
\psline{-}(3,.6)(3,1)
\psline{->}(3,1)(1.4,1)
\psline{-}(1.4,1)(0,1)
\psline{->}(0,1)(0,.4)
\psline{-}(0,.4)(0,0)

\psline{->}(4,0)(5.1,0)
\psline{-}(5.1,0)(6,0)
\psline{->}(6,0)(6,.6)
\psline{-}(6,.6)(6,1)
\psline{->}(6,1)(5.4,1)
\psline{-}(5.4, 1)(5,1)
\psline{->}(5,1)(5,1.6)
\psline{-}(5,1.6)(5,2)
\psline{->}(5,2)(4.4,2)
\psline{-}(4.4, 2)(4,2)
\psline{->}(4,2)(4,.9)
\psline{-}(4, .9)(4,0)

\psline[linestyle = dotted]{->}(6.5, .5)(8.5, .5)

\psline{->}(9,0)(9.6,0)
\psline{-}(9.6,0)(10,0)

\psline{->}(10,0)(11.1,0)
\psline{-}(11.1,0)(12,0)

\psline{->}(12,0)(12,.6)
\psline{-}(12,.6)(12,1)
\psline{->}(12,1)(11.4,1)
\psline{-}(11.4,1)(11,1)
\psline{->}(11,1)(11,1.6)
\psline{-}(11,1.6)(11,2)
\psline{->}(11,2)(10.4,2)
\psline{-}(10.4,2)(10,2)
\psline{->}(10,2)(10,.9)
\psline{-}(10,.9)(10,.05)
\psline{->}(10,.05)(11.1,.05)
\psline{-}(11.1,.05)(11.95,.05)
\psline{->}(11.95,.05)(11.95,.6)
\psline{-}(11.95,.6)(11.95,.95)
\psline{->}(11.95,0.95)(11.4,0.95)
\psline{-}(11.4,0.95)(11,0.95)
\psline{->}(11,0.95)(10.4,0.95)
\psline{-}(10.4,0.95)(10.06,0.95)
\psline{->}(9.94,0.95)(9.4,.95)
\psline{-}(9.4,.95)(9,0.95)
\psline{->}(9,0.95)(9,.4)
\psline{-}(9,.4)(9,0)

\rput(3.5,.5){$\oplus$}
\rput(2.8,.5){$e$}
\rput(5.8,.5){$e$}
\rput(11.75,.5){$e$}
\end{pspicture}
\caption{Positive merger.}
\label{posstfig} 
\end{figure}
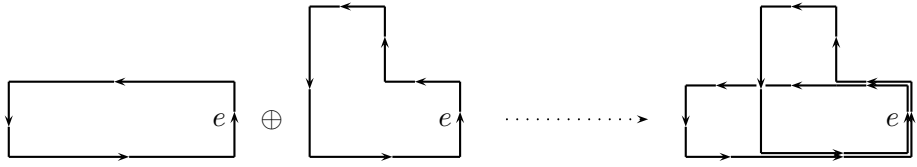

\begin{figure}[t]
\begin{pspicture}(0,0)(11.5,2.5)
\psset{xunit=1cm,yunit=1cm}
\psline{->}(0,0)(1.6,0)
\psline{-}(1.6,0)(3,0)
\psline{->}(3,0)(3,.6)
\psline{-}(3,.6)(3,1)
\psline{->}(3,1)(1.4,1)
\psline{-}(1.4,1)(0,1)
\psline{->}(0,1)(0,.4)
\psline{-}(0,.4)(0,0)

\psline{->}(6,0)(4.9,0)
\psline{-}(4.9,0)(4,0)
\psline{->}(6,1)(6,.4)
\psline{-}(6,.4)(6,0)
\psline{->}(5,1)(5.6,1)
\psline{-}(5.6, 1)(6,1)
\psline{->}(5,2)(5,1.4)
\psline{-}(5,1.4)(5,1)
\psline{->}(4,2)(4.6,2)
\psline{-}(4.6, 2)(5,2)
\psline{->}(4,0)(4,1.1)
\psline{-}(4,1.1)(4,2)

\psline[linestyle = dotted]{->}(7, .5)(9, .5)

\psline{->}(9.5,0)(10.1,0)
\psline{-}(10.1,0)(10.5,0)
\psline{->}(10.5,0)(10.5,1.1)
\psline{-}(10.5,1.1)(10.5,2)
\psline{->}(10.5,2)(11.1,2)
\psline{-}(11.1,2)(11.5,2)
\psline{->}(11.5,2)(11.5,1.4)
\psline{-}(11.5,1.4)(11.5,1)
\psline{->}(11.5,1)(10.9,1)
\psline{-}(10.9,1)(10.6,1)
\psline{->}(10.4,1)(9.9,1)
\psline{-}(9.9,1)(9.5,1)
\psline{->}(9.5,1)(9.5,.4)
\psline{-}(9.5,.4)(9.5,0)

\rput(3.5,.5){$\ominus$}
\rput(2.8,.5){$e$}
\rput(6.4,.5){$e^{-1}$}
\end{pspicture}
\caption{Negative merger.}
\label{negstfig} 
\end{figure}
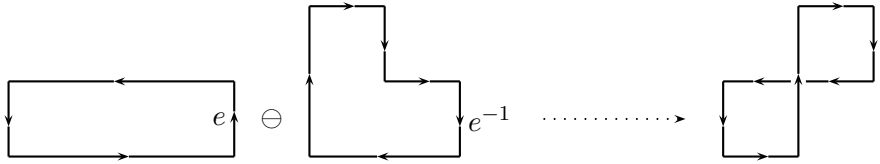

\textbf{Deformation: }Let $l$ be a non-null loop sequence and let $p$ be a plaquette. Suppose that $l$ contains an edge $e$ at location $x$, and $p$ passes through $e$ or $e^{-1}$. Note that since $p$ can pass through $e$ or its inverse $e^{-1}$ only once, the location $y$ of this edge in $p$ is unique. The positive deformation of $l$ at location $x$ by $p$ is defined to be the positive merger of $l$ and $p$ at locations $x$ and $y$, and it is denoted by $l\oplus_{x}p$. Similarly, the negative deformation of $l$ at location $x$ by $p$ is defined to be the negative merger of $l$ and $p$ at locations $x$ and $y$, and it is denoted by $l\ominus_{x}p$. Figure \ref{posdefig} illustrates a positive deformation and Figure \ref{negdefig} illustrates a negative deformation.

\begin{figure}[t]
\begin{pspicture}(0,0)(11,2.5)
\psset{xunit=1cm,yunit=1cm}
\psline{->}(0,0)(1.1,0)
\psline{-}(1.1,0)(2,0)
\psline{->}(2,0)(2,.6)
\psline{-}(2,.6)(2,1)
\psline{->}(2,1)(1.4,1)
\psline{-}(1.4,1)(1,1)
\psline{->}(1,1)(1,1.6)
\psline{-}(1,1.6)(1,2)
\psline{->}(1,2)(.4,2)
\psline{-}(.4,2)(0,2)
\psline{->}(0,2)(0,.9)
\psline{-}(0,0.9)(0,0)

\psline{->}(3.6,0)(4.2,0)
\psline{-}(4.2,0)(4.6,0)
\psline{->}(4.6,0)(4.6,.6)
\psline{-}(4.6,.6)(4.6,1)
\psline{->}(4.6,1)(4,1)
\psline{-}(4,1)(3.6,1)
\psline{->}(3.6,1)(3.6,.4)
\psline{-}(3.6,.4)(3.6,0)

\psline[linestyle = dotted]{->}(5.1, .5)(7.1, .5)

\psline{->}(7.6,0)(8.7,0)
\psline{-}(8.7,0)(9.6,0)
\psline{->}(9.6,0)(9.6,.6)
\psline{-}(9.6,.6)(9.6,0.95)
\psline[linearc=0.07]{-}(9.6,0.95)(9.6,1)(9.65,1)
\psline{->}(9.65,1)(10.2,1)
\psline{-}(10.2,1)(10.6,1)
\psline{->}(10.6,1)(10.6,.4)
\psline{-}(10.6,.4)(10.6,0)
\psline{->}(10.6,0)(10,0)
\psline{-}(10,0)(9.65,0)
\psline{->}(9.65,0)(9.65,.6)
\psline{-}(9.65,.6)(9.65,0.95)
\psline[linearc=0.07]{-}(9.65,0.95)(9.65,1)(9.6,1)
\psline{->}(9.6,1)(9,1)
\psline{-}(9,1)(8.6,1)
\psline{->}(8.6,1)(8.6,1.6)
\psline{-}(8.6,1.6)(8.6,2)
\psline{->}(8.6,2)(8,2)
\psline{-}(8,2)(7.6,2)
\psline{->}(7.6,2)(7.6,.9)
\psline{-}(7.6,0.9)(7.6,0)

\rput(1.8, 0.5){$e$}
\rput(9.4, 0.5){$e$}
\rput(3.2, 0.5){$e^{-1}$}
\rput(4.1, .5){$p$}
\rput(2.5, .5){$\oplus$}
\end{pspicture}
\caption{Positive deformation.}
\label{posdefig}
\end{figure}
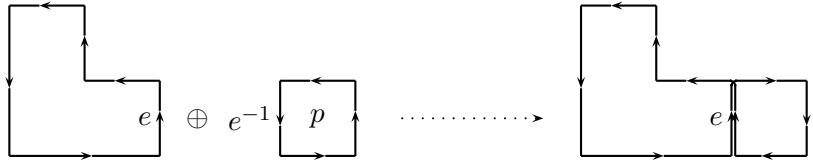

\begin{figure}[t]
\begin{pspicture}(0,0)(11,2.5)
\psset{xunit=1cm,yunit=1cm}
\psline{->}(0,0)(1.1,0)
\psline{-}(1.1,0)(2,0)
\psline{->}(2,0)(2,.6)
\psline{-}(2,.6)(2,1)
\psline{->}(2,1)(1.4,1)
\psline{-}(1.4,1)(1,1)
\psline{->}(1,1)(1,1.6)
\psline{-}(1,1.6)(1,2)
\psline{->}(1,2)(.4,2)
\psline{-}(.4,2)(0,2)
\psline{->}(0,2)(0,.9)
\psline{-}(0,0.9)(0,0)

\psline{->}(3.2,0)(3.2,.6)
\psline{-}(3.2,.6)(3.2,1)
\psline{->}(3.2,1)(3.8,1)
\psline{-}(3.8,1)(4.2,1)
\psline{->}(4.2,1)(4.2,.4)
\psline{-}(4.2,.4)(4.2,0)
\psline{->}(4.2,0)(3.6,0)
\psline{-}(3.6,0)(3.2,0)

\psline[linestyle = dotted]{->}(4.7, .5)(6.7, .5)

\psline{->}(7.2,0)(8.8,0)
\psline{-}(8.8,0)(10.2,0)
\psline{->}(10.2,0)(10.2,.6)
\psline{-}(10.2,.6)(10.2,1)
\psline{->}(10.2,1)(9.1,1)
\psline{-}(9.1,1)(8.2,1)
\psline{->}(8.2,1)(8.2,1.6)
\psline{-}(8.2,1.6)(8.2,2)
\psline{->}(8.2,2)(7.6,2)
\psline{-}(7.6,2)(7.2,2)
\psline{->}(7.2,2)(7.2,.9)
\psline{-}(7.2,0.9)(7.2,0)

\rput(1.8, 0.5){$e$}
\rput(3, 0.5){$e$}
\rput(3.7, .5){$p$}
\rput(2.5, .5){$\ominus$}

\end{pspicture}
\caption{Negative deformation.}
\label{negdefig}
\end{figure}
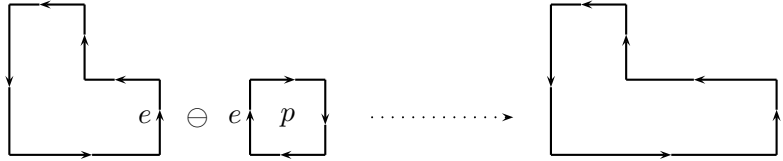

\textbf{Splitting:} Let $l$ be a non-null loop sequence. If $l$ contains an edge $e$ at two distinct locations $x$ and $y$, then write $l=aebced$. Define the positive splitting of $l$ at $x$ and $y$ to be the pair of loops 
\[
\times^1_{x,y} l := [aec]\, , \ \ \times^2_{x,y} l := [be]\, .
\]
If $l$ contains $e$ at location $x$ and $e^{-1}$ at location $y$, write $l = aeb e^{-1} c$ and define the negative splitting of $l$ at $x$ and $y$ to be the pair of loops
\[
\times^1_{x,y} l := [ac]\, , \ \ \times^2_{x,y} l := [b]\, .
\]
Figure \ref{posspfig} illustrates a positive splitting and Figure \ref{negspfig} illustrates a negative splitting.

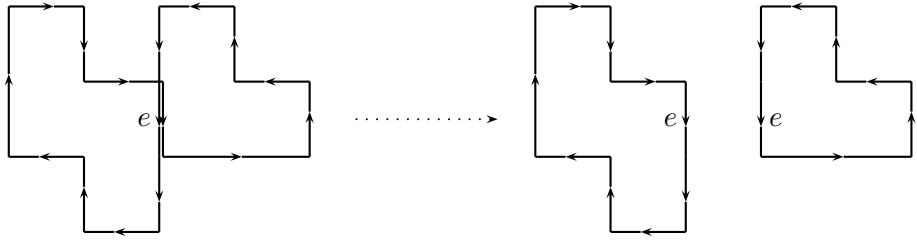
\begin{figure}[t]
\begin{pspicture}(0,0)(12,3.5)
\psset{xunit=1cm,yunit=1cm}
\psline{->}(1, 0)(1,.6)
\psline{-}(1,.6)(1,1)
\psline{->}(1,1)(.4,1)
\psline{-}(.4,1)(0,1)
\psline{->}(0,1)(0,2.1)
\psline{-}(0,2.1)(0,3)
\psline{->}(0,3)(.6,3)
\psline{-}(.6,3)(1,3)
\psline{->}(1,3)(1,2.4)
\psline{-}(1,2.4)(1,2)
\psline{->}(1,2)(1.6,2)
\psline{-}(1.6,2)(2.05,2)
\psline{->}(2.05, 2)(2.05, 1.4)
\psline{-}(2.05,1.4)(2.05,1)
\psline{->}(2.05,1)(3.1,1)
\psline{-}(3.1,1)(4,1)
\psline{->}(4,1)(4,1.6)
\psline{-}(4,1.6)(4,2)
\psline{->}(4,2)(3.4,2)
\psline{-}(3.4,2)(3,2)
\psline{->}(3,2)(3,2.6)
\psline{-}(3,2.6)(3,3)
\psline{->}(3,3)(2.4,3)
\psline{-}(2.4, 3)(2,3)
\psline{->}(2,3)(2,2.4)
\psline{-}(2,2.4)(2,2)
\psline{->}(2,2)(2,1.4)
\psline{-}(2,1.4)(2,1)
\psline{->}(2,1)(2,.4)
\psline{-}(2,.4)(2,0)
\psline{->}(2,0)(1.4,0)
\psline{-}(1.4,0)(1,0)

\psline[linestyle = dotted]{->}(4.5, 1.5)(6.5, 1.5)

\psline{->}(8, 0)(8,.6)
\psline{-}(8,.6)(8,1)
\psline{->}(8,1)(7.4,1)
\psline{-}(7.4,1)(7,1)
\psline{->}(7,1)(7,2.1)
\psline{-}(7,2.1)(7,3)
\psline{->}(7,3)(7.6,3)
\psline{-}(7.6,3)(8,3)
\psline{->}(8,3)(8,2.4)
\psline{-}(8,2.4)(8,2)
\psline{->}(8,2)(8.6,2)
\psline{-}(8.6,2)(9,2)
\psline{->}(9,2)(9,1.4)
\psline{-}(9,1.4)(9,1)
\psline{->}(9,1)(9,.4)
\psline{-}(9,.4)(9,0)
\psline{->}(9,0)(8.4,0)
\psline{-}(8.4,0)(8,0)

\psline{->}(10, 2)(10, 1.4)
\psline{-}(10,1.4)(10,1)
\psline{->}(10,1)(11.1,1)
\psline{-}(11.1,1)(12,1)
\psline{->}(12,1)(12,1.6)
\psline{-}(12,1.6)(12,2)
\psline{->}(12,2)(11.4,2)
\psline{-}(11.4,2)(11,2)
\psline{->}(11,2)(11,2.6)
\psline{-}(11,2.6)(11,3)
\psline{->}(11,3)(10.4,3)
\psline{-}(10.4,3)(10,3)
\psline{->}(10,3)(10,2.4)
\psline{-}(10,2.4)(10,2)

\rput(1.8,1.5){$e$}
\rput(8.8,1.5){$e$}
\rput(10.2,1.5){$e$}
\end{pspicture}
\caption{Positive splitting.}
\label{posspfig}
\end{figure}

\begin{figure}[t]
\begin{pspicture}(0,0)(11,3.5)
\psset{xunit=1cm,yunit=1cm}
\psline{->}(0,0)(0,1.1)
\psline{-}(0,1.1)(0,2)
\psline{->}(0,2)(.6,2)
\psline{-}(.6,2)(1,2)
\psline{->}(1,2)(1,1.4)
\psline{-}(1,1.4)(1,1)
\psline{->}(1,1)(1.6,1)
\psline{-}(1.6,1)(1.95,1)
\psline{->}(1.95,1)(1.95,1.6)
\psline{-}(1.95,1.6)(1.95, 2)
\psline{->}(1.95, 2)(2.6,2)
\psline{-}(2.6,2)(3,2)
\psline{->}(3, 2)(3,.9)
\psline{-}(3,.9)(3,0)
\psline{->}(3,0)(3.6,0)
\psline{-}(3.6,0)(4,0)
\psline{->}(4,0)(4,1.6)
\psline{-}(4,1.6)(4,3)
\psline{->}(4,3)(2.9, 3)
\psline{-}(2.9,3)(2,3)
\psline{->}(2,3)(2,2.4)
\psline{-}(2,2.4)(2,2)
\psline{->}(2,2)(2,1.4)
\psline{-}(2,1.4)(2,1)
\psline{->}(2,1)(2,.4)
\psline{-}(2,.4)(2,0)
\psline{->}(2,0)(.9,0)
\psline{-}(.9,0)(0,0)

\psline[linestyle = dotted]{->}(4.5, 1.5)(6.5, 1.5)

\psline{->}(7,0)(7,1.1)
\psline{-}(7,1.1)(7,2)
\psline{->}(7,2)(7.6,2)
\psline{-}(7.6,2)(8,2)
\psline{->}(8,2)(8,1.4)
\psline{-}(8,1.4)(8,1)
\psline{->}(8,1)(8.6,1)
\psline{-}(8.6,1)(9,1)
\psline{->}(9,1)(9,.4)
\psline{-}(9,.4)(9,0)
\psline{->}(9,0)(7.9,0)
\psline{-}(7.9,0)(7,0)

\psline{->}(9, 2)(9.6,2)
\psline{-}(9.6,2)(10,2)
\psline{->}(10, 2)(10,.9)
\psline{-}(10,.9)(10,0)
\psline{->}(10,0)(10.6,0)
\psline{-}(10.6,0)(11,0)
\psline{->}(11,0)(11,1.6)
\psline{-}(11,1.6)(11,3)
\psline{->}(11,3)(9.9, 3)
\psline{-}(9.9,3)(9,3)
\psline{->}(9,3)(9,2.4)
\psline{-}(9,2.4)(9,2)

\rput(1.75,1.5){$e$}
\rput(2.4,1.58){$e^{-1}$}
\end{pspicture}
\caption{Negative splitting.}
\label{negspfig}
\end{figure}
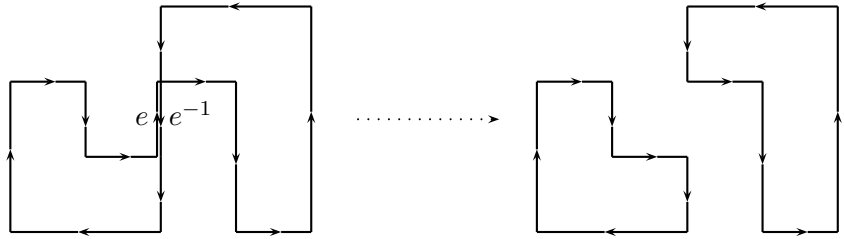

\textbf{Expansion:} Let $l$ be a non-null loop which contains an edge $e$ at location $x$. A positive expansion of $l$ at location $x$ by a plaquette $p$ passing through $e^{-1}$ replaces $l$ with the pair of loops $(l,p)$. A negative expansion of $l$ at location $x$ by a plaquette $p$ passing through $e$ also replaces $l$ with the pair of loops $(l,p)$. Although positive and negative deformations essentially yield the same result, the names are different to indicate whether the loop $l$ and the plaquette $p$ share the same edge $e$ or they have opposite edges $e$ and $e^{-1}$. Figure \ref{posexfig} illustrates a  positive expansion, and Figure \ref{negexfig} illustrates a negative expansion.

\begin{figure}[t]
\begin{pspicture}(0,0)(8,1.5)
\psset{xunit=1cm,yunit=1cm}
\psline{->}(0, 0)(0,.6)
\psline{-}(0,.6)(0,1)
\psline{->}(0,1)(1.1,1)
\psline{-}(1.1,1)(2,1)
\psline{->}(2,1)(2,.4)
\psline{-}(2, .4)(2,0)
\psline{->}(2,0)(.9,0)
\psline{-}(.9,0)(0,0)

\psline[linestyle = dotted]{->}(2.5, .5)(4.5, .5)

\psline{->}(5,0)(5,.6)
\psline{-}(5,.6)(5,1)
\psline{->}(5,1)(6.1,1)
\psline{-}(6.1,1)(7,1)
\psline{->}(7,1)(7,.4)
\psline{-}(7,.4)(7,0)
\psline{->}(7,0)(5.9,0)
\psline{-}(5.9,0)(5,0)

\psline{->}(7.05,0)(7.05,.6)
\psline{-}(7.05,.6)(7.05,1)
\psline{->}(7.05,1)(7.6,1)
\psline{-}(7.6,1)(8,1)
\psline{->}(8,1)(8,.4)
\psline{-}(8,.4)(8,0)
\psline{->}(8,0)(7.4,0)
\psline{-}(7.4,0)(7.05,0)

\rput(1.8,.5){$e$}
\rput(6.8,.5){$e$}
\rput(7.5,.5){$p$}
\end{pspicture}
\caption{Positive expansion.}
\label{posexfig}
\end{figure}
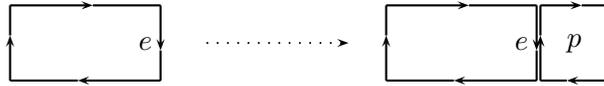

\begin{figure}[t]
\begin{pspicture}(0,0)(8,1.5)
\psset{xunit=1cm,yunit=1cm}
\psline{->}(0, 0)(0,.6)
\psline{-}(0,.6)(0,1)
\psline{->}(0,1)(1.1,1)
\psline{-}(1.1,1)(2,1)
\psline{->}(2,1)(2,.4)
\psline{-}(2, .4)(2,0)
\psline{->}(2,0)(.9,0)
\psline{-}(.9,0)(0,0)

\psline[linestyle = dotted]{->}(2.5, .5)(4.5, .5)

\psline{->}(5,0)(5,.6)
\psline{-}(5,.6)(5,1)
\psline{->}(5,1)(6.1,1)
\psline{-}(6.1,1)(7,1)
\psline{->}(7,1)(7,.4)
\psline{-}(7,.4)(7,0)
\psline{->}(7,0)(5.9,0)
\psline{-}(5.9,0)(5,0)

\psline{->}(7.05,0)(7.6,0)
\psline{-}(7.6,0)(8,0)
\psline{->}(8,0)(8,.6)
\psline{-}(8,.6)(8,1)
\psline{->}(8,1)(7.4,1)
\psline{-}(7.4,1)(7.05,1)
\psline{->}(7.05,1)(7.05,.4)
\psline{-}(7.05,.4)(7.05,0)

\rput(1.8,.5){$e$}
\rput(6.8,.5){$e$}
\rput(7.5,.5){$p$}
\end{pspicture}
\caption{Negative expansion.}
\label{negexfig}
\end{figure}
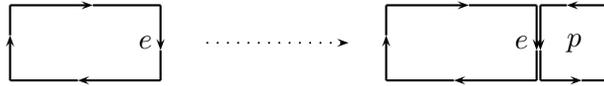

\textbf{Inaction:} The inaction operation does not change anything, and leaves the loop as it is. 

Let $s$ be a loop sequence. If $s'$ is obtained by merging one loop of $s$ to another one, we say that $s'$ is a merger of $s$. Similarly, $s'$ is a deformation of $s$ if it is created by applying deformation operation to one of the loop components of $s$, $s'$ is a splitting of $s$ if it is obtained by splitting one of the loop components $s$, and  $s'$ is an expansion $s$ if it is created by applying expansion operation to one of the loop components of $s$. If inaction is applied, then obviously $s'=s$. Let
\begin{align*}
\fd^+(s) &:= \{s': \text{ $s'$ is a positive deformation of $s$}\}\, ,\\
\fd^-(s) &:= \{s': \text{ $s'$ is a negative deformation of $s$}\}\, ,\\
\fs^+(s) &:= \{s': \text{ $s'$ is a positive splitting of $s$}\}\, ,\\
\fs^-(s) &:= \{s': \text{ $s'$ is a negative splitting of $s$}\}\, ,\\
\fst^+(s) &:= \{s': \text{ $s'$ is a positive merger of $s$}\}\, ,\\
\fst^-(s) &:= \{s': \text{ $s'$ is a negative merger of $s$}\}\, ,\\
\fex^+(s) &:= \{s': \text{ $s'$ is a positive expansion of $s$}\}\, ,\\
\fex^-(s) &:= \{s': \text{ $s'$ is a negative expansion of $s$}\}\, .
\end{align*}
Let $\fd(s)=\fd^+(s)\cup \fd^-(s)$, and define $\fs(s)$, $\fst(s)$ and $\fex(s)$ similarly.

Note that if a loop $l$ contains an edge $e$ at more than one location, then the expansions of $l$ at these locations by a plaquette $p$ passing through $e$ or $e^{-1}$ yield the same results, but we consider them as separate expansions. If $l$ contains edges $e$ or $e^{-1}$ at more than one location, say $x$ and $y$, then we can split $l$ at either $(x, y)$ or $(y,x)$, yielding the same pair of loops but in different order. Since the order of loops in a loop sequence is important, we will count these splittings separately. Similarly, if $s=(l_1, l_2, \ldots, l_n)$ and both $l_1$ and $l_r$ contain edges $e$ or $e^{-1}$ at locations $x$ and $y$, respectively, then we can merge $l_r$ to $l_1$ and obtain $s'=(l_1\oplus_{x,y}l_r, l_{2}, \ldots, l_{r-1}, l_{r+1},\ldots, l_n)$, and we can merge $l_1$ to $l_r$ and obtain $s''=(l_2, \ldots, l_{r-1}, l_r\oplus_{y,x}l_1, l_{r+1}, \ldots,l_n)$. Although $l_1\oplus_{x,y}l_r$ and $l_r\oplus_{y,x}l_1$ are the same loops, and $s'$ and $s''$ have the same loop components, we count them as separate mergers of $s$.

A trajectory is a sequence of loop sequences $(s_0, s_1, s_2, \ldots)$ such that for each non-negative $i$ the loop sequence $s_{i+1}$ is obtained from $s_i$ by applying one of the five loop operations. A trajectory can be finite or infinite. A finite trajectory whose last element is a null loop sequence and any other element is a non-null sequence is called a vanishing trajectory. Let $\mx_{a, b, c, d}(s)$ be the set of all vanishing trajectories starting with a loop sequence $s$ and containing $a$ deformations, $b$ expansions, $c$ mergers and $d$ inactions. We will later prove that the number of splitting operations in any of these trajectories is bounded in terms of $a$, $b$, $c$ and $s$. Let
\[
\mx_{a, b, k}(s):=\bigcup\limits_{c+d=k}\mx_{a, b, c, d}(s)\, ,
\] 
and
\[
\mx_{i, k}(s):=\bigcup\limits_{a+b=i}\mx_{a, b, k}(s)\,.
\]
Finally, let
\[
\mx_{k}(s):=\bigcup\limits_{i=0}^{\infty}\mx_{i, k}(s)\,.
\]
In words, $\mx_k(s)$ is the set of all vanishing trajectories starting at a loop sequence $s$ such that the total number of mergers and inactions is equal to $k$.  Figure \ref{trajfig} illustrates a vanishing trajectory. The performed operations are negative deformation, negative splitting, inaction, positive expansion, negative merger, negative deformation, negative merger,  negative deformation,  negative deformation,  negative deformation and  negative deformation. If $s$ is the loop sequence at the top left corner, then this trajectory is an element of $\mx_{6, 1, 2, 1}(s)$, $\mx_{6, 1, 3}(s)$, $\mx_{7,3}(s)$ and $\mx_{7}(s)$.

\begin{figure}[t]
\begin{pspicture}(0,6.5)(15,20.5)
\psline{->}(0,20)(1.6,20)
\psline{-}(1.6,20)(3,20)
\psline{->}(3,20)(3,19.4)
\psline{-}(3,19.4)(3,19)
\psline{->}(3,19)(1.4,19)
\psline{-}(1.4,19)(0,19)
\psline{->}(0,19)(0,19.6)
\psline{-}(0,19.6)(0,20)

\psline[linestyle = dotted]{->}(3.5, 19.5)(5.5, 19.5)

\psline{->}(6,20)(7.6,20)
\psline{-}(7.6,20)(9,20)
\psline{->}(9,20)(9,19.4)
\psline{-}(9,19.4)(9,19)
\psline{->}(9,19)(8.4,19)
\psline{-}(8.4,19)(8,19)
\psline{->}(8,19)(8,19.6)
\psline{-}(8,19.6)(8,19.95)
\psline{->}(8,19.95)(7.4,19.95)
\psline{-}(7.4,19.95)(7,19.95)
\psline{->}(7,19.95)(7,19.4)
\psline{-}(7,19.4)(7,19)
\psline{->}(7,19)(6.4,19)
\psline{-}(6.4,19)(6,19)
\psline{->}(6,19)(6,19.6)
\psline{-}(6,19.6)(6,20)

\psline[linestyle = dotted]{->}(9.5, 19.5)(11.5, 19.5)

\psline{->}(12,20)(12.6,20)
\psline{-}(12.6,20)(13,20)
\psline{->}(13,20)(13,19.4)
\psline{-}(13,19.4)(13,19)
\psline{->}(13,19)(12.4,19)
\psline{-}(12.4,19)(12,19)
\psline{->}(12,19)(12,19.6)
\psline{-}(12,19.6)(12,20)

\psline{->}(15,20)(15,19.4)
\psline{-}(15,19.4)(15,19)
\psline{->}(15,19)(14.4,19)
\psline{-}(14.4,19)(14,19)
\psline{->}(14,19)(14,19.6)
\psline{-}(14,19.6)(14,20)
\psline{->}(14,20)(14.6,20)
\psline{-}(14.6,20)(15,20)

\psline[linestyle = dotted]{->}(13.5, 18.5)(13.5, 17.5)

\psline{->}(12,17)(12.6,17)
\psline{-}(12.6,17)(13,17)
\psline{->}(13,17)(13,16.4)
\psline{-}(13,16.4)(13,16)
\psline{->}(13,16)(12.4,16)
\psline{-}(12.4,16)(12,16)
\psline{->}(12,16)(12,16.6)
\psline{-}(12,16.6)(12,17)

\psline{->}(15,17)(15,16.4)
\psline{-}(15,16.4)(15,16)
\psline{->}(15,16)(14.4,16)
\psline{-}(14.4,16)(14,16)
\psline{->}(14,16)(14,16.6)
\psline{-}(14,16.6)(14,17)
\psline{->}(14,17)(14.6,17)
\psline{-}(14.6,17)(15,17)

\psline[linestyle = dotted]{->}(11.5, 16.5)(9.5, 16.5)

\psline{->}(6,17)(6.6,17)
\psline{-}(6.6,17)(7,17)
\psline{->}(7,17)(7,16.4)
\psline{-}(7,16.4)(7,16)
\psline{->}(7,16)(6.4,16)
\psline{-}(6.4,16)(6,16)
\psline{->}(6,16)(6,16.6)
\psline{-}(6,16.6)(6,17)

\psline{->}(6,15.95)(6.6,15.95)
\psline{-}(6.6,15.95)(7,15.95)
\psline{->}(7,15.95)(7,15.35)
\psline{-}(7,15.35)(7,14.95)
\psline{->}(7,14.95)(6.4,14.95)
\psline{-}(6.4,14.95)(6,14.95)
\psline{->}(6,14.95)(6,15.55)
\psline{-}(6,15.55)(6,15.95)

\psline{->}(9,17)(9,16.4)
\psline{-}(9,16.4)(9,16)
\psline{->}(9,16)(8.4,16)
\psline{-}(8.4,16)(8,16)
\psline{->}(8,16)(8,16.6)
\psline{-}(8,16.6)(8,17)
\psline{->}(8,17)(8.6,17)
\psline{-}(8.6,17)(9,17)

\psline[linestyle = dotted]{->}(5.5, 16.5)(3.5, 16.5)

\psline{->}(0,17)(.6,17)
\psline{-}(.6,17)(1,17)
\psline{->}(1,17)(1,15.9)
\psline{-}(1,15.9)(1,15)
\psline{->}(1,15)(.4,15)
\psline{-}(.4,15)(0,15)
\psline{->}(0,15)(0,16.1)
\psline{-}(0,16.1)(0,17)

\psline{->}(3,17)(3,16.4)
\psline{-}(3,16.4)(3,16)
\psline{->}(3,16)(2.4,16)
\psline{-}(2.4,16)(2,16)
\psline{->}(2,16)(2,16.6)
\psline{-}(2,16.6)(2,17)
\psline{->}(2,17)(2.6,17)
\psline{-}(2.6,17)(3,17)

\psline[linestyle = dotted]{->}(1.5, 14.5)(1.5, 13.5)

\psline{->}(0,13)(.6,13)
\psline{-}(.6,13)(1,13)
\psline{->}(1,13)(1,12.4)
\psline{-}(1,12.4)(1,12)
\psline{->}(1,12)(1,11.4)
\psline{-}(1,11.4)(1,11)
\psline{->}(1,11)(.4,11)
\psline{-}(.4,11)(0,11)
\psline{->}(0,11)(0,12.1)
\psline{-}(0,12.1)(0,13)

\psline{->}(3,13)(3,12.4)
\psline{-}(3,12.4)(3,12)
\psline{->}(3,12)(1.9,12)
\psline{-}(1.9,12)(1.05,12)
\psline{->}(1.05,12)(1.05,12.6)
\psline{-}(1.05,12.6)(1.05,13)
\psline{->}(1.05,13)(2.1,13)
\psline{-}(2.1,13)(3,13)

\psline[linestyle = dotted]{->}(3.5, 12)(5.5, 12)

\psline{->}(6,13)(7.6,13)
\psline{-}(7.6,13)(9,13)
\psline{->}(9,13)(9,12.4)
\psline{-}(9,12.4)(9,12)
\psline{->}(9,12)(7.9,12)
\psline{-}(7.9,12)(7,12)
\psline{->}(7,12)(7,11.4)
\psline{-}(7,11.4)(7,11)
\psline{->}(7,11)(6.4,11)
\psline{-}(6.4,11)(6,11)
\psline{->}(6,11)(6,12.1)
\psline{-}(6,12.1)(6,13)

\psline[linestyle = dotted]{->}(9.5, 12)(11.5, 12)

\psline{->}(12,13)(13.1,13)
\psline{-}(13.1,13)(14,13)
\psline{->}(14,13)(14,12.4)
\psline{-}(14,12.4)(14,12)
\psline{->}(14,12)(13.4,12)
\psline{-}(13.4,12)(13,12)
\psline{->}(13,12)(13,11.4)
\psline{-}(13,11.4)(13,11)
\psline{->}(13,11)(12.4,11)
\psline{-}(12.4,11)(12,11)
\psline{->}(12,11)(12,12.1)
\psline{-}(12,12.1)(12,13)

\psline[linestyle = dotted]{->}(13, 10.5)(13, 9.5)

\psline{->}(12,9)(12.6,9)
\psline{-}(12.6,9)(13,9)
\psline{->}(13,9)(13,8.4)
\psline{-}(13,8.4)(13,8)
\psline{->}(13,8)(13,7.4)
\psline{-}(13,7.4)(13,7)
\psline{->}(13,7)(12.4,7)
\psline{-}(12.4,7)(12,7)
\psline{->}(12,7)(12,8.1)
\psline{-}(12,8.1)(12,9)

\psline[linestyle = dotted]{->}(11.5, 8)(9.5, 8)

\psline{->}(6,9)(6.6,9)
\psline{-}(6.6,9)(7,9)
\psline{->}(7,9)(7,8.4)
\psline{-}(7,8.4)(7,8)
\psline{->}(7,8)(6.4,8)
\psline{-}(6.4,8)(6,8)
\psline{->}(6,8)(6,8.6)
\psline{-}(6,8.6)(6,9)

\psline[linestyle = dotted]{->}(5.5, 8)(3.5, 8)

\end{pspicture}
\caption{A vanishing trajectory of a loop.}
\label{trajfig}
\end{figure}

Define the weight of the transition from a loop sequence $s$ to another loop sequence $s'$ at inverse coupling strength $\beta$ as
\[
w_\beta(s,s') := 
\begin{cases}
\ell(s)/|s| &\text{ if $s'=s$,}\\
-1/|s| &\text{ if $s'\in  \fst^+(s)\cup \fs^+(s)$,}\\
1/|s| &\text{ if  $s'\in \fst^-(s)\cup \fs^-(s)$,}\\
-\beta/(2|s|) &\text{ if $s'\in  \fd^+(s)\cup \fex^+(s)$,}\\
\beta/(2|s|) &\text{ if  $s'\in \fd^-(s)\cup \fex^-(s)$.}
\end{cases}
\]
The weight of a finite trajectory is defined to be the product of the transition weights of consecutive loops in the trajectory. That is, if $X=(s_0, s_1, \ldots, s_n)$, then
\[
w_{\beta}(X):=w_{\beta}(s_0, s_1)w_{\beta}(s_1, s_2)\cdots w_{\beta}(s_{n-1}, s_n)\, .
\]
Since the weight of the transition from a loop sequence to another one can be negative, the weight of a trajectory also can be negative. For example, the weight of trajectory in Figure \ref{trajfig} is $-\beta^{7}/226492416000$.

By picturing the loop sequences within a trajectory in the $d$-dimensional lattice on top of each other one can visualize the trajectory as a $d+1$-dimensional surface. For instance, the surface in Figure \ref{surffig} in traced out by the trajectory from Figure \ref{trajfig}. As time passes the splitting and merging operations give rise to handles in such an illustration of a trajectory. We define the genus of a trajectory to be the total number of mergers and inactions in the trajectory. Note that this definition is not same as traditional definition of a genus. For example, the genus of trajectory in Figure \ref{trajfig} is two but the topological genus of a surface in Figure \ref{surffig} is one. It is easy to see that these two definitions coincide only if the trajectory does not contain any inaction.

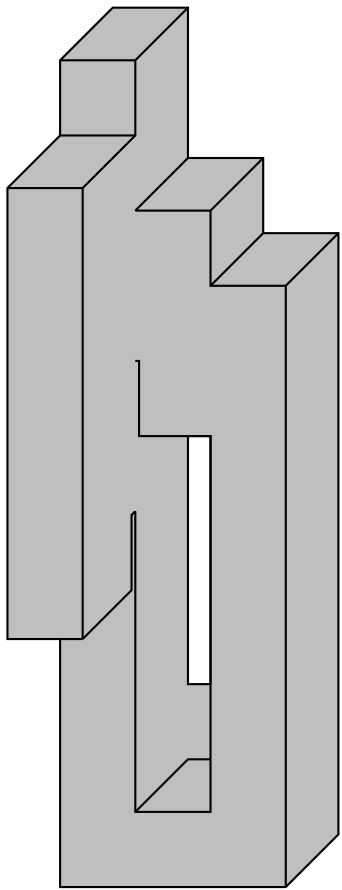
\begin{figure}[t]
\begin{pspicture}(1,.5)(6,13.2)
\psset{xunit=1cm,yunit=1cm}
\pspolygon[fillstyle=solid, fillcolor=lightgray](2,1)(5,1)(5.7,1.7)(5.7,9.7)(4.7,9.7)(4.7,10.7)(3.7,10.7)(3.7,12.7)(2.7,12.7)(2,12)(2,11) (1.3,10.3)(1.3,4.3)(2,4.3)(2,1)
\pspolygon[fillstyle=solid, fillcolor=white](3.7,3.7)(3.7,7)(4,7)(4,3.7)
\psline{-}(5,1)(5,9)(4,9)(4,10)(3,10)
\psline{-}(3.7,12.7)(3,12)(3,11)(2.3,10.3)(2.3,4.3)
\psline{-}(3,6)(3,2)(4,2)(4,7)(3.05,7)(3.05,8)(3,8)
\psline{-}(4,2.7)(3.7,2.7)(3,2)
\psline{-}(2.3, 4.3)(2.95,4.95)(2.95,5.95)(3,6)
\psline{-}(2,4.3)(2.3, 4.3)
\psline{-}(5,9)(5.7,9.7)
\psline{-}(4.7,9.7)(4,9)
\psline{-}(4,10)(4.7,10.7)
\psline{-}(3,10)(3.7,10.7)
\psline{-}(3,11)(2,11)
\psline{-}(2.3,10.3)(1.3,10.3)
\psline{-}(3,12)(2,12)
\end{pspicture}
\caption{The surface traced out by the trajectory from Figure \ref{trajfig}.}
\label{surffig}
\end{figure}

\section{Gauge theory on a lattice}\label{gauge}
We will assume that the reader is familiar with the notations introduced in Section \ref{string}. Let $G$ be a compact Lie group, and let $\rho$ be a matrix representation of $G$. Let $\beta$ be a real number and $d\geq 2$ be an integer. As in Section \ref{string}, let $\zz^d$ be the $d$-dimensional integer lattice. Let $\Lambda$ be a finite subset of $\zz^d$ and let $E^+_{\Lambda}$ be the set of all positively oriented edges with both starting and ending points lying in $\Lambda$. For any $e \in E^+_{\Lambda}$, assign a Haar distributed random element $Q_e$ of $G$, and let $Q_{e^{-1}}=Q_e^{-1} = Q_e^*$, where $Q_e^*$ is the adjoint of $Q_e$. Let $Q_{\Lambda}$ denote the set of all such assignments. For any plaquette $p=e_1e_2e_3e_4$, let $Q_p=Q_{e_1}Q_{e_2}Q_{e_3}Q_{e_4}$. Let $\mu_{\Lambda, N, \beta}$ be a probability measure on $Q_{\Lambda}$ such that at any $Q \in Q_{\Lambda}$
\[
d\mu_{\Lambda, N, \beta}(Q)=Z^{-1}_{\Lambda, N, \beta}\exp\biggl(N\beta\sum_{p\in \cp^+_{\Lambda}}\Re \tr \rho(Q_p)\biggr)\prod\limits_{e\in E^+_{\Lambda}}d\sigma(Q_e)
\]
where $\cp_{\Lambda}^+$ is the set of all positively oriented plaquettes with all edges lying in $\Lambda$, $\beta$ is a real number called the inverse coupling strength, $\sigma$ is the Haar measure on $G$, and $Z_{\Lambda, N, \beta}$ is the normalizing constant. The measure $\mu_{\Lambda, N, \beta}$ defines lattice gauge theory for the gauge group $G$ on the set $\Lambda$. Note that $\mu_{\Lambda, N, \beta}$ also depends on the representation $\rho$. In this paper, unless otherwise specified, we will assume that $G$ is $SU(N)$, the group of $N\times N$ unitary matrices, and $\rho$ is the identity representation.

For any real valued function defined on the set of configurations $Q_{\Lambda}$, let
\begin{align}
\smallavg{f}_{\Lambda, N, \beta}:=\int f(Q)d\mu_{\Lambda, N, \beta}(Q) \, . \label{avg}
\end{align}
We will omit subscripts and simply write $\smallavg{f}$ whenever $\Lambda$, $N$ and $\beta$ are clear from the context.

The main object of interest in lattice gauge theories are the Wilson loop expectations. For any loop $l=e_1e_2\cdots e_n$, with all edges lying in $\Lambda$, the Wilson loop variable $W_l$ is defined as
\[
W_l=\tr(Q_{e_1}Q_{e_2}\cdots Q_{e_n})\, ,
\]
and its expectation $\smallavg{W_l}$ is defined via \eqref{avg}. Understanding Wilson loop expectations is one of the main goals of lattice gauge theory. The $1/N$ expansion is an asymptotic series expansion for Wilson loop expectations and partition functions of large $N$ lattice gauge theories. The main result of this paper, presented in the next section, gives the $1/N$ expansion for Wilson loop expectations in $SU(N)$ lattice gauge theory, where the coefficients are expressed in terms of the string theory defined in Section~\ref{string}. 



\section{Results}\label{main}
For any non-null loop sequence $s$ with minimal representation $(l_1, \ldots, l_n)$ let
\begin{align*}
\phi_{N, \Lambda, \beta}(s)=\frac{\smallavg{W_{l_1}\cdots W_{l_n}}}{N^n} 
\end{align*}
We will simply write $\phi_N(s)$ whenever $\Lambda$ and $\beta$ are clear from the context. The following theorem is the main result of this paper. 
\begin{thm}\label{maintheorem}
There exists sequence of positive real numbers $\{\beta_0(d, k)\}_{k\geq 0}$, depending only on dimension $d$, such that for all $|\beta|\leq \beta_0(d,k)$ the following statements are true.
\begin{itemize}
\item[\text{(i)}] For any loop sequence $s$ the sum
\begin{align}
f_{2k}(s):=\sum_{X\in \mx_k(s)}w_\beta(X)  \label{f2kdef}
\end{align}
is absolutely convergent.
\item[\text{(ii)}] For any increasing $\Lambda_1\subseteq \Lambda_2\subseteq \ldots \subseteq \zz^d$ such that $\cup_{j=1}^{\infty}\Lambda_j=\zz^d$ and for any loop sequence~$s$,
\begin{align}
\lim_{N\to \infty}N^{2k}\biggr(\phi_{N, \Lambda_N, \beta}(s)-f_0(s)-\frac{1}{N^2}f_2(s)-\frac{1}{N^4}f_4(s)-\cdots-\frac{1}{N^{2k}}f_{2k}(s)\biggr)=0 \, . \label{maineq}
\end{align}
\item[\text{(iii)}] $|f_{2k}(s)|\leq (2^{3k+12}d)^{|s|}$ for any loop sequence $s$.
\end{itemize}
\end{thm}
The proof of the above theorem, like proofs of analogous results in \cite{chatterjee15} and \cite{chatterjeejafarov16}, is based on a rigorous formulation and proof of the Makeenko--Migdal equations for $SU(N)$ lattice gauge theory. The Makeenko--Migdal equations for lattice gauge theories were originally stated in \cite{makeenkomigdal79}, although that formulation was based on a certain unproved ``factorization'' property. The first rigorous version  was established for two-dimensional gauge theories in \cite{levy11}, with alternative proofs and extensions in \cite{driveretal2, driveretal}. In dimensions higher than two, the first proof was given in \cite{chatterjee15}. These equations belong to a general class of equations arising in random matrix theory,  known as Schwinger--Dyson equations. Schwinger--Dyson equations for random matrix models have been studied deeply by Alice Guionnet and coauthors over many years. In particular, the $1/N$ expansion for a large class of matrix models was established by Guionnet and Novak~\cite{guionnetnovak14}. The proof techniques in~\cite{chatterjee15} and~\cite{chatterjeejafarov16}, as well as in this paper, borrow several ideas from the paper of Collins, Guionnet and Maurel-Segala~\cite{collinsetal09}. Basu and Ganguly~\cite{bg16} have recently developed combinatorial techniques for analyzing the lattice string theories of~\cite{chatterjee15, chatterjeejafarov16}, which may be applicable to the lattice string theory developed in this paper as well. For further references, see~\cite{chatterjee15, chatterjee16, chatterjeejafarov16}.

The following theorem gives the Makeenko--Migdal equation for $SU(N)$ lattice gauge theory. It is basically a recursive equation for the Wilson loop expectation $\phi_N(s)$.  The main difference between the unrigorous equations derived in the physics literature, and the equation displayed below, is that the equation given here is defined on loop sequences instead of loops. Recall the definitions of $\fst^-(s)$, $\fst^+(s)$, $\fs^-(s)$, $\fs^+(s)$, $\fd^-(s)$, $\fd^+(s)$, $\fex^-(s)$, $\fex^+(s)$, $|s|$ and $\ell(s)$ from Section~\ref{string}.
\begin{thm}\label{symfinmaster}
For any non-null loop sequence $s$,
\begin{align}
\left(|s|-\frac{\ell(s)}{N^2}\right)&\phi_N(s)=\frac{1}{N^2}\sum_{s'\in \fst^-(s)}\phi_N(s')-\frac{1}{N^2}\sum_{s'\in \fst^+(s)}\phi_N(s')+\sum_{s'\in \fs^-(s)}\phi_N(s')-\sum_{s'\in \fs^+(s)}\phi_N(s')\notag\\
&+\frac{\beta}{2}\sum_{s'\in \fd^-(s)}\phi_N(s')-\frac{\beta}{2}\sum_{s'\in \fd^+(s)}\phi_N(s')+\frac{\beta}{2}\sum_{s'\in \fex^-(s)}\phi_N(s')-\frac{\beta}{2}\sum_{s'\in \fex^+(s)}\phi_N(s')\, . \label{symfinmastereq}
\end{align}
\end{thm}
The next result proves the so called ``factorization property'' of $SU(N)$ Wilson loop variables. It states that as $N\to \infty$ the Wilson loop variables become uncorrelated. Although this fact was not mathematically proven before, it was widely used in many theoretical computations in physics literature. The factorization property and area law upper bound for $SO(N)$ Wilson loop variables was proven in \cite{chatterjee15}. 
\begin{cor}\label{factor}
Let $\beta_{0,d}(0,d)$ and $\{\Lambda_j\}_{j\geq 1}$ be as in Theorem \ref{maintheorem}. Suppose that $|\beta|\le \beta_0(0,d)$. Then for any loop sequence $(l_1, l_2,\ldots, l_n)$
\begin{align}
\lim_{N\to \infty} \frac{\smallavg{W_{l_1}W_{l_2}\cdots W_{l_n}}_{\Lambda_N, N, \beta}}{N^n}=\prod_{i=1}^{n}\lim_{N\to \infty} \frac{\smallavg{W_{l_i}}_{\Lambda_N, N, \beta}}{N} \, .\label{factoreq}
\end{align}
\end{cor}
The next corollary shows that in $N\to\infty$ regime the Wilson loop expectations of $SO(N)$ and $SU(N)$ gauge theories coincide.
\begin{cor}\label{corresp}
Let $SO(N)$ be the group of $N\times N$ orthogonal matrices,  and let $\smallavg{\cdot }'_{\Lambda, N, \beta}$ denote the expectation with respect to $SO(N)$ gauge theory on $\Lambda$. There exists $\beta'(d)>0$, depending only on dimension $d$, such that for any $|\beta|\leq \beta'(d)$, and loop sequence $(l_1, l_2, \ldots, l_n)$
\begin{align}
\lim_{N\to \infty} \frac{\smallavg{W_{l_1}W_{l_2}\cdots W_{l_n}}_{\Lambda_N, N, 2\beta}}{N^n}=\lim_{N\to \infty} \frac{\smallavg{W_{l_1}W_{l_2}\cdots W_{l_n}}'_{\Lambda_N, N, \beta}}{N^n} \, .\label{correspondence}
\end{align}
\end{cor}
Note that Corollary \ref{factor} can be obtained from Corollary \ref{corresp} and the factorization property of $SO(N)$ Wilson variables. We state Corollary \ref{factor} separately because it will be used to prove Corollary \ref{corresp}.
\section{Preliminary lemmas}\label{prelim}
Recall the definitions of the length, the size and the index of a loop sequence from Section \ref{string}. We first present how these attributes of a loop sequence change under the loop operations defined in Section \ref{string}.
\begin{lmm}\label{expand}
If $s$ is a non-null loop sequence and $s'$ is an expansion of $s$, then $|s'|=|s|+4$ and $\iota(s')=\iota(s)+3$.
\end{lmm}
\begin{proof}
The expansion operation adds an extra plaquette near one of the loop components of $s$. Since each plaquette has four edges, $|s'|=|s|+4$. Since $\#s'=\#s+1$, we also get 
$$\iota(s')=|s'|-\#s'=(|s|+4)-(\#s+1)=\iota(s)+3\, .$$
\end{proof}

The following lemmas about the loop operations splitting, merging and deformation are taken from \cite{chatterjee15} and \cite{chatterjeejafarov16}.
\begin{lmm}[\cite{chatterjeejafarov16}]\label{merger}
If  $s$ is a non-null loop sequence and $s'$ is a merger of $s$, then $|s'|\leq |s|$ and $\iota(s')\leq \iota(s)+1$.
\end{lmm}
\begin{lmm}[\cite{chatterjeejafarov16}]\label{deform}
If $s$ is a non-null loop sequence and $s'$ is a deformation of $s$, then $|s'|\leq |s|+4$ and $\iota(s')\leq \iota(s)+4$.
\end{lmm}
\begin{lmm}[\cite{chatterjee15}]\label{split2}
Let $l$ be a non-null loop and suppose that $x$ and $y$ are two distinct locations in $l$ such that $l$ admits a negative splitting at $x$ and $y$. Let $l_1 := \times^1_{x,y} l$ and $l_2 := \times^2_{x,y} l$. Then $l_1$ and $l_2$ are non-null loops, $|l_1|\le |l| - |y-x|-1$, and $|l_2|\le |y-x|-1$. 
\end{lmm}
\begin{lmm}[\cite{chatterjee15}]\label{split1}
Let $l$ be a non-null loop and suppose that $x$ and $y$ are two distinct locations in $l$ such that $l$ admits a positive splitting at $x$ and $y$. Let $l_1 := \times^1_{x,y} l$ and $l_2 := \times^2_{x,y} l$. Then $l_1$ and $l_2$ are non-null loops, $|l_1|\le |l| - |y-x|$, and $|l_2|\le |y-x|$. 
\end{lmm}
\begin{lmm}[\cite{chatterjee15}]\label{iota2}
If $s'$ is obtained from $s$ by a splitting operation, then $s'\neq \emptyset$ and $\iota(s')<\iota(s)$.
\end{lmm}

\begin{lmm}\label{fintraj}
For any non-negative integers $a$, $b$, $c$, $d$, $i$, $k$ and loop sequence $s$ the sets $\mx_{a, b, c, d}(s)$, $\mx_{a, b, k}(s)$ and $\mx_{i, k}(s)$ are finite.
\end{lmm}
\begin{proof}
By Lemma \ref{merger} the merger operation increases the index $\iota(s)$ at most by one, by Lemma \ref{deform} the deformation operation increases $\iota(s)$ at most by four and by Lemma \ref{expand} the expansion operation increases $\iota(s)$ exactly by three. On the other hand, by Lemma \ref{iota2} the splitting operation reduces index at least by one. Therefore, the number of splitting operations in any trajectory in  $\mx_{a, b, c, d}(s)$ is no more than $\iota(s)+4a+3b+c$. Thus there are only finitely many trajectories in this set. Since the equation $c+d=k$ has only finitely many non-negative solutions the set $\mx_{a, b, k}(s)$ is also finite. Similarly, $\mx_{i, k}(s)$ is finite too.
\end{proof}
Recall the definition of Catalan number:
\[
C_{n}=\frac{1}{n+1}\binom{2n}{n}\, .
\]
One can easily check that 
\begin{align}
C_{n+1}&\leq 4C_n \label{catalan0}
\end{align}
for any $n\geq 0$. Another well known property of the Catalan numbers, that we will use many times, is the following identity.
\begin{lmm}\label{catalan}
For any $n\geq 0$
\[
C_n=\sum_{k=0}^{n-1}C_{n-1-k}C_k
\]
where $C_n$ is the $n^{\textup{th}}$ Catalan number.
\end{lmm}
The following lemma about the Catalan numbers is taken from \cite{chatterjeejafarov16}.

\begin{lmm}[\cite{chatterjeejafarov16}]\label{catalan1}
Let $C_k$ be the $k^{\mathrm{th}}$ Catalan number. If $n, m$ are positive integers, then
$$C_{n+m-1}\leq (n+m)^2C_{n-1}C_{m-1}\,.$$
\end{lmm}

\section{Stein's exchangeable pair for $SU(N)$}\label{stein}
As in \cite{chatterjee15} and \cite{chatterjeejafarov16}, we will now derive an integration-by-parts identity for $SU(N)$ using an approach based on Stein's method of exchangeable pairs~\cite{chatterjeemeckes, meckes, stein72, stein86, stein95}. A pair of random variables $(X, Y)$ is said to be exchangeable if $(X,Y)$ and $(Y, X)$ have the same joint distribution. We start by constructing an exchangeable pair of $SU(N)$ random matrices. 

Fix $N$ and let $(U, V)$ be a uniformly chosen pair of numbers from the set $\{(u, v): 1\leq u\neq v\leq N\}$. Let $\eta$ and $\xi$ be independent random variables that take values $1$ and $-1$ with equal probability, and let $0\leq \theta , \phi \leq 2\pi$ be two non-random real numbers. For any $\ep \in (-1, 1)$ define a random $N\times N$ matrix $R_{\ep, \theta, \phi}=(r_{uv})_{1\leq u, v \leq N}$ as
\begin{align*}
r_{UU}&=\sqrt{1-\ep^2}+i\ep \eta\cos \theta\, , \qquad r_{UV}=\ep e^{i\phi}  \xi \sin\theta\, ,\\
r_{VV}&=\sqrt{1-\ep^2}-i\ep \eta\cos \theta \, , \qquad r_{VU}=-\ep e^{-i\phi}  \xi\sin \theta\, ,
\end{align*}
and for all $k\neq U, V$ and $1\leq k'\leq N$
\[
r_{Uk} = r_{Vk} = 0\, , \ \text{ and } \ 
r_{kk'} = 
\begin{cases}
1 &\text{ if } k=k'\, ,\\
0 &\text{ if } k\ne k'\, .
\end{cases}
\]
Of course, here $i$ stands for $\sqrt{-1}$. It is easy to check that $R_{\ep, \theta, \phi}$ is an $SU(N)$ matrix. For any $Q\in SU(N)$ let $Q_{\ep, \theta, \phi}:=R_{\ep, \theta, \phi}Q$.
\begin{lmm}
If $Q$ is a Haar distributed random $SU(N)$ matrix, then $(Q, Q_{\ep, \theta, \phi})$ is an exchangeable pair for all possible values of $\ep$, $\theta$ and $\phi$.
\end{lmm}
\begin{proof}
Since $Q$ is Haar distributed, conditional on $\eta$ and $\xi$ the matrices $Q$ and $R^*_{\ep, \theta, \phi}Q$ have the same distribution. So, by replacing $Q$ with $R^*_{\ep, \theta, \phi}Q$ we see that, conditional on $\eta$ and $\xi$, the pairs $(Q,R_{\ep, \theta, \phi}Q)$ and $(R^*_{\ep, \theta, \phi}Q, Q)$ have the same distribution. So their unconditional distributions are also same. The matrices $R^*_{\ep, \theta, \phi}$ and $R_{\ep, \theta, \phi}$ also have the same distribution because both $\xi$ and $\eta$ are symmetrically distributed around zero. Thus, conditional on $Q$ the distribution of the pairs $(R^*_{\ep, \theta, \phi}Q, Q)$ and $(R_{\ep, \theta, \phi}Q, Q)$ are same. Therefore their unconditional distributions are also same. Hence, we showed that $(Q,R_{\ep, \theta, \phi}Q)$ and $(R_{\ep, \theta, \phi}Q, Q)$ have the same distribution.
\end{proof}

The following simple lemma is taken from \cite{chatterjee15}, where it was proved for $SO(N)$. Versions of this lemma in various contexts have appeared in many places in the literature on Stein's method, beginning with Stein's original paper~\cite{stein72}. It is essentially an abstract integration-by-parts identity. The proof, which we will omit, involves a simple application of exchangeability of the pair $(Q, Q_{\ep, \theta, \phi})$.

\begin{lmm}\label{exchlmm}
For any Borel measurable $f,g:SU(N)\ra \cc$,
\[
\ee((f(Q_{\ep, \theta, \phi}) - f(Q))g(Q)) = -\frac{1}{2} \ee((f(Q_{\ep, \theta, \phi})-f(Q))(g(Q_{\ep,\theta,\phi})-g(Q)))\, .
\]
\end{lmm}

\section{Schwinger--Dyson equations for $SU(N)$}\label{sd}
The following theorem is the main result of this section. 
\begin{thm}\label{sdthm}
Let $\cc^{N\times N}$ be the space of $N\times N$ complex matrices with the Euclidean topology. Let $f$ and $g$ be $C^2$ functions defined on an open subset of $\cc^{N\times N}$ that contains $SU(N)$. Let $Q= (q_{uk}+it_{uk})_{1\le u, k\le N}$ be a Haar-distributed random element of $SU(N)$, and let $f$ and $g$ be shorthand notations for the random variables $f(Q)$ and $g(Q)$. 
Then
\begin{align}
&2(N^2-1)\ee\biggl[\sum_{u,k}\biggl(q_{uk} \fpar{f}{q_{uk}}+t_{uk} \fpar{f}{t_{uk}}\biggr) g\biggr]=N\ee\biggl[ \sum_{u,k} \left(\spar{f}{q_{uk}}g+\spar{f}{t_{uk}}g+\fpar{f}{q_{uk}}\fpar{g}{q_{uk}}+\fpar{f}{t_{uk}}\fpar{g}{t_{uk}}\right)\biggr]\notag\\
&\qquad-N\ee\biggl[\sum_{u,v,k,k'} \biggl((q_{vk}q_{uk'}-t_{vk}t_{uk'}) \left(\mpar{f}{q_{uk}}{q_{vk'}}-\mpar{f}{t_{uk}}{t_{vk'}}\right) +2(q_{vk}t_{uk'}+t_{vk}q_{uk'}) \mpar{f}{q_{uk}}{t_{vk'}} \biggr)g\biggr]\notag\\
&\qquad-2\ee\biggl[\sum_{u,v,k,k'} \biggl( t_{uk}t_{vk'} \mpar{f}{q_{uk}}{q_{vk'}}+q_{uk}q_{vk'} \mpar{f}{t_{uk}}{t_{vk'}}-2t_{uk}q_{vk'} \mpar{f}{q_{uk}}{t_{vk'}} \biggr)g\biggr]\notag\\
&\qquad-N\ee\biggl[\sum_{u,v,k,k'} \biggl((q_{vk}q_{uk'}-t_{vk}t_{uk'})\left(\fpar{f}{q_{uk}}\fpar{g}{q_{vk'}}- \fpar{f}{t_{uk}}\fpar{g}{t_{vk'}}\right)\notag\\
&\qquad \qquad \qquad \qquad \qquad \qquad \qquad \qquad \qquad \qquad  +(q_{vk}t_{uk'}+t_{vk}q_{uk'})\left(\fpar{f}{q_{uk}}\fpar{g}{t_{vk'}}+ \fpar{f}{t_{uk}}\fpar{g}{q_{vk'}} \right)\biggr)\biggr]\notag\\
&\qquad -2\ee\biggl[\sum_{u,v,k,k'} \biggl(t_{uk}t_{vk'} \fpar{f}{q_{uk}}\fpar{g}{q_{vk'}}+q_{uk}q_{vk'} \fpar{f}{t_{uk}}\fpar{g}{t_{vk'}}-t_{uk}q_{vk'} \fpar{f}{q_{uk}}\fpar{g}{t_{vk'}}-q_{uk}t_{vk'} \fpar{f}{t_{uk}}\fpar{g}{q_{vk'}} \biggr)\biggr]\notag\\ \label{schdys}
\end{align}
where all indices run from $1$ to $N$ and $\ee$ denotes integration with respect to the Haar measure. 
\end{thm}
To prove Theorem \ref{sdthm} we will first find three different Schwinger--Dyson type equations for $SU(N)$ matrices and then we will combine them to get equation \eqref{schdys}. The proofs of the next three lemmas are obtained by applying Lemma \ref{exchlmm} to the exchangeable pair $(Q, Q_{\theta, \phi, \epsilon})$ from Section~\ref{stein} with some special choices of $\theta$ and $\phi$. Throughout this section $O(\epsilon^3)$ denotes any random or nan-random quantity that can be bounded by some constant multiple of $\epsilon^3$.
\begin{lmm}\label{sdthm1}
Under the assumptions of Theorem \ref{sdthm}
\begin{align}
&\ee\biggl[\sum_{k}\biggl(q_{Uk} \fpar{f}{q_{Uk}}+t_{Uk} \fpar{f}{t_{Uk}}\biggr) g\biggr]\notag\\
&= \ee\biggl[ \sum_{k,k'} \biggl(q_{Vk}q_{Vk'} \mpar{f}{q_{Uk}}{q_{Uk'}} + t_{Vk}t_{Vk'} \mpar{f}{t_{Uk}}{t_{Uk'}}+2q_{Vk}t_{Vk'} \mpar{f}{q_{Uk}}{t_{Uk'}} \biggr)g \biggr]\notag\\
&- \ee\biggl[\sum_{k,k'} \biggl(q_{Vk}q_{Uk'} \mpar{f}{q_{Uk}}{q_{Vk'}} +t_{Vk}t_{Uk'} \mpar{f}{t_{Uk}}{t_{Vk'}} +2q_{Vk}t_{Uk'} \mpar{f}{q_{Uk}}{t_{Vk'}} \biggr)g\biggr]\notag\\
&+\ee\biggl[\sum_{k,k'} \biggl(q_{Vk}q_{Vk'} \fpar{f}{q_{Uk}}\fpar{g}{q_{Uk'}}+ t_{Vk}t_{Vk'} \fpar{f}{t_{Uk}}\fpar{g}{t_{Uk'}}+q_{Vk}t_{Vk'} \fpar{f}{q_{Uk}}\fpar{g}{t_{Uk'}}+ t_{Vk}q_{Vk'} \fpar{f}{t_{Uk}}\fpar{g}{q_{Uk'}}\biggr)\biggr]\notag\\
& -\ee\biggl[\sum_{k,k'} \biggl(q_{Vk}q_{Uk'} \fpar{f}{q_{Uk}}\fpar{g}{q_{Vk'}}+t_{Vk}t_{Uk'} \fpar{f}{t_{Uk}}\fpar{g}{t_{Vk'}}+q_{Vk}t_{Uk'} \fpar{f}{q_{Uk}}\fpar{g}{t_{Vk'}}+t_{Vk}q_{Uk'} \fpar{f}{t_{Uk}}\fpar{g}{q_{Vk'}} \biggr)\biggr] \label{schdys1}
\end{align}
where all indices run from $1$ to $N$ and $\ee$ denotes integration with respect to the Haar measure. 
\end{lmm}
\begin{proof}
Let $Q_\ep:=Q_{\ep, \frac{\pi}{2}, 0}$ and denote the $(u,k)^{\textup{th}}$ entry of $Q_\ep$ by $q_{uk}^\ep+it_{uk}^\ep$. Then, for  any $1\le k\le N$,
\begin{align*}
q^\ep_{Uk} &= \sqrt{1-\ep^2}\, q_{Uk} +  \ep q_{Vk}\xi\, ,\\
t^\ep_{Uk} &= \sqrt{1-\ep^2}\, t_{Uk} + \ep t_{Vk} \xi\, ,\\
q^\ep_{Vk} &= -\ep q_{Uk} \xi + \sqrt{1-\ep^2}\, q_{Vk}\, ,\\
t^\ep_{Vk} &= -\ep t_{Uk} \xi + \sqrt{1-\ep^2}\, t_{Vk}\, ,\\
q^\ep_{k'k} &= q_{k'k} \ \text{ and } \ t^\ep_{k'k} = t_{k'k} \ \text{ for all } k' \neq U,V \, .
\end{align*}
For each $u$ and $k$, let $\delta_{uk} := q^\ep_{uk} - q_{uk}$ and $\gamma_{uk} := t^\ep_{uk} - t_{uk}$. Then
\begin{align*}
\delta_{Uk} &= -\frac{\ep^2}{2} q_{Uk} + \ep q_{Vk} \xi + O(\ep^3)\, ,\\
\gamma_{Uk} &= -\frac{\ep^2}{2} t_{Uk} + \ep t_{Vk} \xi + O(\ep^3)\, ,\\
\delta_{Vk} &= -\ep q_{Uk} \xi -\frac{\ep^2}{2} q_{Vk} + O(\ep^3)\, ,\\
\gamma_{Vk} &= -\ep t_{Uk} \xi-\frac{\ep^2}{2} t_{Vk} + O(\ep^3)\, , \ \text{ and }\\
\delta_{k'k} &= \gamma_{k'k} = 0 \ \text{ for all } k' \ne U, V\, .
\end{align*}
By compactness of $SU(N)$, $f$, $g$ and their derivatives are bounded on $SU(N)$, and therefore
\begin{align}
f(Q_\ep)&- f(Q)= \sum_k \biggl(\delta_{Uk} \fpar{f}{q_{Uk}} + \delta_{Vk} \fpar{f}{q_{Vk}}+\gamma_{Uk}\fpar{f}{t_{Uk}}+ \gamma_{Vk}\fpar{f}{t_{Vk}}\biggr)\notag \\
&+\frac{1}{2} \sum_{k,k'}\biggl(\delta_{Uk}\delta_{Uk'} \mpar{f}{q_{Uk}}{q_{Uk'}}+\delta_{Vk}\delta_{Vk'} \mpar{f}{q_{Vk}}{q_{Vk'}}+\gamma_{Uk}\gamma_{Uk'} \mpar{f}{t_{Uk}}{t_{Uk'}} + \gamma_{Vk}\gamma_{Vk'} \mpar{f}{t_{Vk}}{t_{Vk'}}\biggr)\nonumber\\
&+\sum_{k,k'}\biggl(\delta_{Uk}\delta_{Vk'} \mpar{f}{q_{Uk}}{q_{Vk'}}+\gamma_{Uk}\gamma_{Vk'} \mpar{f}{t_{Uk}}{t_{Vk'}}+\delta_{Uk}\gamma_{Vk'} \mpar{f}{q_{Uk}}{t_{Vk'}} +\gamma_{Uk}\delta_{Vk'} \mpar{f}{t_{Uk}}{q_{Vk'}}\biggr)\nonumber\\
&+\sum_{k,k'}\biggl(\delta_{Uk}\gamma_{Uk'} \mpar{f}{q_{Uk}}{t_{Uk'}}+\delta_{Vk}\gamma_{Vk'} \mpar{f}{q_{Vk}}{t_{Vk'}}\biggr)+ O(\ep^3)\, . \label{taylor1}
\end{align}
Note that $\xi$ has zero mean and $\xi^2=1$. So by replacing each quantity in \eqref{taylor1} with its value in terms of $q_{uk}$ and $t_{uk}$ and using symmetry between $U$ and $V$ we get

\begin{align}
\ee \biggl[ &\biggl(f(Q_\ep) - f(Q)\biggr)g(Q)\biggl] = -\ep^2\ee \biggl[ \sum_k \biggl(q_{Uk} \fpar{f}{q_{Uk}} + t_{Uk} \fpar{f}{t_{Uk}}\biggr)g\biggl]\notag \\
&+\ep^2 \ee \biggl[\sum_{k,k'}\biggl(q_{Vk}q_{Vk'} \mpar{f}{q_{Uk}}{q_{Uk'}}+t_{Vk}t_{Vk'} \mpar{f}{t_{Uk}}{t_{Uk'}}+2q_{Vk}t_{Vk'} \mpar{f}{q_{Uk}}{t_{Uk'}}\biggr)g\biggr]\nonumber\\
&-\ep^2 \ee \biggl[\sum_{k,k'}\biggl(q_{Vk}q_{Uk'} \mpar{f}{q_{Uk}}{q_{Vk'}}+t_{Vk}t_{Uk'} \mpar{f}{t_{Uk}}{t_{Vk'}}+2q_{Vk}t_{Uk'} \mpar{f}{q_{Uk}}{t_{Vk'}}\biggr]+ O(\ep^3)\, .
 \label{f1}
\end{align}
Similarly,
\begin{align}
(f(Q_\ep)-f(Q))&(g(Q_\ep)-g(Q))= \biggl[\sum_k \biggl(\delta_{Uk} \fpar{f}{q_{Uk}} + \gamma_{Uk}\fpar{f}{t_{Uk}}+ \delta_{Vk} \fpar{f}{q_{Vk}}+ \gamma_{Vk}\fpar{f}{t_{Vk}}\biggr)\biggr]\notag\\
&\qquad \cdot\biggl[\sum_{k'} \biggl(\delta_{Uk'} \fpar{g}{q_{Uk'}} + \gamma_{Uk'}\fpar{g}{t_{Uk'}}+ \delta_{Vk'} \fpar{g}{q_{Vk'}}+ \gamma_{Vk'}\fpar{g}{t_{Vk'}}\biggr)\biggr]+ O(\ep^3)\, . \label{taylor2}
\end{align}
Again replacing everything in terms of $q_{uk}$ and $t_{uk}$ in the equation \eqref{taylor2}, and using symmetry between $U$ and $V$ we obtain
\begin{align}
&\frac{1}{2}\ee\biggl[(f(Q_\ep)-f(Q))(g(Q_\ep)-g(Q))\biggr]\notag\\
&= \ep^2 \ee\biggl[\sum_{k,k'}\biggl( q_{Vk}q_{Vk'} \fpar{f}{q_{Uk}} \fpar{g}{q_{Uk'}}+t_{Vk}t_{Vk'} \fpar{f}{t_{Uk}} \fpar{g}{t_{Uk'}}+q_{Vk}t_{Vk'} \fpar{f}{q_{Uk}} \fpar{g}{t_{Uk'}}+t_{Vk}q_{Vk'} \fpar{f}{t_{Uk}} \fpar{g}{q_{Uk'}}\biggr)\biggr]\notag\\
&-\ep^2 \ee\biggl[\sum_{k,k'}\biggl( q_{Vk}q_{Uk'} \fpar{f}{q_{Uk}} \fpar{g}{q_{Vk'}}+t_{Vk}t_{Uk'} \fpar{f}{t_{Uk}} \fpar{g}{t_{Vk'}}+ q_{Vk}t_{Uk'} \fpar{f}{q_{Uk}} \fpar{g}{t_{Vk'}}+t_{Vk}q_{Uk'} \fpar{f}{t_{Uk}} \fpar{g}{q_{Vk'}}\biggr)\biggr] \notag \\
&+ O(\ep^3)\, .\label{f2}
\end{align}
The proof is completed by combining equations $\eqref{f1}$, $\eqref{f2}$ and Lemma \ref{exchlmm}, and taking $\ep$ to zero. 
\end{proof}
\begin{lmm}\label{sdthm2}
Under the assumptions of Theorem \ref{sdthm}
\begin{align}
&\ee\biggl[\sum_{k}\biggl(q_{Uk} \fpar{f}{q_{Uk}}+t_{Uk} \fpar{f}{t_{Uk}}\biggr) g\biggr]\notag\\
&= \ee\biggl[ \sum_{k,k'} \biggl(t_{Vk}t_{Vk'} \mpar{f}{q_{Uk}}{q_{Uk'}}+q_{Vk}q_{Vk'} \mpar{f}{t_{Uk}}{t_{Uk'}} -2t_{Vk}q_{Vk'} \mpar{f}{q_{Uk}}{t_{Uk'}} \biggr)g \biggr]\notag\\
&+\ee\biggl[\sum_{k,k'} \biggl( t_{Vk}t_{Uk'} \mpar{f}{q_{Uk}}{q_{Vk'}}+q_{Vk}q_{Uk'} \mpar{f}{t_{Uk}}{t_{Vk'}}-2t_{Vk}q_{Uk'} \mpar{f}{q_{Uk}}{t_{Vk'}} \biggr)g\biggr]\notag\\
&+\ee\biggl[\sum_{k,k'} \biggl(t_{Vk}t_{Vk'} \fpar{f}{q_{Uk}}\fpar{g}{q_{Uk'}}+q_{Vk}q_{Vk'} \fpar{f}{t_{Uk}}\fpar{g}{t_{Uk'}}- t_{Vk}q_{Vk'} \fpar{f}{q_{Uk}}\fpar{g}{t_{Uk'}}-q_{Vk}t_{Vk'} \fpar{f}{t_{Uk}}\fpar{g}{q_{Uk'}}\biggr)\biggr]\notag\\
&+\ee\biggl[\sum_{k,k'} \biggl(t_{Vk}t_{Uk'} \fpar{f}{q_{Uk}}\fpar{g}{q_{Vk'}}+q_{Vk}q_{Uk'} \fpar{f}{t_{Uk}}\fpar{g}{t_{Vk'}}-t_{Vk}q_{Uk'} \fpar{f}{q_{Uk}}\fpar{g}{t_{Vk'}}-q_{Vk}t_{Uk'} \fpar{f}{t_{Uk}}\fpar{g}{q_{Vk'}} \biggr)\biggr] \label{schdys2}
\end{align}
where all indices run from $1$ to $N$ and $\ee$ denotes integration with respect to the Haar measure.
\end{lmm}
\begin{proof}
Let $Q_\ep:=Q_{\epsilon, \frac{\pi}{2}, \frac{\pi}{2}}$ and let $q_{uk}^\ep+it_{uk}^\ep$ denote the $(u,k)^{\textup{th}}$ entry of $Q_\ep$. Then, for  any $1\le k\le N$,
\begin{align*}
q^\ep_{Uk} &= \sqrt{1-\ep^2}\, q_{Uk} - \ep t_{Vk} \xi\, ,\\
t^\ep_{Uk} &= \sqrt{1-\ep^2}\, t_{Uk} + \ep q_{Vk} \xi\, ,\\
q^\ep_{Vk} &= -\ep t_{Uk} \xi + \sqrt{1-\ep^2}\, q_{Vk}\, ,\\
t^\ep_{Vk} &= \ep q_{Uk} \xi + \sqrt{1-\ep^2}\, t_{Vk}\, , \ \text{ and }\\
q^\ep_{k'k} &= q_{k'k} \ \text{ and } \ t^\ep_{k'k} = t_{k'k} \ \text{ for all } k' \neq U,V \, .
\end{align*}
Again for each $u$ and $k$, let $\delta_{uk} := q^\ep_{uk} - q_{uk}$ and $\gamma_{uk} := t^\ep_{uk} - t_{uk}$. Then for all $1\leq k\leq N$,
\begin{align*}
\delta_{Uk} &= -\frac{\ep^2}{2} q_{Uk} - \ep t_{Vk} \xi + O(\ep^3)\, ,\\
\gamma_{Uk} &= -\frac{\ep^2}{2} t_{Uk} + \ep q_{Vk} \xi + O(\ep^3)\, ,\\
\delta_{Vk} &= -\ep t_{Uk} \xi -\frac{\ep^2}{2} q_{Vk} + O(\ep^3)\, ,\\
\gamma_{Vk} &= \ep q_{Uk} \xi -\frac{\ep^2}{2} t_{Vk} + O(\ep^3)\, , \ \text{ and }\\
\delta_{k'k} &= \gamma_{k'k} = 0 \ \text{ for all } k' \ne U, V\, .
\end{align*}
Thus by using equation \eqref{taylor1} and symmetry between $U$ and $V$ we get
\begin{align}
\ee \biggl[ &\biggl(f(Q_\ep) - f(Q)\biggr)g(Q)\biggl] = -\ep^2\ee \biggl[ \sum_k \biggl(q_{Uk} \fpar{f}{q_{Uk}} + t_{Uk} \fpar{f}{t_{Uk}}\biggr)g\biggl]\notag \\
&+\ep^2 \ee \biggl[\sum_{k,k'}\biggl(t_{Vk}t_{Vk'} \mpar{f}{q_{Uk}}{q_{Uk'}}+q_{Vk}q_{Vk'} \mpar{f}{t_{Uk}}{t_{Uk'}}-2t_{Vk}q_{Vk'} \mpar{f}{q_{Uk}}{t_{Uk'}}\biggr)g\biggr]\nonumber\\
&+\ep^2 \ee \biggl[\sum_{k,k'}\biggl(t_{Vk}t_{Uk'} \mpar{f}{q_{Uk}}{q_{Vk'}}+q_{Vk}q_{Uk'} \mpar{f}{t_{Uk}}{t_{Vk'}}-2t_{Vk}q_{Uk'} \mpar{f}{q_{Uk}}{t_{Vk'}}\biggr)g\biggr]+ O(\ep^3)\, .
 \label{f3}
\end{align}
Similarly, by equation \eqref{taylor2} and symmetry between $U$ and $V$ 
\begin{align}
&\frac{1}{2}\ee\biggl[(f(Q_\ep)-f(Q))(g(Q_\ep)-g(Q))\biggr]\notag\\
&= \ep^2 \ee\biggl[\sum_{k,k'}\biggl( t_{Vk}t_{Vk'} \fpar{f}{q_{Uk}} \fpar{g}{q_{Uk'}}+q_{Vk}q_{Vk'} \fpar{f}{t_{Uk}} \fpar{g}{t_{Uk'}}-t_{Vk}q_{Vk'} \fpar{f}{q_{Uk}} \fpar{g}{t_{Uk'}}-q_{Vk}t_{Vk'} \fpar{f}{t_{Uk}} \fpar{g}{q_{Uk'}}\biggr)\biggr]\notag\\
&+\ep^2 \ee\biggl[\sum_{k,k'}\biggl( t_{Vk}t_{Uk'} \fpar{f}{q_{Uk}} \fpar{g}{q_{Vk'}}+q_{Vk}q_{Uk'} \fpar{f}{t_{Uk}} \fpar{g}{t_{Vk'}}- t_{Vk}q_{Uk'} \fpar{f}{q_{Uk}} \fpar{g}{t_{Vk'}}-q_{Vk}t_{Uk'} \fpar{f}{t_{Uk}} \fpar{g}{q_{Vk'}}\biggr)\biggr]\notag \\
&+ O(\ep^3)\, .\label{f4}
\end{align}
We finish the proof by combining equations $\eqref{f3}$, $ \eqref{f4}$ and Lemma \ref{exchlmm}, and taking $\ep$ to zero. 
\end{proof}
\begin{lmm}\label{sdthm3}
Under the assumptions of Theorem \ref{sdthm}
\begin{align}
&\ee\biggl[\sum_{k}\biggl(q_{Uk} \fpar{f}{q_{Uk}}+t_{Uk} \fpar{f}{t_{Uk}}\biggr) g\biggr]\notag\\
&=\ee\biggl[\sum_{k,k'} \biggl(t_{Uk}t_{Uk'} \fpar{f}{q_{Uk}}\fpar{g}{q_{Uk'}}+q_{Uk}q_{Uk'} \fpar{f}{t_{Uk}}\fpar{g}{t_{Uk'}}- t_{Uk}q_{Uk'} \fpar{f}{q_{Uk}}\fpar{g}{t_{Uk'}}-q_{Uk}t_{Uk'} \fpar{f}{t_{Uk}}\fpar{g}{q_{Uk'}}\biggr)\biggr]\notag\\
&+\ee\biggl[ \sum_{k,k'} \biggl(t_{Uk}t_{Uk'} \mpar{f}{q_{Uk}}{q_{Uk'}}+q_{Uk}q_{Uk'} \mpar{f}{t_{Uk}}{t_{Uk'}}-2t_{Uk}q_{Uk'} \mpar{f}{q_{Uk}}{t_{Uk'}}\biggr)g \biggr]\notag\\
&-\ee\biggl[\sum_{k,k'} \biggl(t_{Uk}t_{Vk'} \fpar{f}{q_{Uk}}\fpar{g}{q_{Vk'}}+ q_{Uk}q_{Vk'} \fpar{f}{t_{Uk}}\fpar{g}{t_{Vk'}} -t_{Uk}q_{Vk'} \fpar{f}{q_{Uk}}\fpar{g}{t_{Vk'}}- q_{Uk}t_{Vk'} \fpar{f}{t_{Uk}}\fpar{g}{q_{Vk'}} \biggr)\biggr] \notag\\
&-\ee\biggl[\sum_{k,k'} \biggl( t_{Uk}t_{Vk'} \mpar{f}{q_{Uk}}{q_{Vk'}}+ q_{Uk}q_{Vk'} \mpar{f}{t_{Uk}}{t_{Vk'}}-2t_{Uk}q_{Vk'} \mpar{f}{q_{Uk}}{t_{Vk'}} \biggr)g\biggr]\label{schdys3}
\end{align}
where all indices run from $1$ to $N$ and $\ee$ denotes integration with respect to the Haar measure.
\end{lmm}
\begin{proof}
Let $Q_\ep:=Q_{\ep, 0, 0}$ and denote the $(u,k)^{\textup{th}}$ entry of $Q_\ep$ by $q_{uk}^\ep+it_{uk}^\ep$. Then, for any $1\le k\le N$,
\begin{align*}
q^\ep_{Uk} &= \sqrt{1-\ep^2}\, q_{Uk} - \eta \ep t_{Uk}\, ,\\
t^\ep_{Uk} &= \sqrt{1-\ep^2}\, t_{Uk} + \eta \ep q_{Uk}\, ,\\
q^\ep_{Vk} &= \eta \ep t_{Vk} + \sqrt{1-\ep^2}\, q_{Vk}\, ,\\
t^\ep_{Vk} &= -\eta \ep q_{Vk} + \sqrt{1-\ep^2}\, t_{Vk}\, , \ \text{ and }\\
q^\ep_{k'k} &= q_{k'k} \ \text{ and } \ t^\ep_{k'k} = t_{k'k} \ \text{ for all } k' \neq U,V \, .
\end{align*}
As before, for each $u$ and $k$, let $\delta_{uk} := q^\ep_{uk} - q_{uk}$ and $\gamma_{uk} := t^\ep_{uk} - t_{uk}$. Then, for all $1\leq k\leq N$,
\begin{align*}
\delta_{Uk} &= -\frac{\ep^2}{2} q_{Uk} - \eta \ep t_{Uk} + O(\ep^3)\, ,\\
\gamma_{Uk} &= -\frac{\ep^2}{2} t_{Uk} + \eta \ep q_{Uk} + O(\ep^3)\, ,\\
\delta_{Vk} &= \eta \ep t_{Vk} -\frac{\ep^2}{2} q_{Vk} + O(\ep^3)\, ,\\
\gamma_{Vk} &= -\eta \ep q_{Vk} -\frac{\ep^2}{2} t_{Vk} + O(\ep^3)\, , \ \text{ and }\\
\delta_{k'k} &= \gamma_{k'k} = 0 \ \text{ for all } k' \ne U,V\, .
\end{align*}
Thus by equation \eqref{taylor1} and symmetry between $U$ and $V$ we get
\begin{align}
\ee \biggl[ &\biggl(f(Q_\ep) - f(Q)\biggr)g(Q)\biggl]= -\ep^2\ee \biggl[ \sum_k \biggl(q_{Uk} \fpar{f}{q_{Uk}} + t_{Uk} \fpar{f}{t_{Uk}}\biggr)g\biggl]\notag \\
&+\ep^2 \ee \biggl[\sum_{k,k'}\biggl(t_{Uk}t_{Uk'} \mpar{f}{q_{Uk}}{q_{Uk'}}+q_{Uk}q_{Uk'} \mpar{f}{t_{Uk}}{t_{Uk'}}-2t_{Uk}q_{Uk'} \mpar{f}{q_{Uk}}{t_{Uk'}}\biggr)g\biggr]\nonumber\\
&-\ep^2 \ee \biggl[\sum_{k,k'}\biggl(t_{Uk}t_{Vk'} \mpar{f}{q_{Uk}}{q_{Vk'}}+q_{Uk}q_{Vk'} \mpar{f}{t_{Uk}}{t_{Vk'}}-2t_{Uk}q_{Vk'} \mpar{f}{q_{Uk}}{t_{Vk'}}\biggr)g\biggr]+ O(\ep^3)\, .
 \label{f5}
\end{align}
Similarly, by equation \eqref{taylor2} and symmetry between $U$ and $V$ 
\begin{align}
&\frac{1}{2}\ee\biggl[(f(Q_\ep)-f(Q))(g(Q_\ep)-g(Q))\biggr]\notag\\
&= \ep^2 \ee\biggl[\sum_{k,k'}\biggl(t_{Uk}t_{Uk'} \fpar{f}{q_{Uk}} \fpar{g}{q_{Uk'}}+q_{Uk}q_{Uk'} \fpar{f}{t_{Uk}} \fpar{g}{t_{Uk'}}-t_{Uk}q_{Uk'} \fpar{f}{q_{Uk}} \fpar{g}{t_{Uk'}}-q_{Uk}t_{Uk'} \fpar{f}{t_{Uk}} \fpar{g}{q_{Uk'}}\biggr)\biggr]\notag\\
&-\ep^2 \ee\biggl[\sum_{k,k'}\biggl(t_{Uk}t_{Vk'} \fpar{f}{q_{Uk}} \fpar{g}{q_{Vk'}}+q_{Uk}q_{Vk'} \fpar{f}{t_{Uk}} \fpar{g}{t_{Vk'}}-t_{Uk}q_{Vk'} \fpar{f}{q_{Uk}} \fpar{g}{t_{Vk'}}-q_{Uk}t_{Vk'} \fpar{f}{t_{Uk}} \fpar{g}{q_{Vk'}}\biggr)\biggr]\notag \\
& + O(\ep^3)\, .\label{f6}
\end{align}
The proof is completed by combining equations $\eqref{f5}$, $ \eqref{f6}$ and Lemma \ref{exchlmm}, and letting $\ep \to 0$. 
\end{proof}
Now we are ready to prove Theorem \ref{sdthm}. The proof combines the terms on each side of the equations obtained in previous three lemmas. 
\begin{proof}[Proof of Theorem \ref{sdthm}]
First note that since $Q$ is unitary,
\begin{align*}
\ee\biggl[q_{Vk}q_{Vk'}+t_{Vk}t_{Vk'} \mid U, Q\biggr]&=\dfrac{1}{N-1}\sum_{v\neq U}(q_{vk}q_{vk'}+t_{vk}t_{vk'})\notag\\
&=\begin{cases}
-(q_{Uk}q_{Uk'}+t_{Uk}t_{Uk'})/(N-1) &\text{ if } k\neq k'\\
\left(1-q_{Uk}^2-t^2_{Uk}\right)/(N-1) &\text{ if } k= k'
\end{cases}
\end{align*}
and
\begin{align*}
\ee\biggl[q_{Vk}t_{Vk'}-t_{Vk}q_{Vk'} \mid U, Q\biggr]&=\dfrac{1}{N-1}\sum_{v\neq U}(q_{vk}t_{vk'}-t_{vk}q_{vk'})\notag\\
&=-(q_{Uk}t_{Uk'}-t_{Uk}q_{Uk'})/(N-1)\, .
\end{align*}
Now, we add equations \eqref{schdys1} and \eqref{schdys2} side by side and multiply both sides with $N(N-1)$. On the left hand side we get
\begin{align}
2N(N-1)\ee\biggl[\sum_{k}\biggl(q_{Uk} \fpar{f}{q_{Uk}}+t_{Uk} \fpar{f}{t_{Uk}}\biggr) g\biggr]=2(N-1)\ee\biggl[\sum_{u,k}\biggl(q_{uk} \fpar{f}{q_{uk}}+t_{uk} \fpar{f}{t_{uk}}\biggr) g\biggr]\, .  \label{left}
\end{align}
Combine corresponding terms on the right hand sides of these equations in the following way.\\\\
$1^{\text{st}}$ terms:
\begin{align}
N(N-1)\ee\biggl[ \sum_{k,k'}(q_{Vk}q_{Vk'}&+t_{Vk}t_{Vk'})\mpar{f}{q_{Uk}}{q_{Uk'}}g\biggr]\notag\\
&=N\ee\biggl[ \sum_{k} \spar{f}{q_{Uk}}g\biggr]-N\ee\biggl[ \sum_{k, k'} (q_{Uk}q_{Uk'}+t_{Uk}t_{Uk'})\mpar{f}{q_{Uk}}{q_{Uk'}}g\biggr]\notag\\
&=\ee\biggl[ \sum_{u,k} \spar{f}{q_{uk}}g\biggr]-\ee\biggl[ \sum_{u, k, k'} (q_{uk}q_{uk'}+t_{uk}t_{uk'})\mpar{f}{q_{uk}}{q_{uk'}}g\biggr]\, , \label{aa1}
\end{align}
$2^{\text{nd}}$ terms:
\begin{align}
N(N-1)\ee\biggl[ \sum_{k,k'}(q_{Vk}q_{Vk'}&+t_{Vk}t_{Vk'}) \mpar{f}{t_{Uk}}{t_{Uk'}}g\biggr]\notag\\
&=N\ee\biggl[ \sum_{k} \spar{f}{t_{Uk}}g\biggr]-N\ee\biggl[ \sum_{k, k'} (q_{Uk}q_{Uk'}+t_{Uk}t_{Uk'})\mpar{f}{t_{Uk}}{t_{Uk'}}g\biggr] \notag\\
&=\ee\biggl[ \sum_{u,k} \spar{f}{t_{uk}}g\biggr]-\ee\biggl[ \sum_{u,k, k'} (q_{uk}q_{uk'}+t_{uk}t_{uk'})\mpar{f}{t_{uk}}{t_{uk'}}g\biggr]\, ,   \label{aa2}
\end{align}
$3^{\text{rd}}$ terms:
\begin{align}
N(N-1)\ee\biggl[ \sum_{k,k'} 2(q_{Vk}t_{Vk'}&-t_{Vk}q_{Vk'})\mpar{f}{q_{Uk}}{t_{Uk'}}g\biggr]=N\ee\biggl[ \sum_{k, k'} -2(q_{Uk}t_{Uk'}-t_{Uk}q_{Uk'})\mpar{f}{q_{Uk}}{t_{Uk'}}g\biggr] \notag\\
&\quad=\ee\biggl[ \sum_{u, k, k'} (2t_{uk}q_{uk'}-2q_{uk}t_{uk'})\mpar{f}{q_{uk}}{t_{uk'}}g\biggr]\, , \label{aa3}
\end{align}
$4^{\text{th}}$, $5^{\text{th}}$ and $6^{\text{th}}$ terms:
\begin{align}
-N(N-1)&\ee\biggl[\sum_{k,k'} \biggl(q_{Vk}q_{Uk'} \mpar{f}{q_{Uk}}{q_{Vk'}} +t_{Vk}t_{Uk'} \mpar{f}{t_{Uk}}{t_{Vk'}} +2q_{Vk}t_{Uk'} \mpar{f}{q_{Uk}}{t_{Vk'}} \biggr)g\biggr]\notag\\
+N(N-1)&\ee\biggl[\sum_{k,k'} \biggl( t_{Vk}t_{Uk'} \mpar{f}{q_{Uk}}{q_{Vk'}}+q_{Vk}q_{Uk'} \mpar{f}{t_{Uk}}{t_{Vk'}}-2t_{Vk}q_{Uk'} \mpar{f}{q_{Uk}}{t_{Vk'}} \biggr)g\biggr]\notag\\
&\quad=-\ee\biggl[\sum_{u,v,k,k'} \biggl(q_{vk}q_{uk'} \mpar{f}{q_{uk}}{q_{vk'}} +t_{vk}t_{uk'} \mpar{f}{t_{uk}}{t_{vk'}} +2q_{vk}t_{uk'} \mpar{f}{q_{uk}}{t_{vk'}} \biggr)g\biggr]\notag\\
&\qquad+\ee\biggl[\sum_{u,v,k,k'} \biggl( t_{vk}t_{uk'} \mpar{f}{q_{uk}}{q_{vk'}}+q_{vk}q_{uk'} \mpar{f}{t_{uk}}{t_{vk'}}-2t_{vk}q_{uk'} \mpar{f}{q_{uk}}{t_{vk'}} \biggr)g\biggr]\notag\\
&\qquad+\ee\biggl[\sum_{u,k,k'} \biggl(q_{uk}q_{uk'} \mpar{f}{q_{uk}}{q_{uk'}} +t_{uk}t_{uk'} \mpar{f}{t_{uk}}{t_{uk'}} +2q_{uk}t_{uk'} \mpar{f}{q_{uk}}{t_{uk'}} \biggr)g\biggr]\notag\\
&\qquad-\ee\biggl[\sum_{u,k,k'} \biggl( t_{uk}t_{uk'} \mpar{f}{q_{uk}}{q_{uk'}}+q_{uk}q_{uk'} \mpar{f}{t_{uk}}{t_{uk'}}-2t_{uk}q_{uk'} \mpar{f}{q_{uk}}{t_{uk'}} \biggr)g\biggr]\, .\label{aa4}
\end{align}
If we add equations \eqref{aa1}, \eqref{aa2}, \eqref{aa3} and \eqref{aa4} side by side we can see that on the right hand side six out of twenty terms cancel each other and three of the remaining terms is same as the another three. Thus, after simplification on the right hand side we get
\begin{align}
&\ee\biggl[ \sum_{u,k} \left(\spar{f}{q_{uk}}+\spar{f}{t_{uk}}\right)g\biggr]\notag\\
&\quad-\ee\biggl[\sum_{u,v,k,k'} \biggl(q_{vk}q_{uk'} \mpar{f}{q_{uk}}{q_{vk'}} +t_{vk}t_{uk'} \mpar{f}{t_{uk}}{t_{vk'}} +2q_{vk}t_{uk'} \mpar{f}{q_{uk}}{t_{vk'}} \biggr)g\biggr]\notag\\
&\quad+\ee\biggl[\sum_{u,v,k,k'} \biggl( t_{vk}t_{uk'} \mpar{f}{q_{uk}}{q_{vk'}}+q_{vk}q_{uk'} \mpar{f}{t_{uk}}{t_{vk'}}-2t_{vk}q_{uk'} \mpar{f}{q_{uk}}{t_{vk'}} \biggr)g\biggr]\notag\\
&\quad-2\ee\biggl[ \sum_{u, k, k'}\left(t_{uk}t_{uk'}\mpar{f}{q_{uk}}{q_{uk'}}+q_{uk}q_{uk'}\mpar{f}{t_{uk}}{t_{uk'}}-2t_{uk}q_{uk'}\mpar{f}{q_{uk}}{t_{uk'}}\right)g\biggr]\, . \label{rhs1}
\end{align}
We apply the same kind of grouping to the rest of the terms.\\\\
$7^{\text{th}}$ terms:
\begin{align}
N(N-1)\ee\biggl[\sum_{k,k'}& \left(q_{Vk}q_{Vk'}+t_{Vk}t_{Vk'}\right) \fpar{f}{q_{Uk}}\fpar{g}{q_{Uk'}}\biggr]\notag\\
&=N\ee\biggl[\sum_{k} \fpar{f}{q_{Uk}}\fpar{g}{q_{Uk}}\biggr]-N\ee\biggl[\sum_{k,k'} \left(q_{Uk}q_{Uk'}+t_{Uk}t_{Uk'}\right) \fpar{f}{q_{Ik}}\fpar{g}{q_{Ik'}}\biggr]\notag\\
&=\ee\biggl[\sum_{u,k} \fpar{f}{q_{uk}}\fpar{g}{q_{uk}}\biggr]- \ee\biggl[\sum_{u,k,k'} \left(q_{uk}q_{uk'}+t_{uk}t_{uk'}\right) \fpar{f}{q_{uk}}\fpar{g}{q_{uk'}}\biggr]\, , \label{aa5}
\end{align}
$8^{\text{th}}$ terms:
\begin{align}
N(N-1)\ee\biggl[\sum_{k,k'}& \left(q_{Vk}q_{Vk'}+t_{Vk}t_{Vk'}\right) \fpar{f}{t_{Uk}}\fpar{g}{t_{Uk'}}\biggr]\notag\\
&=N\ee\biggl[\sum_{k} \fpar{f}{t_{Uk}}\fpar{g}{t_{Uk}}\biggr]-N\ee\biggl[\sum_{k,k'} \left(q_{Uk}q_{Uk'}+t_{Uk}t_{Uk'}\right) \fpar{f}{t_{Uk}}\fpar{g}{t_{Uk'}}\biggr]\notag\\
&=\ee\biggl[\sum_{k} \fpar{f}{t_{uk}}\fpar{g}{t_{uk}}\biggr]-\ee\biggl[\sum_{u,k,k'} \left(q_{uk}q_{uk'}+t_{uk}t_{uk'}\right) \fpar{f}{t_{uk}}\fpar{g}{t_{uk'}}\biggr]\, , \label{aa6}
\end{align}
$9^{\text{th}}$ terms:
\begin{align}
N(N-1)\ee\biggl[\sum_{k,k'}\left(q_{Vk}t_{Vk'}-t_{Vk}q_{Vk'}\right) \fpar{f}{q_{Uk}}\fpar{g}{t_{Uk'}}\biggr]&=N\ee\biggl[\sum_{k,k'} -\left(q_{Uk}t_{Uk'}-t_{Uk}q_{Uk'}\right) \fpar{f}{q_{Uk}}\fpar{g}{t_{Uk'}}\biggr] \notag\\
&=\ee\biggl[\sum_{u, k,k'} \left(t_{uk}q_{uk'}-q_{uk}t_{uk'}\right) \fpar{f}{q_{uk}}\fpar{g}{t_{uk'}}\biggr]\, , \label{aa7}
\end{align}
$10^{\text{th}}$ terms:
\begin{align}
N(N-1)\ee\biggl[\sum_{k,k'}\left(t_{Vk}q_{Vk'}-q_{Vk}t_{Vk'}\right) \fpar{f}{t_{Uk}}\fpar{g}{q_{Uk'}}\biggr]&=N\ee\biggl[\sum_{k,k'} -\left(t_{Uk}q_{Uk'}-q_{Uk}t_{Uk'}\right) \fpar{f}{t_{Uk}}\fpar{g}{q_{Uk'}}\biggr] \notag\\
&=\ee\biggl[\sum_{u, k,k'} \left(q_{uk}t_{uk'}-t_{uk}q_{uk'}\right) \fpar{f}{t_{uk}}\fpar{g}{q_{uk'}}\biggr]\, , \label{aa8}
\end{align}
$11^{\text{th}}$, $12^{\text{th}}$, $13^{\text{th}}$ and $14^{\text{th}}$ terms:
\begin{align}
& -N(N-1)\ee\biggl[\sum_{k,k'} \biggl(q_{Vk}q_{Uk'} \fpar{f}{q_{Uk}}\fpar{g}{q_{Vk'}}+t_{Vk}t_{Uk'} \fpar{f}{t_{Uk}}\fpar{g}{t_{Vk'}} \notag\\
& \qquad \qquad \qquad \qquad \qquad \qquad \qquad \qquad \qquad \qquad +q_{Vk}t_{Uk'} \fpar{f}{q_{Uk}}\fpar{g}{t_{Vk'}}+t_{Vk}q_{Uk'} \fpar{f}{t_{Uk}}\fpar{g}{q_{Vk'}} \biggr)\biggr]\notag\\
&+N(N-1)\ee\biggl[\sum_{k,k'} \biggl(t_{Vk}t_{Uk'} \fpar{f}{q_{Uk}}\fpar{g}{q_{Vk'}}+q_{Vk}q_{Uk'} \fpar{f}{t_{Uk}}\fpar{g}{t_{Vk'}}\notag\\
& \qquad \qquad \qquad \qquad \qquad \qquad \qquad \qquad \qquad \qquad -t_{Vk}q_{Uk'} \fpar{f}{q_{Uk}}\fpar{g}{t_{Vk'}}-q_{Vk}t_{Uk'} \fpar{f}{t_{Uk}}\fpar{g}{q_{Vk'}} \biggr)\biggr]\notag\\
&=-\ee\biggl[\sum_{u,v,k,k'} \biggl(q_{vk}q_{uk'} \fpar{f}{q_{uk}}\fpar{g}{q_{vk'}}+t_{vk}t_{uk'} \fpar{f}{t_{uk}}\fpar{g}{t_{vk'}}+q_{vk}t_{uk'} \fpar{f}{q_{uk}}\fpar{g}{t_{vk'}}+t_{vk}q_{uk'} \fpar{f}{t_{uk}}\fpar{g}{q_{vk'}} \biggr)\biggr]\notag\\
&\quad+\ee\biggl[\sum_{u,v,k,k'} \biggl(t_{vk}t_{uk'} \fpar{f}{q_{uk}}\fpar{g}{q_{vk'}}+q_{vk}q_{uk'} \fpar{f}{t_{uk}}\fpar{g}{t_{vk'}}-t_{vk}q_{uk'} \fpar{f}{q_{uk}}\fpar{g}{t_{vk'}}-q_{vk}t_{uk'} \fpar{f}{t_{uk}}\fpar{g}{q_{vk'}}\biggr)\biggr]\notag\\
&\quad+\ee\biggl[\sum_{u,k,k'} \biggl(q_{uk}q_{uk'} \fpar{f}{q_{uk}}\fpar{g}{q_{uk'}}+t_{uk}t_{uk'} \fpar{f}{t_{uk}}\fpar{g}{t_{uk'}}+q_{uk}t_{uk'} \fpar{f}{q_{uk}}\fpar{g}{t_{uk'}}+t_{uk}q_{uk'} \fpar{f}{t_{uk}}\fpar{g}{q_{uk'}} \biggr)\biggr]\notag\\
&\quad-\ee\biggl[\sum_{u,k,k'} \biggl(t_{uk}t_{uk'} \fpar{f}{q_{uk}}\fpar{g}{q_{uk'}}+q_{uk}q_{uk'} \fpar{f}{t_{uk}}\fpar{g}{t_{uk'}}-t_{uk}q_{uk'} \fpar{f}{q_{uk}}\fpar{g}{t_{uk'}}-q_{uk}t_{uk'} \fpar{f}{t_{uk}}\fpar{g}{q_{uk'}}\biggr)\biggr]\, . \label{aa9}
\end{align}
By adding equations \eqref{aa5}, \eqref{aa6}, \eqref{aa7}, \eqref{aa8} and \eqref{aa9} we see that on the right hand side eight out of twenty six terms cancel each other and four of the remaining terms are same as the another four. So, after simplification on the right hand side we get
\begin{align}
&\ee\biggl[\sum_{u,k} \left(\fpar{f}{q_{uk}}\fpar{g}{q_{uk}}+\fpar{f}{t_{uk}}\fpar{g}{t_{uk}}\right)\biggr]\notag\\
&-\ee\biggl[\sum_{u,v,k,k'} \biggl(q_{vk}q_{uk'} \fpar{f}{q_{uk}}\fpar{g}{q_{vk'}}+t_{vk}t_{uk'} \fpar{f}{t_{uk}}\fpar{g}{t_{vk'}}+q_{vk}t_{uk'} \fpar{f}{q_{uk}}\fpar{g}{t_{vk'}}+t_{vk}q_{uk'} \fpar{f}{t_{uk}}\fpar{g}{q_{vk'}} \biggr)\biggr]\notag\\
&+\ee\biggl[\sum_{u,v,k,k'} \biggl(t_{vk}t_{uk'} \fpar{f}{q_{uk}}\fpar{g}{q_{vk'}}+q_{vk}q_{uk'} \fpar{f}{t_{uk}}\fpar{g}{t_{vk'}}-t_{vk}q_{uk'} \fpar{f}{q_{uk}}\fpar{g}{t_{vk'}}-q_{vk}t_{uk'} \fpar{f}{t_{uk}}\fpar{g}{q_{vk'}}\biggr)\biggr]\notag\\
&-2\ee\biggl[\sum_{u,k,k'} \biggl(t_{uk}t_{uk'} \fpar{f}{q_{uk}}\fpar{g}{q_{uk'}}+q_{uk}q_{uk'} \fpar{f}{t_{uk}}\fpar{g}{t_{uk'}}-t_{uk}q_{uk'} \fpar{f}{q_{uk}}\fpar{g}{t_{uk'}}-q_{uk}t_{uk'} \fpar{f}{t_{uk}}\fpar{g}{q_{uk'}}\biggr)\biggr]\, . \label{rhs2}
\end{align}
Thus equations \eqref{left}, \eqref{rhs1} and \eqref{rhs2} together give
\begin{align}
&2(N-1)\ee\biggl[\sum_{u,k}\biggl(q_{uk} \fpar{f}{q_{uk}}+t_{uk} \fpar{f}{t_{uk}}\biggr) g\biggr]=\ee\biggl[ \sum_{u,k} \left(\spar{f}{q_{uk}}g+\spar{f}{t_{uk}}g+\fpar{f}{q_{uk}}\fpar{g}{q_{uk}}+\fpar{f}{t_{uk}}\fpar{g}{t_{uk}}\right)\biggr]\notag\\
&\quad-\ee\biggl[\sum_{u,v,k,k'} \biggl(q_{vk}q_{uk'} \mpar{f}{q_{uk}}{q_{vk'}} +t_{vk}t_{uk'} \mpar{f}{t_{uk}}{t_{vk'}} +2q_{vk}t_{uk'} \mpar{f}{q_{uk}}{t_{vk'}} \biggr)g\biggr]\notag\\
&\quad+\ee\biggl[\sum_{u,v,k,k'} \biggl( t_{vk}t_{uk'} \mpar{f}{q_{uk}}{q_{vk'}}+q_{vk}q_{uk'} \mpar{f}{t_{uk}}{t_{vk'}}-2t_{vk}q_{uk'} \mpar{f}{q_{uk}}{t_{vk'}} \biggr)g\biggr]\notag\\
&\quad-\ee\biggl[\sum_{u,v,k,k'} \biggl(q_{vk}q_{uk'} \fpar{f}{q_{uk}}\fpar{g}{q_{vk'}}+t_{vk}t_{uk'} \fpar{f}{t_{uk}}\fpar{g}{t_{vk'}}+q_{vk}t_{uk'} \fpar{f}{q_{uk}}\fpar{g}{t_{vk'}}+t_{vk}q_{uk'} \fpar{f}{t_{uk}}\fpar{g}{q_{vk'}} \biggr)\biggr]\notag\\
&\quad+\ee\biggl[\sum_{u,v,k,k'} \biggl(t_{vk}t_{uk'} \fpar{f}{q_{uk}}\fpar{g}{q_{vk'}}+q_{vk}q_{uk'} \fpar{f}{t_{uk}}\fpar{g}{t_{vk'}}-t_{vk}q_{uk'} \fpar{f}{q_{uk}}\fpar{g}{t_{vk'}}-q_{vk}t_{uk'} \fpar{f}{t_{uk}}\fpar{g}{q_{vk'}} \biggr)\biggr]\notag\\
&\quad -2\ee\biggl[\sum_{u,k,k'} \biggl(t_{uk}t_{uk'} \fpar{f}{q_{uk}}\fpar{g}{q_{uk'}}+q_{uk}q_{uk'} \fpar{f}{t_{uk}}\fpar{g}{t_{uk'}}-t_{uk}q_{uk'} \fpar{f}{q_{uk}}\fpar{g}{t_{uk'}}-q_{uk}t_{uk'} \fpar{f}{t_{uk}}\fpar{g}{q_{uk'}} \biggr)\biggr]\notag\\
&\quad-2\ee\biggl[ \sum_{u, k, k'}\left(t_{uk}t_{uk'}\mpar{f}{q_{uk}}{q_{uk'}}+q_{uk}q_{uk'}\mpar{f}{t_{uk}}{t_{uk'}}-2t_{uk}q_{uk'}\mpar{f}{q_{uk}}{t_{uk'}}\right)g\biggr] \, . \label{sbys1}
\end{align}
Observe that the last two lines on the right side of the above equation are similar to the first two lines on the right side of equation \eqref{schdys3}. We will use this observation to simplify equation \eqref{sbys1}. To do so we will first simplify equation \eqref{schdys3}. After multiplying both sides of this equation by $N(N-1)$ on the left side we get
\begin{align}
(N-1)\ee\biggl[\sum_{u,k}\biggl(q_{uk} \fpar{f}{q_{uk}}+t_{uk} \fpar{f}{t_{uk}}\biggr) g\biggr]\, . \label{aa10}
\end{align}
The first two lines on the right side are
\begin{align}
&(N-1)\ee\biggl[\sum_{u,k,k'} \biggl(t_{uk}t_{uk'} \fpar{f}{q_{uk}}\fpar{g}{q_{uk'}}+q_{uk}q_{uk'} \fpar{f}{t_{uk}}\fpar{g}{t_{uk'}}- t_{uk}q_{uk'} \fpar{f}{q_{uk}}\fpar{g}{t_{uk'}}-q_{uk}t_{uk'} \fpar{f}{t_{uk}}\fpar{g}{q_{uk'}}\biggr)\biggr]\notag \\
&+(N-1)\ee\biggl[ \sum_{u,k,k'} \biggl(t_{uk}t_{uk'} \mpar{f}{q_{uk}}{q_{uk'}}+q_{uk}q_{uk'} \mpar{f}{t_{uk}}{t_{uk'}}-2t_{uk}q_{uk'} \mpar{f}{q_{uk}}{t_{uk'}}\biggr)g \biggr] \, . \label{aa11}
\end{align}
Moreover, the last two lines can be written as
\begin{align}
&-\ee\biggl[\sum_{u,v,k,k'} \biggl(t_{uk}t_{vk'} \fpar{f}{q_{uk}}\fpar{g}{q_{vk'}}+ q_{uk}q_{vk'} \fpar{f}{t_{uk}}\fpar{g}{t_{vk'}}- t_{uk}q_{vk'} \fpar{f}{q_{uk}}\fpar{g}{t_{vk'}}- q_{uk}t_{vk'} \fpar{f}{t_{uk}}\fpar{g}{q_{vk'}} \biggr)\biggr]\notag \\
&-\ee\biggl[\sum_{u,v,k,k'} \biggl( t_{uk}t_{vk'} \mpar{f}{q_{uk}}{q_{vk'}}+ q_{uk}q_{vk'} \mpar{f}{t_{uk}}{t_{vk'}}- 2t_{uk}q_{vk'} \mpar{f}{q_{uk}}{t_{vk'}} \biggr)g\biggr]\notag\\
&+\ee\biggl[\sum_{u,k,k'} \biggl(t_{uk}t_{uk'} \fpar{f}{q_{uk}}\fpar{g}{q_{uk'}}+ q_{uk}q_{uk'} \fpar{f}{t_{uk}}\fpar{g}{t_{uk'}}- t_{uk}q_{uk'} \fpar{f}{q_{uk}}\fpar{g}{t_{uk'}}- q_{uk}t_{uk'} \fpar{f}{t_{uk}}\fpar{g}{q_{uk'}} \biggr)\biggr] \notag\\
&+\ee\biggl[\sum_{u,k,k'} \biggl( t_{uk}t_{uk'} \mpar{f}{q_{uk}}{q_{uk'}}+ q_{uk}q_{uk'} \mpar{f}{t_{uk}}{t_{uk'}}- 2t_{uk}q_{uk'} \mpar{f}{q_{uk}}{t_{uk'}} \biggr)g\biggr]\, . \label{aa12}
\end{align}
By combining \eqref{aa10}, \eqref{aa11} and \eqref{aa12} we get
\begin{align}
&(N-1)\ee\biggl[\sum_{u,k}\biggl(q_{uk} \fpar{f}{q_{uk}}+t_{uk} \fpar{f}{t_{uk}}\biggr) g\biggr]\notag\\
&=-\ee\biggl[\sum_{u,v,k,k'} \biggl( t_{uk}t_{vk'} \mpar{f}{q_{uk}}{q_{vk'}}+ q_{uk}q_{vk'} \mpar{f}{t_{uk}}{t_{vk'}}- 2t_{uk}q_{vk'} \mpar{f}{q_{uk}}{t_{vk'}} \biggr)g\biggr]\notag\\
&\quad -\ee\biggl[\sum_{u,v,k,k'} \biggl(t_{uk}t_{vk'} \fpar{f}{q_{uk}}\fpar{g}{q_{vk'}}+ q_{uk}q_{vk'} \fpar{f}{t_{uk}}\fpar{g}{t_{vk'}}- t_{uk}q_{vk'} \fpar{f}{q_{uk}}\fpar{g}{t_{vk'}} - q_{uk}t_{vk'} \fpar{f}{t_{uk}}\fpar{g}{q_{vk'}}\biggr)\biggr]\notag \\
&\quad+N\ee\biggl[\sum_{u,k,k'} \biggl(t_{uk}t_{uk'} \fpar{f}{q_{uk}}\fpar{g}{q_{uk'}}+ q_{uk}q_{uk'} \fpar{f}{t_{uk}}\fpar{g}{t_{uk'}}- t_{uk}q_{uk'} \fpar{f}{q_{uk}}\fpar{g}{t_{uk'}} - q_{uk}t_{uk'} \fpar{f}{t_{uk}}\fpar{g}{q_{uk'}}\biggr)\biggr] \notag\\
&\quad+N\ee\biggl[\sum_{u,k,k'} \biggl( t_{uk}t_{uk'} \mpar{f}{q_{uk}}{q_{uk'}}+ q_{uk}q_{uk'} \mpar{f}{t_{uk}}{t_{uk'}}- 2t_{uk}q_{uk'} \mpar{f}{q_{uk}}{t_{uk'}} \biggr)g\biggr]\, . \label{sbys2}
\end{align}
We finish the proof by multiplying \eqref{sbys1} by $N$ and multiplying \eqref{sbys2} by $2$, and then adding them side by side.
\end{proof}
\section{The master loop equation for finite $N$}\label{mm1sec}
The goal of this section is to prove Theorem~\ref{symfinmaster}. For any edge $e$ let $\cp(e)$ and $\cp^+(e)$ be as in Section \ref{string}. Recall that for any loop $l=e_1e_2\cdots e_k$,
\[
W_{l}=\tr(Q_{e_1}Q_{e_2}\cdots Q_{e_k})\, .
\]
The following theorem will be the main part of our proof. 
\begin{thm}\label{mastern}
Take any $N\ge 2$, $\beta\in \rr$ and a finite set $\Lambda\subseteq \zz^d$, and consider $SU(N)$ lattice gauge theory on $\Lambda$ at inverse coupling strength $\beta$. Take a loop sequence $s$ with minimal representation $(l_1,\ldots, l_n)$ such that all vertices of $\zz^d$ that are at distance $\le 1$ from any of the loops in $s$ are contained in $\Lambda$. Let $e$ be the first edge of $l_1$. For each $1\le r\le n$, let $A_r$ be the set of locations in $l_r$ where $e$ occurs, and let $B_r$ be the set of locations in $l_r$ where $e^{-1}$ occurs. Let $C_r = A_r \cup B_r$ and let $m$ be the size of $C_1$. Let $t_r=|A_r|-|B_r|$ and $t=t_1+\cdots+t_n$. Then 
\begin{align}
\biggl(mN-\frac{t_1t}{N}\biggr)\smallavg{W_{l_1}W_{l_2}\cdots W_{l_n}} &=\textup{splitting term}+ \textup{merger term}+ \textup{deformation term}\notag \\
&\quad +\textup{expansion term}  \label{unsym}\, , 
\end{align}
where the splitting term is given by
\begin{align*}
&\sum_{x\in A_1, \, y\in B_1} \smallavg{W_{\times_{x,y}^1 l_1} W_{\times_{x,y}^2 l_1}W_{l_2}\cdots W_{l_n}}+\sum_{x\in B_1, \, y\in A_1} \smallavg{W_{\times_{x,y}^1 l_1} W_{\times_{x,y}^2 l_1}W_{l_2}\cdots W_{l_n}}\\
&\quad - \sum_{\substack{x,y\in A_1\\ x\ne y}} \smallavg{W_{\times^1_{x,y} l_1} W_{\times^2_{x,y} l_1}W_{l_2}\cdots W_{l_n}} - \sum_{\substack{x,y\in B_1\\ x\ne y}} \smallavg{W_{\times^1_{x,y} l_1} W_{\times^2_{x,y} l_1}W_{l_2}\cdots W_{l_n}}\, ,
\end{align*}
the merger term is given by
\begin{align*}
&\sum_{r=2}^n\sum_{x\in A_1, \, y\in B_r} \bigavg{W_{l_1 \ominus_{x,y} l_r}\prod_{\substack{2\le t\le n\\t\ne r}} W_{l_t}}+\sum_{r=2}^n\sum_{x\in B_1, \, y\in A_r} \bigavg{W_{l_1 \ominus_{x,y} l_r}\prod_{\substack{2\le t\le n\\t\ne r}} W_{l_t}}\\
&\quad-\sum_{r=2}^n\sum_{x\in A_1, \, y\in A_r}\bigavg{ W_{l_1 \oplus_{x,y} l_r}\prod_{\substack{2\le t\le n\\t\ne r}} W_{l_t}}-\sum_{r=2}^n\sum_{x\in B_1, \, y\in B_r}\bigavg{ W_{l_1 \oplus_{x,y} l_r}\prod_{\substack{2\le t\le n\\t\ne r}} W_{l_t}}\, ,
\end{align*}
the deformation term is given by
\begin{align*}
&\frac{N\beta}{2} \sum_{p\in \cp^+(e)}\sum_{x\in C_1}\smallavg{ W_{l_1 \ominus_{x}p}W_{l_2}\cdots W_{l_n}}-\frac{N\beta}{2} \sum_{p\in \cp^+(e)}\sum_{x\in C_1}\smallavg{ W_{l_1 \oplus_{x}p}W_{l_2}\cdots W_{l_n}}  \, ,
\end{align*}
and the expansion term is given by
\begin{align*}
&\frac{\beta}{2}\sum_{x\in A_1}\sum_{p\in \cp(e)}\smallavg{ W_{l_1}W_{p}W_{l_2}\cdots W_{l_n}} + \frac{\beta}{2}\sum_{x\in B_1}\sum_{p\in \cp(e^{-1})}\smallavg{ W_{l_1}W_{p}W_{l_2}\cdots W_{l_n}} \\
&\quad - \frac{\beta}{2}\sum_{x\in B_1}\sum_{p\in \cp(e)}\smallavg{ W_{l_1}W_{p}W_{l_2}\cdots W_{l_n}} - \frac{\beta}{2}\sum_{x\in A_1}\sum_{p\in \cp(e^{-1})}\smallavg{ W_{l_1}W_{p}W_{l_2}\cdots W_{l_n}} \, .
\end{align*}
In all of the above, empty sums denote zero.
\end{thm}
The proof of Theorem \ref{mastern} requires several preparatory lemmas. In all of the following lemmas $A_1$, $B_1$, $C_1$, $m$, $t_1$ and $t$ are defined to be as in Theorem \ref{mastern}.
\begin{lmm}\label{lem1}
For any edge $e$,
\[
\fpar{Q_e}{t_{uk}^e}=i\fpar{Q_e}{q_{uk}^e} \, , \ \text{ and } \ \fpar{Q^*_e}{t_{uk}^e}=-i\fpar{Q^*_e}{q_{uk}^e} \, .
\]
Moreover,
\[
\sum_{u,k} \biggl(q_{uk}^e \fpar{Q_e}{q_{uk}^e}+t_{uk}^e \fpar{Q_e}{t_{uk}^e}\biggr) = Q_e\, , \ \text{ and } \ \sum_{u,k} \biggl(q_{uk}^e \fpar{Q_e^*}{q_{uk}^e}+t_{uk}^e \fpar{Q_e^*}{t_{uk}^e}\biggr) = Q_e^*\, .
\]
\end{lmm}
\begin{proof}
Note that 
\[
\fpar{Q_e}{q^e_{uk}} = b_u b_k^T \, , \ \text{ and } \  \fpar{Q_e}{t^e_{uk}} = ib_u b_k^T \, ,
\]
where $b_u\in \rr^N$ is the vector whose $u^{\text{th}}$ standard basis vector. Thus, 
\[
\sum_{u,k} \biggl(q_{uk}^e \fpar{Q_e}{q_{uk}^e}+t_{uk}^e \fpar{Q_e}{t_{uk}^e}\biggr) = \sum_{u,k} (q_{uk}^e+it_{uk}) b_u b_k^T = Q_e\, ,
\]
Similarly,
\[
\fpar{Q^*_e}{q^e_{uk}} = b_k b_u^T \, , \ \text{ and } \  \fpar{Q^*_e}{t^e_{uk}} = -ib_k b_u^T \, ,
\]
and 
\[
\sum_{u,k} \biggl(q_{uk}^e \fpar{Q^*_e}{q_{uk}^e}+t_{uk}^e \fpar{Q^*_e}{t_{uk}^e}\biggr) = \sum_{u,k} (q_{uk}^e-it_{uk}) b_k b_u^T = Q^*_e\, .
\]
This completes the proof of the lemma. 
\end{proof}
\begin{lmm}\label{lem2}
For any edge $e$,
\[
\sum_{u,k} \biggl( q_{uk}^e \fpar{W_{l_1}}{q_{uk}^e}+t_{uk}^e \fpar{W_{l_1}}{t_{uk}^e}\biggr) = m W_{l_1} \, .
\]
\end{lmm}
\begin{proof}
Let
\[
W_{l_1} = \tr(P_1 Q_1P_2Q_2\cdots P_m Q_m P_{m+1})\, ,
\]
where each $Q_1,\ldots, Q_m$ is either $Q_e$ or $Q_e^*$, and each $P_1,\ldots, P_{m+1}$ is a product of $Q_{e'}$ where $e'$ is neither $e$ nor $e^{-1}$. Note that
\[
\fpar{W_{l_1}}{q_{uk}^e} = \sum_{r=1}^m \tr\biggl(P_1Q_1\cdots P_r \fpar{Q_r}{q_{uk}^e} P_{r+1}\cdots P_mQ_mP_{m+1}\biggr)
\]
and
\[
\fpar{W_{l_1}}{t_{uk}^e} = \sum_{r=1}^m \tr\biggl(P_1Q_1\cdots P_r \fpar{Q_r}{t_{uk}^e} P_{r+1}\cdots P_mQ_mP_{m+1}\biggr)\, .
\]
Thus, by Lemma \ref{lem1}
\begin{align*}
\sum_{u,k} \biggl( q_{uk}^e \fpar{W_{l_1}}{q_{uk}^e}+t_{uk}^e \fpar{W_{l_1}}{t_{uk}^e}\biggr)
&= \sum_{u,k} \biggl(q_{uk}^e\sum_{r=1}^m\tr\biggl(P_1Q_1\cdots P_r \fpar{Q_r}{q_{uk}^e} P_{r+1}\cdots P_mQ_mP_{m+1}\biggr)\\
&\qquad\quad+t_{uk}^e\sum_{r=1}^m \tr\biggl(P_1Q_1\cdots P_r \fpar{Q_r}{t_{uk}^e} P_{r+1}\cdots P_mQ_mP_{m+1}\biggr)\biggr)\\
&= \sum_{r=1}^m \tr\biggl(P_1Q_1\cdots P_r\biggl(\sum_{u,k} \biggl(q_{uk}^e \fpar{Q_r}{q_{uk}^e}+t_{uk}^e \fpar{Q_r}{t_{uk}^e}\biggr)\biggr)P_{r+1}\cdots P_mQ_mP_{m+1}\biggr)\\
&= \sum_{r=1}^m \tr(P_1Q_1\cdots P_r Q_rP_{r+1}\cdots P_mQ_mP_{m+1}) = m W_{l_1}\, .
\end{align*}
This completes the proof the lemma. 
\end{proof}
\begin{lmm}\label{lem3}
\begin{align*}
\sum_{u,k}\biggl(\spar{W_{l_1}}{{q^e_{uk}}} +\spar{W_{l_1}}{{t^e_{uk}}}\biggr) &= 2\sum_{x\in A_1,\, y\in B_1} W_{\times^1_{x,y} l_1} W_{\times^2_{x,y} l_1}+2\sum_{x\in B_1,\, y\in A_1} W_{\times^1_{x,y} l_1} W_{\times^2_{x,y} l_1}\, .
\end{align*}
\end{lmm}
\begin{proof}
Using the notations from Lemma \ref{lem2}
\begin{align*}
\sum_{u,k} \spar{W_{l_1}}{{q^e_{uk}}}
 &= 2\sum_{1\le r< s\le m}\sum_{u,k}\tr\biggl(P_1Q_1\cdots P_r \fpar{Q_r}{q_{uk}^e} P_{r+1}\cdots P_s \fpar{Q_s}{q_{uk}^e}P_{s+1}\cdots P_mQ_mP_{m+1}\biggr)
\end{align*}
and
\begin{align*}
\sum_{u,k} \spar{W_{l_1}}{{t^e_{uk}}}
 &= 2\sum_{1\le r< s\le m}\sum_{u,k}\tr\biggl(P_1Q_1\cdots P_r \fpar{Q_r}{t_{uk}^e} P_{r+1}\cdots P_s \fpar{Q_s}{t_{uk}^e}P_{s+1}\cdots P_mQ_mP_{m+1}\biggr)\, .
\end{align*}
Observe that if $(Q_r, Q_s)=(Q_e, Q_e)$ or $(Q_r, Q_s)=(Q_e^*, Q_e^*)$, then by Lemma \ref{lem1} 
\begin{align*}
P_1Q_1\cdots P_r &\fpar{Q_r}{q_{uk}^e} P_{r+1}\cdots P_s \fpar{Q_s}{q_{uk}^e}P_{s+1}\cdots P_mQ_mP_{m+1}\\
&+ P_1Q_1\cdots P_r \fpar{Q_r}{t_{uk}^e} P_{r+1}\cdots P_s \fpar{Q_s}{t_{uk}^e}P_{s+1}\cdots P_mQ_mP_{m+1} = 0\, . 
\end{align*}
So these pairs do not contribute anything to the sum
$$\sum_{u,k}\biggl(\spar{W_{l_1}}{{q^e_{uk}}} +\spar{W_{l_1}}{{t^e_{uk}}}\biggr)\, .$$
Now, suppose that $Q_r = Q_e$ and $Q_s = Q_e^*$. Let $b_u$ be as in the proof of Lemma \ref{lem1}. Then 
\begin{align}
&\tr\biggl(P_1Q_1\cdots P_r \fpar{Q_r}{q_{uk}^e} P_{r+1}\cdots P_s \fpar{Q_s}{q_{uk}^e}P_{s+1}\cdots P_mQ_mP_{m+1}\biggr) \nonumber\\
&\quad= \tr(P_1Q_1\cdots P_r b_u b_k^T P_{r+1}\cdots P_s b_k b_{u}^TP_{s+1}\cdots P_mQ_mP_{m+1})\nonumber\\
&\quad= (b_{u}^TP_{s+1}\cdots P_mQ_mP_{m+1}P_1Q_1\cdots P_r b_u)( b_k^T P_{r+1}\cdots P_s b_k)\nonumber\\
&\quad= (P_{s+1}\cdots P_mQ_mP_{m+1}P_1Q_1\cdots P_r)_{uu}(P_{r+1}\cdots P_s)_{kk}\, . \nonumber
\end{align}
where we are following the convention that $M_{ij}$ denotes the $(i,j)$th entry of a matrix $M$. Similarly,
\begin{align}
&\tr\biggl(P_1Q_1\cdots P_r \fpar{Q_r}{t_{uk}^e} P_{r+1}\cdots P_s \fpar{Q_s}{t_{uk}^e}P_{s+1}\cdots P_mQ_mP_{m+1}\biggr) \nonumber\\
&\quad= (P_{s+1}\cdots P_mQ_mP_{m+1}P_1Q_1\cdots P_r)_{uu}(P_{r+1}\cdots P_s)_{kk}\, . \nonumber
\end{align}
Then
\begin{align*}
&\sum_{u,k} \tr\biggl(P_1Q_1\cdots P_r \fpar{Q_r}{q_{uk}^e} P_{r+1}\cdots P_s \fpar{Q_s}{q_{uk}^e}P_{s+1}\cdots P_mQ_mP_{m+1}\biggr)\\
&\qquad+\sum_{u,k} \tr\biggl(P_1Q_1\cdots P_r \fpar{Q_r}{t_{uk}^e} P_{r+1}\cdots P_s \fpar{Q_s}{t_{uk}^e}P_{s+1}\cdots P_mQ_mP_{m+1}\biggr)\\
&= 2\sum_{u,k}(P_{s+1}\cdots P_mQ_mP_{m+1}P_1Q_1\cdots P_r)_{uu}(P_{r+1}\cdots P_s)_{kk}\\
&= 2\tr (P_{s+1}\cdots P_mQ_mP_{m+1}P_1Q_1\cdots P_r)\tr(P_{r+1}Q_{r+1}\cdots Q_{s-1}P_s)\\
&= 2\tr (P_1Q_1\cdots P_rP_{s+1}\cdots P_mQ_mP_{m+1})\tr(P_{r+1}Q_{r+1}\cdots Q_{s-1}P_s)\, .
\end{align*}
By the definition of a negative splitting, the sum of above quantities over such $(r,s)$ pairs gives
\[
2\sum_{x\in A_1, y\in B_1}W_{\times^1_{x,y} l_1} W_{\times^2_{x,y} l_1}\, .
\]
This produces the first term on the right. Similarly, if $Q_r = Q_e^*$ and $Q_s = Q_e$, then 
\begin{align*}
&\sum_{u,k} \tr\biggl(P_1Q_1\cdots P_r \fpar{Q_r}{q_{uk}^e} P_{r+1}\cdots P_s \fpar{Q_s}{q_{uk}^e}P_{s+1}\cdots P_mQ_mP_{m+1}\biggr)\\
&\qquad+\sum_{u,k} \tr\biggl(P_1Q_1\cdots P_r \fpar{Q_r}{t_{uk}^e} P_{r+1}\cdots P_s \fpar{Q_s}{t_{uk}^e}P_{s+1}\cdots P_mQ_mP_{m+1}\biggr)\\
&= 2\sum_{u,k}(P_{s+1}\cdots P_mQ_mP_{m+1}P_1Q_1\cdots P_r)_{kk}(P_{r+1}\cdots P_s)_{uu}\\
&= 2\tr (P_{s+1}\cdots P_mQ_mP_{m+1}P_1Q_1\cdots P_r)\tr(P_{r+1}Q_{r+1}\cdots Q_{s-1}P_s)\\
&= 2\tr (P_1Q_1\cdots P_rP_{s+1}\cdots P_mQ_mP_{m+1})\tr(P_{r+1}Q_{r+1}\cdots Q_{s-1}P_s)\, ,
\end{align*}
and the sum of these quantities over such $(r,s)$ pairs is equal to
\[
2\sum_{x\in B_1, y\in A_1} W_{\times^1_{x,y} l_1} W_{\times^2_{x,y} l_1}\, .
\]
This gives the second sum on the right side.
\end{proof}
\begin{lmm}\label{lem4}
\begin{align*}
\sum_{u,v,k,k'} \biggl((q_{vk}q_{uk'}-t_{vk}t_{uk'})\biggl(\mpar{W_{l_1}}{q_{uk}}{q_{vk'}}&-\mpar{W_{l_1}}{t_{uk}}{t_{vk'}}\biggr) +2(q_{vk}t_{uk'}+t_{vk}q_{uk'})\mpar{W_{l_1}}{q_{uk}}{t_{vk'}}\biggr)\\
&=  2\sum_{\substack{x,y\in A_1\\ x\ne y}}W_{\times^1_{x,y} l_1} W_{\times^2_{x,y} l_1} + 2\sum_{\substack{x,y\in B_1\\ x\ne y}}W_{\times^1_{x,y} l_1} W_{\times^2_{x,y} l_1} \, .
\end{align*}
\end{lmm}
\begin{proof}
First observe that by renaming $(u, v, k, k')$ as $(v, u, k', k)$ we get
\[
\sum_{u,v,k,k'} (q_{vk}t_{uk'}+t_{vk}q_{uk'})\mpar{W_{l_1}}{q_{uk}}{t_{vk'}}=\sum_{u,v,k,k'} (q_{uk'}t_{vk}+t_{uk'}q_{vk})\mpar{W_{l_1}}{q_{vk'}}{t_{uk}} \, .
\]
Thus
\[
\sum_{u,v,k,k'} 2(q_{vk}t_{uk'}+t_{vk}q_{uk'})\mpar{W_{l_1}}{q_{uk}}{t_{vk'}}=\sum_{u,v,k,k'} (q_{vk}t_{uk'}+t_{vk}q_{uk'})\biggl(\mpar{W_{l_1}}{q_{uk}}{t_{vk'}}+\mpar{W_{l_1}}{q_{vk'}}{t_{uk}}\biggr) \, .
\]
Note that
\begin{align}
\mpar{W_{l_1}}{q_{uk}^e}{q_{vk'}^e}
&= \sum_{1\le r< s\le m}\tr\biggl(P_1Q_1\cdots P_r \fpar{Q_r}{q_{uk}^e} P_{r+1}\cdots P_s \fpar{Q_s}{q_{vk'}^e}P_{s+1}\cdots P_mQ_mP_{m+1}\biggr)\notag\\
&\quad + \sum_{1\le r<s\le m} \tr\biggl(P_1Q_1\cdots P_r \fpar{Q_r}{q_{vk'}^e} P_{r+1}\cdots P_s \fpar{Q_s}{q_{uk}^e}P_{s+1}\cdots P_mQ_mP_{m+1}\biggr)\, \label{pqeq0}.
\end{align}
Similarly,
\begin{align*}
\mpar{W_{l_1}}{t_{uk}^e}{t_{vk'}^e}&=\sum_{1\le r< s\le m} \tr\biggl(P_1Q_1\cdots P_r \fpar{Q_r}{t_{uk}^e} P_{r+1}\cdots P_s \fpar{Q_s}{t_{vk'}^e}P_{s+1}\cdots P_mQ_mP_{m+1}\biggr)\\
&\quad +\sum_{1\le r< s\le m} \tr\biggl(P_1Q_1\cdots P_r \fpar{Q_r}{t_{vk'}^e} P_{r+1}\cdots P_s \fpar{Q_s}{t_{uk}^e}P_{s+1}\cdots P_mQ_mP_{m+1}\biggr)\,,
\end{align*}
\begin{align*}
\mpar{W_{l_1}}{q_{uk}^e}{t_{vk'}^e}&= \sum_{1\le r< s\le m} \tr\biggl(P_1Q_1\cdots P_r \fpar{Q_r}{q_{uk}^e} P_{r+1}\cdots P_s \fpar{Q_s}{t_{vk'}^e}P_{s+1}\cdots P_mQ_mP_{m+1}\biggr)\\
&\quad + \sum_{1\le r<s\le m} \tr\biggl(P_1Q_1\cdots P_r \fpar{Q_r}{t_{vk'}^e} P_{r+1}\cdots P_s \fpar{Q_s}{q_{uk}^e}P_{s+1}\cdots P_mQ_mP_{m+1}\biggr)\,,
\end{align*}
and
\begin{align*}
\mpar{W_{l_1}}{t_{uk}^e}{q_{vk'}^e}&= \sum_{1\le r< s\le m} \tr\biggl(P_1Q_1\cdots P_r \fpar{Q_r}{t_{uk}^e} P_{r+1}\cdots P_s \fpar{Q_s}{q_{vk'}^e}P_{s+1}\cdots P_mQ_mP_{m+1}\biggr)\\
&\quad + \sum_{1\le r<s\le m} \tr\biggl(P_1Q_1\cdots P_r \fpar{Q_r}{q_{vk'}^e} P_{r+1}\cdots P_s \fpar{Q_s}{t_{uk}^e}P_{s+1}\cdots P_mQ_mP_{m+1}\biggr)\,.
\end{align*}
As in the proof of Lemma \ref{lem3} we can see that by Lemma \ref{lem1} the pairs $(Q_r, Q_s)=(Q_e, Q_e^*)$ and $(Q_r, Q_s)=(Q_e^*, Q_e)$ do not contribute anything to the sum on the left hand side of the equation in the statement of the lemma. Now, let $1\le r<s\le m$ and suppose that $Q_r = Q_s = Q_e$. If $b_u$ is as in the proof of Lemma \ref{lem1}, then 
\begin{align}
&\tr\biggl(P_1Q_1\cdots P_r \fpar{Q_r}{q_{uk}^e} P_{r+1}\cdots P_s \fpar{Q_s}{q_{vk'}^e}P_{s+1}\cdots P_mQ_mP_{m+1}\biggr) \nonumber\\
&\quad= \tr(P_1Q_1\cdots P_r b_u b_k^T P_{r+1}\cdots P_s b_v b_{k'}^TP_{s+1}\cdots P_mQ_mP_{m+1})\nonumber\\
&\quad= (b_{k'}^TP_{s+1}\cdots P_mQ_mP_{m+1}P_1Q_1\cdots P_r b_u)( b_k^T P_{r+1}\cdots P_s b_v)\nonumber\\
&\quad= (P_{s+1}\cdots P_mQ_mP_{m+1}P_1Q_1\cdots P_r)_{k'u}(P_{r+1}\cdots P_s)_{kv}\, , \label{pqeq1}
\end{align}
Similarly,
\begin{align*}
&\tr\biggl(P_1Q_1\cdots P_r \fpar{Q_r}{t_{uk}^e} P_{r+1}\cdots P_s \fpar{Q_s}{t_{vk'}^e}P_{s+1}\cdots P_mQ_mP_{m+1}\biggr) \nonumber\\
&\qquad=-(P_{s+1}\cdots P_mQ_mP_{m+1}P_1Q_1\cdots P_r)_{k'u}(P_{r+1}\cdots P_s)_{kv}\, ,
\end{align*}
\begin{align*}
&\tr\biggl(P_1Q_1\cdots P_r \fpar{Q_r}{q_{uk}^e} P_{r+1}\cdots P_s \fpar{Q_s}{t_{vk'}^e}P_{s+1}\cdots P_mQ_mP_{m+1}\biggr) \nonumber\\
&\qquad=i(P_{s+1}\cdots P_mQ_mP_{m+1}P_1Q_1\cdots P_r)_{k'u}(P_{r+1}\cdots P_s)_{kv}
\end{align*}
and
\begin{align*}
&\tr\biggl(P_1Q_1\cdots P_r \fpar{Q_r}{t_{uk}^e} P_{r+1}\cdots P_s \fpar{Q_s}{q_{vk'}^e}P_{s+1}\cdots P_mQ_mP_{m+1}\biggr) \nonumber\\
&\qquad=i(P_{s+1}\cdots P_mQ_mP_{m+1}P_1Q_1\cdots P_r)_{k'u}(P_{r+1}\cdots P_s)_{kv}\, .
\end{align*}
Thus plugging these in above equations and using symmetry between pairs $(u,k)$ and $(v, k')$ we get
\begin{align*}
&2\sum_{u, v, k, k'}(q_{vk}q_{uk'}-t_{vk}t_{uk'}+iq_{vk}t_{uk'}+it_{vk}q_{uk'})(P_{s+1}\cdots P_mQ_mP_{m+1}P_1Q_1\cdots P_r)_{k'u}(P_{r+1}\cdots P_s)_{kv}\\
&\quad= 2\sum_{u, v, k, k'}(q_{vk}+it_{vk})(q_{uk'}+it_{uk'})(P_{s+1}\cdots P_mQ_mP_{m+1}P_1Q_1\cdots P_r)_{k'u}(P_{r+1}\cdots P_s)_{kv}\\
&\quad= 2\tr(P_{s+1}\cdots P_mQ_mP_{m+1}P_1Q_1\cdots P_rQ_r)\tr(P_{r+1}\cdots P_sQ_s)\\
&\quad= 2\tr(P_1Q_1\cdots P_rQ_rP_{s+1}\cdots P_mQ_mP_{m+1})\tr(P_{r+1}\cdots P_sQ_s)\,.
\end{align*}
By the definition of a positive splitting, the sum of these variables over all such $(r, s)$ pairs equals
\[
2\sum_{\substack{x,y\in A_1\\ x\ne y}}W_{\times^1_{x,y} l_1} W_{\times^2_{x,y} l_1}\, .
\]
Next, note that if $Q_r=Q_s=Q_e^*$, then
\begin{align}
&\tr\biggl(P_1Q_1\cdots P_r \fpar{Q_r}{q_{uk}^e} P_{r+1}\cdots P_s \fpar{Q_s}{q_{vk'}^e}P_{s+1}\cdots P_mQ_mP_{m+1}\biggr) \nonumber\\
&\quad= \tr(P_1Q_1\cdots P_r b_k b_u^T P_{r+1}\cdots P_s b_{k'} b_{v}^TP_{s+1}\cdots P_mQ_mP_{m+1})\nonumber\\
&\quad= (b_{v}^TP_{s+1}\cdots P_mQ_mP_{m+1}P_1Q_1\cdots P_r b_k)( b_u^T P_{r+1}\cdots P_s b_{k'})\nonumber\\
&\quad= (P_{s+1}\cdots P_mQ_mP_{m+1}P_1Q_1\cdots P_r)_{vk}(P_{r+1}\cdots P_s)_{uk'}\, . \label{pqeq2}
\end{align}
Similarly,
\begin{align*}
&\tr\biggl(P_1Q_1\cdots P_r \fpar{Q_r}{t_{uk}^e} P_{r+1}\cdots P_s \fpar{Q_s}{t_{vk'}^e}P_{s+1}\cdots P_mQ_mP_{m+1}\biggr) \nonumber\\
&\qquad=-(P_{s+1}\cdots P_mQ_mP_{m+1}P_1Q_1\cdots P_r)_{vk}(P_{r+1}\cdots P_s)_{uk'}\, ,
\end{align*}
\begin{align*}
&\tr\biggl(P_1Q_1\cdots P_r \fpar{Q_r}{q_{uk}^e} P_{r+1}\cdots P_s \fpar{Q_s}{t_{vk'}^e}P_{s+1}\cdots P_mQ_mP_{m+1}\biggr) \nonumber\\
&\qquad=-i(P_{s+1}\cdots P_mQ_mP_{m+1}P_1Q_1\cdots P_r)_{vk}(P_{r+1}\cdots P_s)_{uk'} 
\end{align*}
and
\begin{align*}
&\tr\biggl(P_1Q_1\cdots P_r \fpar{Q_r}{t_{uk}^e} P_{r+1}\cdots P_s \fpar{Q_s}{q_{vk'}^e}P_{s+1}\cdots P_mQ_mP_{m+1}\biggr) \nonumber\\
&\qquad=-i(P_{s+1}\cdots P_mQ_mP_{m+1}P_1Q_1\cdots P_r)_{vk}(P_{r+1}\cdots P_s)_{uk'}\, . 
\end{align*}
Thus, by symmetry between pairs $(u,k)$ and $(v, k')$ we get
\begin{align*}
& 2\sum_{u, v, k, k'}(q_{vk}q_{uk'}-t_{vk}t_{uk'}-iq_{vk}t_{uk'}-it_{vk}q_{uk'})(P_{s+1}\cdots P_mQ_mP_{m+1}P_1Q_1\cdots P_r)_{vk}(P_{r+1}\cdots P_s)_{uk'}\\
&\quad= 2\sum_{u, v, k, k'}(q_{vk}-it_{vk})(q_{uk'}-it_{uk'})(P_{s+1}\cdots P_mQ_mP_{m+1}P_1Q_1\cdots P_r)_{vk}(P_{r+1}\cdots P_s)_{uk'}\\
&\quad= 2\tr(P_{s+1}\cdots P_mQ_mP_{m+1}P_1Q_1\cdots P_rQ_r)\tr(P_{r+1}\cdots P_sQ_s)\\
&\quad= 2\tr(P_1Q_1\cdots P_rQ_rP_{s+1}\cdots P_mQ_mP_{m+1})\tr(P_{r+1}\cdots P_sQ_s)\, .
\end{align*}
Again by the definition of a positive splitting, the sum of above quantities over all such $(r,s)$ pairs gives
\[
2\sum_{\substack{x,y\in B_1\\ x\ne y}}W_{\times^1_{x,y} l_1} W_{\times^2_{x,y} l_1}\, .
\]
This completes the proof of the lemma.
\end{proof}
\begin{lmm}\label{lem5}
\begin{align*}
\sum_{u,v,k,k'} \biggl( t_{uk}t_{vk'} \mpar{W_{l_1}}{q_{uk}}{q_{vk'}}+q_{uk}q_{vk'} \mpar{W_{l_1}}{t_{uk}}{t_{vk'}}-2t_{uk}q_{vk'} \mpar{W_{l_1}}{q_{uk}}{t_{vk'}} \biggr)&=(m-t_1^2)W_{l_1}\, .
\end{align*}
\end{lmm}
\begin{proof}
First note that since
\[
\sum_{u,v,k,k'}t_{uk}q_{vk'} \mpar{W_{l_1}}{q_{uk}}{t_{vk'}}=\sum_{u,v,k,k'}t_{vk'}q_{uk} \mpar{W_{l_1}}{q_{vk'}}{t_{uk}}
\]
we can write
\[
2\sum_{u,v,k,k'}t_{uk}q_{vk'} \mpar{W_{l_1}}{q_{uk}}{t_{vk'}}=\sum_{u,v,k,k'}\biggl(t_{uk}q_{vk'} \mpar{W_{l_1}}{q_{uk}}{t_{vk'}}+q_{uk}t_{vk'} \mpar{W_{l_1}}{t_{uk}}{q_{vk'}}\biggr)\, .
\]
So, if $Q_r=Q_s=Q_e$, then by \eqref{pqeq0}, \eqref{pqeq1} and similar equations we get
\begin{align*}
2\sum_{u,v,k,k'}& (t_{uk}t_{vk'}-q_{uk}q_{vk'}-it_{uk}q_{vk'}-iq_{uk}t_{vk'})(P_{s+1}\cdots P_mQ_mP_{m+1}P_1Q_1\cdots P_r)_{k'u}(P_{r+1}\cdots P_s)_{kv}\\
&=-2(q_{uk}+it_{uk})(q_{vk'}+it_{vk'})(P_{s+1}\cdots P_mQ_mP_{m+1}P_1Q_1\cdots P_r)_{k'u}(P_{r+1}\cdots P_s)_{kv}\\
&=-2\tr((P_{s+1}\cdots P_mQ_mP_{m+1}P_1Q_1\cdots P_r)Q_r(P_{r+1}\cdots P_s)Q_s)\\
&=-2W_{l_1}\,.
\end{align*}
Similarly, by \eqref{pqeq0}, \eqref{pqeq2} and similar equations if $Q_r=Q_s=Q_e^*$, then 
\begin{align*}
2\sum_{u,v,k,k'}& (t_{uk}t_{vk'}-q_{uk}q_{vk'}+it_{uk}q_{vk'}+iq_{uk}t_{vk'})(P_{s+1}\cdots P_mQ_mP_{m+1}P_1Q_1\cdots P_r)_{vk}(P_{r+1}\cdots P_s)_{uk'}\\
&=-2(q_{uk}-it_{uk})(q_{vk'}-it_{vk'})(P_{s+1}\cdots P_mQ_mP_{m+1}P_1Q_1\cdots P_r)_{vk}(P_{r+1}\cdots P_s)_{uk'}\\
&=-2\tr((P_{s+1}\cdots P_mQ_mP_{m+1}P_1Q_1\cdots P_r)Q_r(P_{r+1}\cdots P_s)Q_s)\\
&=-2W_{l_1}
\end{align*}
On the other hand, if $Q_r=Q_e$ and $Q_s=Q_e^*$, then 
\begin{align}
&\tr\biggl(P_1Q_1\cdots P_r \fpar{Q_r}{q_{uk}^e} P_{r+1}\cdots P_s \fpar{Q_s}{q_{vk'}^e}P_{s+1}\cdots P_mQ_mP_{m+1}\biggr) \nonumber\\
&= \tr(P_1Q_1\cdots P_r b_u b_k^T P_{r+1}\cdots P_s b_{k'} b_{v}^TP_{s+1}\cdots P_mQ_mP_{m+1})\nonumber\\
&= (b_{v}^TP_{s+1}\cdots P_mQ_mP_{m+1}P_1Q_1\cdots P_r b_u)( b_k^T P_{r+1}\cdots P_s b_{k'})\nonumber\\
&= (P_{s+1}\cdots P_mQ_mP_{m+1}P_1Q_1\cdots P_r)_{vu}(P_{r+1}\cdots P_s)_{kk'}\, . \nonumber
\end{align}
Hence we get
\begin{align*}
2\sum_{u,v,k,k'}& (t_{uk}t_{vk'}+q_{uk}q_{vk'}+it_{uk}q_{vk'}-iq_{uk}t_{vk'})(P_{s+1}\cdots P_mQ_mP_{m+1}P_1Q_1\cdots P_r)_{vu}(P_{r+1}\cdots P_s)_{kk'}\\
&=2(q_{uk}+it_{uk})(q_{vk'}-it_{vk'})(P_{s+1}\cdots P_mQ_mP_{m+1}P_1Q_1\cdots P_r)_{vu}(P_{r+1}\cdots P_s)_{kk'}\\
&=2\tr((P_{s+1}\cdots P_mQ_mP_{m+1}P_1Q_1\cdots P_r)Q_r(P_{r+1}\cdots P_s)Q_s)\\
&=2W_{l_1}\, .
\end{align*}
Similarly, if $Q_r=Q_e^*$ and $Q_s=Q_e$, then
\begin{align}
&\tr\biggl(P_1Q_1\cdots P_r \fpar{Q_r}{q_{uk}^e} P_{r+1}\cdots P_s \fpar{Q_s}{q_{vk'}^e}P_{s+1}\cdots P_mQ_mP_{m+1}\biggr) \nonumber\\
&= \tr(P_1Q_1\cdots P_r b_k b_u^T P_{r+1}\cdots P_s b_{v} b_{k'}^TP_{s+1}\cdots P_mQ_mP_{m+1})\nonumber\\
&= (b_{k'}^TP_{s+1}\cdots P_mQ_mP_{m+1}P_1Q_1\cdots P_r b_{k})( b_{u}^T P_{r+1}\cdots P_s b_{v})\nonumber\\
&= (P_{s+1}\cdots P_mQ_mP_{m+1}P_1Q_1\cdots P_r)_{k'k}(P_{r+1}\cdots P_s)_{uv}\, . \nonumber
\end{align}
So we get
\begin{align*}
2\sum_{u,v,k,k'}& (t_{uk}t_{vk'}+q_{uk}q_{vk'}-it_{uk}q_{vk'}+iq_{uk}t_{vk'})(P_{s+1}\cdots P_mQ_mP_{m+1}P_1Q_1\cdots P_r)_{k'k}(P_{r+1}\cdots P_s)_{uv}\\
&=2(q_{uk}-it_{uk})(q_{vk'}+it_{vk'})(P_{s+1}\cdots P_mQ_mP_{m+1}P_1Q_1\cdots P_r)_{k'k}(P_{r+1}\cdots P_s)_{uv}\\
&=2\tr((P_{s+1}\cdots P_mQ_mP_{m+1}P_1Q_1\cdots P_r)Q_r(P_{r+1}\cdots P_s)Q_s)\\
&=2W_{l_1}\, .
\end{align*} 
By combining all these we get
\begin{align*}
\sum_{u,v,k,k'} &\biggl( t_{uk}t_{vk'} \mpar{W_{l_1}}{q_{uk}}{q_{vk'}}+q_{uk}q_{vk'} \mpar{W_{l_1}}{t_{uk}}{t_{vk'}}-t_{uk}q_{vk'} \mpar{W_{l_1}}{q_{uk}}{t_{vk'}}-q_{uk}t_{vk'} \mpar{W_{l_1}}{t_{uk}}{q_{vk'}} \biggr)\\
&\qquad=\left(-2\dbinom{|A_1|}{2}-2\dbinom{|B_1|}{2}+2|A_1||B_1|\right)W_{l_1}=(m-t_1^2)W_{l_1}\, .
\end{align*}
\end{proof}
\begin{lmm}\label{lem6}
Let $l'$ be a non-null loop such that all points within distance $1$ of $l'$ belong to $\Lambda$. Let $A'$ be the set of locations in $l'$ where $e$ occurs, and let $B'$ be the set of locations in $l'$ where $e^{-1}$ occurs. Let $C' = A'\cup B'$ and assume that $C'$ is nonempty. Then 
\[
\sum_{u,k} \biggl(\fpar{W_{l_1}}{q^e_{uk}} \fpar{W_{l'}}{q^e_{uk}}+\fpar{W_{l_1}}{t^e_{uk}} \fpar{W_{l'}}{t^e_{uk}}\biggr) = 2\sum_{x\in A_1,\, y\in B'} W_{l_1\ominus_{x,y} l'}+2\sum_{x\in B_1,\, y\in A'} W_{l_1\ominus_{x,y} l'}\, .
\]
\end{lmm}
\begin{proof}
Let
\[
W_{l'} = \tr(P_1'Q_1'P_2'Q_2'\cdots P_{m'}'Q_{m'}' P_{m'+1}')\,  
\]
where $m'$ is the size of $C'$, each of $Q_1',\ldots, Q_{m'}'$ is either $Q_e$ or $Q_e^*$, and each of $P_1',\ldots, P_{m'+1}'$ is a product of $Q_{e'}$ where $e'$ is neither $e$ nor $e^{-1}$. Then
\begin{align*}
&\sum_{u,k}\fpar{W_{l_1}}{q_{uk}^e}\fpar{W_{l'}}{q_{uk'}^e}\\
&= \sum_{\substack{1\le r\le m\\1\le s\le m'}}\sum_{u,k}\tr\biggl(P_1Q_1\cdots P_r \fpar{Q_r}{q_{uk}^e} P_{r+1}\cdots Q_m P_{m+1}\biggr)\tr\biggl( P_1'Q_1'\cdots P_s'\fpar{Q_s'}{q_{uk}^e}P_{s+1}'\cdots Q'_{m'}P'_{m'+1}\biggr)
\end{align*}
and
\begin{align*}
&\sum_{u,k}\fpar{W_{l_1}}{t_{uk}^e}\fpar{W_{l'}}{t_{uk'}^e}\\
&= \sum_{\substack{1\le r\le m\\1\le s\le m'}}\sum_{u,k}\tr\biggl(P_1Q_1\cdots P_r \fpar{Q_r}{t_{uk}^e} P_{r+1}\cdots Q_m P_{m+1}\biggr)\tr\biggl( P_1'Q_1'\cdots P_s'\fpar{Q_s'}{t_{uk}^e}P_{s+1}'\cdots Q'_{m'}P'_{m'+1}\biggr)\, .
\end{align*}
Observe that by Lemma \ref{lem1} the pairs $(Q_r, Q'_s)=(Q_e, Q_e)$ and $(Q_r, Q'_s)=(Q_e^*, Q_e^*)$ do not contribute anything to the sum on the left hand side of the equation in the statement of the lemma. Moreover, for any other pair
\begin{align*}
\sum_{u,k}\tr\biggl(&P_1Q_1\cdots P_r\fpar{Q_r}{q_{uk}^e} P_{r+1}\cdots Q_m P_{m+1}\biggr)\tr\biggl( P_1'Q_1'\cdots P_s'\fpar{Q_s'}{q_{uk}^e}P_{s+1}'\cdots Q'_{m'}P'_{m'+1}\biggr)\\
&= \sum_{u,k}\tr\biggl(P_1Q_1\cdots P_r \fpar{Q_r}{t_{uk}^e} P_{r+1}\cdots Q_m P_{m+1}\biggr)\tr\biggl( P_1'Q_1'\cdots P_s'\fpar{Q_s'}{t_{uk}^e}P_{s+1}'\cdots Q'_{m'}P'_{m'+1}\biggr) \, .
\end{align*}
Next, note that if $(Q_r, Q'_s)= (Q_e, Q_e^*)$, then we get
\begin{align*}
&2\sum_{u,k} \tr\biggl(P_1Q_1\cdots P_r \fpar{Q_r}{q_{uk}^e} P_{r+1}\cdots Q_m P_{m+1}\biggr)\tr\biggl( P_1'Q_1'\cdots P_s'\fpar{Q_s'}{q_{uk}^e}P_{s+1}'\cdots Q'_{m'}P'_{m'+1}\biggr)\\
&=2\sum_{u,k} (P_{r+1}\cdots Q_m P_{m+1}P_1Q_1\cdots P_r )_{ku}( P_{s+1}'\cdots Q'_{m'}P'_{m'+1}P_1'Q_1'\cdots P_s')_{uk}\\
&= 2\tr(P_{r+1}\cdots Q_m P_{m+1}P_1Q_1\cdots P_rP_{s+1}'\cdots Q'_{m'}P'_{m'+1}P_1'Q_1'\cdots P_s')\\
&= 2\tr(P_1Q_1\cdots P_rP_{s+1}'\cdots Q'_{m'}P'_{m'+1}P_1'Q_1'\cdots P_s'P_{r+1}\cdots Q_m P_{m+1})\, . 
\end{align*}
This expression equals $2W_{l_1\ominus_{x,y} l'}$ where $x\in A_1$ and $y\in B'$. Similarly, if $(Q_r, Q'_s) = (Q_e^*, Q_e)$, then we obtain
\begin{align*}
&2\sum_{u,k} \tr\biggl(P_1Q_1\cdots P_r \fpar{Q_r}{q_{uk}^e} P_{r+1}\cdots Q_m P_{m+1}\biggr)\tr\biggl( P_1'Q_1'\cdots P_s'\fpar{Q_s'}{q_{uk}^e}P_{s+1}'\cdots Q'_{m'}P'_{m'+1}\biggr)\\
&=2\sum_{u,k} (P_{r+1}\cdots Q_m P_{m+1}P_1Q_1\cdots P_r )_{uk}( P_{s+1}'\cdots Q'_{m'}P'_{m'+1}P_1'Q_1'\cdots P_s')_{ku}\\
&=2\tr(P_{r+1}\cdots Q_m P_{m+1}P_1Q_1\cdots P_rP_{s+1}'\cdots Q'_{m'}P'_{m'+1}P_1'Q_1'\cdots P_s')\\
&=2 \tr(P_1Q_1\cdots P_rP_{s+1}'\cdots Q'_{m'}P'_{m'+1}P_1'Q_1'\cdots P_s'P_{r+1}\cdots Q_m P_{m+1})\, .
\end{align*}
This is equal to $2W_{l_1\ominus_{x,y} l'}$ where $x\in B_1$ and $y\in A'$. 
\end{proof}
\begin{lmm}\label{lem7}
Let $l'$, $A'$ and $B'$ be as in the previous lemma. Then 
\begin{align*}
\sum_{u,v,k,k'} \biggl((q_{vk}q_{uk'}&-t_{vk}t_{uk'})\left(\fpar{W_{l_1}}{q_{uk}}\fpar{W_{l'}}{q_{vk'}}- \fpar{W_{l_1}}{t_{uk}}\fpar{W_{l'}}{t_{vk'}}\right)\\
&\qquad\qquad\qquad\qquad+(q_{vk}t_{uk'}+t_{vk}q_{uk'})\left(\fpar{W_{l_1}}{q_{uk}}\fpar{W_{l'}}{t_{vk'}}+ \fpar{W_{l_1}}{t_{uk}}\fpar{W_{l'}}{q_{vk'}} \right)\biggr)\\
&=2\sum_{x\in A_1\,, y\in A'}W_{l_1\oplus_{x,y} l'}+2\sum_{x\in B_1\,, y\in B'}W_{l_1\oplus_{x,y} l'}\, .
\end{align*}
\end{lmm}
\begin{proof}
We continue using notation in the proof of Lemma \ref{lem6}. First, note that
\begin{align}
&\fpar{W_{l_1}}{q_{uk}^e}\fpar{W_{l'}}{q_{vk'}^e} \notag\\
&= \sum_{\substack{1\le r\le m\\1\le s\le m'}}\tr\biggl(P_1Q_1\cdots P_r \fpar{Q_r}{q_{uk}^e} P_{r+1}\cdots Q_m P_{m+1}\biggr)\tr\biggl(P_1'Q_1'\cdots P_s'\fpar{Q_s'}{q_{vk'}^e}P_{s+1}'\cdots Q'_{m'}P'_{m'+1}\biggr)\, , \label{pqeq3}
\end{align}
\begin{align*}
&\fpar{W_{l_1}}{t_{uk}^e}\fpar{W_{l'}}{t_{vk'}^e}\\
&= \sum_{\substack{1\le r\le m\\1\le s\le m'}}\tr\biggl(P_1Q_1\cdots P_r \fpar{Q_r}{t_{uk}^e} P_{r+1}\cdots Q_m P_{m+1}\biggr)\tr\biggl(P_1'Q_1'\cdots P_s'\fpar{Q_s'}{t_{vk'}^e}P_{s+1}'\cdots Q'_{m'}P'_{m'+1}\biggr)\, ,
\end{align*}
\begin{align*}
&\fpar{W_{l_1}}{q_{uk}^e}\fpar{W_{l'}}{t_{vk'}^e}\\
&= \sum_{\substack{1\le r\le m\\1\le s\le m'}}\tr\biggl(P_1Q_1\cdots P_r \fpar{Q_r}{q_{uk}^e} P_{r+1}\cdots Q_m P_{m+1}\biggr)\tr\biggl(P_1'Q_1'\cdots P_s'\fpar{Q_s'}{t_{vk'}^e}P_{s+1}'\cdots Q'_{m'}P'_{m'+1}\biggr)\, ,
\end{align*}
and
\begin{align*}
&\fpar{W_{l_1}}{t_{uk}^e}\fpar{W_{l'}}{q_{vk'}^e}\\
&= \sum_{\substack{1\le r\le m\\1\le s\le m'}}\tr\biggl(P_1Q_1\cdots P_r \fpar{Q_r}{t_{uk}^e} P_{r+1}\cdots Q_m P_{m+1}\biggr)\tr\biggl(P_1'Q_1'\cdots P_s'\fpar{Q_s'}{q_{vk'}^e}P_{s+1}'\cdots Q'_{m'}P'_{m'+1}\biggr)\, .
\end{align*}
It is easy to see that if $(Q_r, Q_s') = (Q_e, Q_e^*)$ or $(Q_r, Q_s') = (Q_e^*, Q_e)$, then by Lemma \ref{lem1} the corresponding terms cancel each other. This means these pairs do not contribute anything to our sum. Now, suppose $Q_r=Q_s'=Q_e$. Note that
\begin{align}
&\tr\biggl(P_1Q_1\cdots P_r \fpar{Q_r}{q_{uk}^e} P_{r+1}\cdots Q_m P_{m+1}\biggr)\tr\biggl( P_1'Q_1'\cdots P_s'\fpar{Q_s'}{q_{vk'}^e}P_{s+1}'\cdots Q'_{m'}P'_{m'+1}\biggr) \nonumber\\
&= \tr(P_1Q_1\cdots P_r b_u b_k^T P_{r+1}\cdots Q_m P_{m+1})\tr( P_1'Q_1'\cdots P_s'b_v b_{k'}^TP_{s+1}'\cdots Q'_{m'}P'_{m'+1}) \nonumber\\
&= (b_k^T P_{r+1}\cdots Q_m P_{m+1}P_1Q_1\cdots P_r b_u )( b_{k'}^TP_{s+1}'\cdots Q'_{m'}P'_{m'+1}P_1'Q_1'\cdots P_s'b_v )\nonumber \\
&= (P_{r+1}\cdots Q_m P_{m+1}P_1Q_1\cdots P_r )_{ku}( P_{s+1}'\cdots Q'_{m'}P'_{m'+1}P_1'Q_1'\cdots P_s')_{k'v}\, .\label{pqeq4}
\end{align}
By writing similar expressions for the other terms we get 
\begin{align*}
&2\sum_{u,v,k,k'} (q_{vk}^eq_{uk'}^e-t_{vk}^et_{uk'}^e+iq_{vk}^et_{uk'}^e+it_{vk}^eq_{uk'}^e)\\
&\qquad \qquad\qquad (P_{r+1}\cdots Q_m P_{m+1}P_1Q_1\cdots P_r )_{ku}(P_{s+1}'\cdots Q'_{m'}P'_{m'+1}P_1'Q_1'\cdots P_s')_{k'v}\\
&=2\sum_{u,v,k,k'} (q_{vk}^e+it_{vk}^e)(q_{uk'}^e+it_{uk'}^e)\\
&\qquad \qquad\qquad (P_{r+1}\cdots Q_m P_{m+1}P_1Q_1\cdots P_r )_{ku}(P_{s+1}'\cdots Q'_{m'}P'_{m'+1}P_1'Q_1'\cdots P_s')_{k'v}\\
&= 2\tr(P_{r+1}\cdots Q_m P_{m+1}P_1Q_1\cdots P_rQ_r P_{s+1}'\cdots Q'_{m'}P'_{m'+1}P_1'Q_1'\cdots P_s'Q_s')\\
&= 2\tr(P_1Q_1\cdots P_rQ_r P_{s+1}'\cdots Q'_{m'}P'_{m'+1}P_1'Q_1'\cdots P_s'Q_s'P_{r+1}\cdots Q_m P_{m+1})\, .
\end{align*}
By definition this is equal to $2W_{l_1\oplus_{x,y} l'}$ where $x\in A$ and $y\in A'$. Next, suppose that $Q_r = Q_s' = Q_e^*$. Then 
\begin{align}
&\tr\biggl(P_1Q_1\cdots P_r \fpar{Q_r}{q_{uk}^e} P_{r+1}\cdots Q_m P_{m+1}\biggr)\tr\biggl( P_1'Q_1'\cdots P_s'\fpar{Q_s'}{q_{vk'}^e}P_{s+1}'\cdots Q'_{m'}P'_{m'+1}\biggr) \nonumber\\
&= \tr(P_1Q_1\cdots P_r b_k b_u^T P_{r+1}\cdots Q_m P_{m+1})\tr( P_1'Q_1'\cdots P_s'b_{k'} b_{v}^TP_{s+1}'\cdots Q'_{m'}P'_{m'+1}) \nonumber\\
&= (b_u^T P_{r+1}\cdots Q_m P_{m+1}P_1Q_1\cdots P_r b_k)( b_{v}^TP_{s+1}'\cdots Q'_{m'}P'_{m'+1}P_1'Q_1'\cdots P_s'b_{k'} ) \nonumber\\
&= (P_{r+1}\cdots Q_m P_{m+1}P_1Q_1\cdots P_r )_{uk}( P_{s+1}'\cdots Q'_{m'}P'_{m'+1}P_1'Q_1'\cdots P_s')_{vk'}\, .\label{pqeq5}
\end{align}
Therefore by writing similar expression for the other terms we get, 
\begin{align*}
&2\sum_{u,v,k,k'} (q_{vk}^eq_{uk'}^e-t_{vk}^et_{uk'}^e-iq_{vk}^et_{uk'}^e-it_{vk}^eq_{uk'}^e)\\
&\qquad \qquad\qquad (P_{r+1}\cdots Q_m P_{m+1}P_1Q_1\cdots P_r )_{uk}( P_{s+1}'\cdots Q'_{m'}P'_{m'+1}P_1'Q_1'\cdots P_s')_{vk'}\\
&=2\sum_{u,v,k,k'} (q_{vk}^e-it_{vk}^e)(q_{uk'}^e-it_{uk'}^e)\\
&\qquad \qquad\qquad (P_{r+1}\cdots Q_m P_{m+1}P_1Q_1\cdots P_r )_{uk}( P_{s+1}'\cdots Q'_{m'}P'_{m'+1}P_1'Q_1'\cdots P_s')_{vk'}\\
&= 2\tr(P_{r+1}\cdots Q_m P_{m+1}P_1Q_1\cdots P_rQ_r P_{s+1}'\cdots Q'_{m'}P'_{m'+1}P_1'Q_1'\cdots P_s'Q_s')\\
&= 2\tr(P_1Q_1\cdots P_rQ_r P_{s+1}'\cdots Q'_{m'}P'_{m'+1}P_1'Q_1'\cdots P_s'Q_s'P_{r+1}\cdots Q_m P_{m+1})\, .
\end{align*}
This corresponds to the term $2W_{l_1\oplus_{x,y} l'}$ with $x\in B_1$ and $y\in B'$. 
\end{proof}
\begin{lmm}\label{lem8}
Let $l'$, $A'$ and $B'$ be as in the previous lemmas. Then 
\begin{align*}
\sum_{u,v,k,k'} \biggl(t_{uk}t_{vk'} \fpar{W_{l_1}}{q_{uk}}\fpar{W_{l'}}{q_{vk'}}+q_{uk}q_{vk'} \fpar{W_{l_1}}{t_{uk}}\fpar{W_{l'}}{t_{vk'}}&-t_{uk}q_{vk'} \fpar{W_{l_1}}{q_{uk}}\fpar{W_{l'}}{t_{vk'}}-q_{uk}t_{vk'} \fpar{W_{l_1}}{t_{uk}}\fpar{W_{l'}}{q_{vk'}} \biggr)\\
&=-(A_1-B_1)(A'-B')W_{l_1}W_{l'}\, .
\end{align*}
\end{lmm}
\begin{proof}
We continue using the notations from the previous lemmas. If $(Q_r, Q_s')=(Q_e, Q_e)$, then by \eqref{pqeq3}, \eqref{pqeq4} and similar equations we get
\begin{align*}
&\sum_{u,k} (t_{uk}t_{vk'}-q_{uk}q_{vk'}-it_{uk}q_{vk'}-iq_{uk}t_{vk'})\\
&\qquad\qquad\qquad(P_{r+1}\cdots Q_m P_{m+1}P_1Q_1\cdots P_r )_{ku}( P_{s+1}'\cdots Q'_{m'}P'_{m'+1}P_1'Q_1'\cdots P_s')_{k'v}\\
&=-\sum_{u,k} (q_{uk}+it_{uk})(q_{vk'}+it_{vk'})\\
&\qquad\qquad\qquad(P_{r+1}\cdots Q_m P_{m+1}P_1Q_1\cdots P_r )_{ku}( P_{s+1}'\cdots Q'_{m'}P'_{m'+1}P_1'Q_1'\cdots P_s')_{k'v}\\
&=-\tr(P_{r+1}\cdots Q_m P_{m+1}P_1Q_1\cdots P_r Q_r)\tr(P_{s+1}'\cdots Q'_{m'}P'_{m'+1}P_1'Q_1'\cdots P_s'Q_s')\\
&=-\tr(P_1Q_1\cdots P_r Q_rP_{r+1}\cdots Q_m P_{m+1})\tr(P_1'Q_1'\cdots P_s'Q_s'P_{s+1}'\cdots Q'_{m'}P'_{m'+1})\\
&=-W_{l_1}W_{l'}\, .
\end{align*}
If $(Q_r, Q_s')=(Q^*_e, Q^*_e)$, then by equations \eqref{pqeq3}, \eqref{pqeq5} and similar equations we obtain
\begin{align*}
&\sum_{u,k} (t_{uk}t_{vk'}-q_{uk}q_{vk'}+it_{uk}q_{vk'}+iq_{uk}t_{vk'})\\
&\qquad\qquad\qquad(P_{r+1}\cdots Q_m P_{m+1}P_1Q_1\cdots P_r )_{uk}( P_{s+1}'\cdots Q'_{m'}P'_{m'+1}P_1'Q_1'\cdots P_s')_{vk'}\\
&=-\sum_{u,k} (q_{uk}-it_{uk})(q_{vk'}-it_{vk'})\\
&\qquad\qquad\qquad(P_{r+1}\cdots Q_m P_{m+1}P_1Q_1\cdots P_r )_{uk}( P_{s+1}'\cdots Q'_{m'}P'_{m'+1}P_1'Q_1'\cdots P_s')_{vk'}\\
&=-\tr(P_{r+1}\cdots Q_m P_{m+1}P_1Q_1\cdots P_r Q_r)\tr(P_{s+1}'\cdots Q'_{m'}P'_{m'+1}P_1'Q_1'\cdots P_s'Q_s')\\
&=-\tr(P_1Q_1\cdots P_r Q_rP_{r+1}\cdots Q_m P_{m+1})\tr(P_1'Q_1'\cdots P_s'Q_s'P_{s+1}'\cdots Q'_{m'}P'_{m'+1})\\
&=-W_{l_1}W_{l'}\, .
\end{align*}
Using the same ideas we get that if $(Q_r, Q_s')=(Q_e, Q^*_e)$, then
\begin{align*}
&\sum_{u,k} (t_{uk}t_{vk'}+q_{uk}q_{vk'}+it_{uk}q_{vk'}-iq_{uk}t_{vk'})\\
&\qquad\qquad\qquad(P_{r+1}\cdots Q_m P_{m+1}P_1Q_1\cdots P_r )_{ku}( P_{s+1}'\cdots Q'_{m'}P'_{m'+1}P_1'Q_1'\cdots P_s')_{vk'}\\
&=\sum_{u,k} (q_{uk}+it_{uk})(q_{vk'}-it_{vk'})\\
&\qquad\qquad\qquad(P_{r+1}\cdots Q_m P_{m+1}P_1Q_1\cdots P_r )_{ku}( P_{s+1}'\cdots Q'_{m'}P'_{m'+1}P_1'Q_1'\cdots P_s')_{vk'}\\
&=\tr(P_{r+1}\cdots Q_m P_{m+1}P_1Q_1\cdots P_r Q_r)\tr(P_{s+1}'\cdots Q'_{m'}P'_{m'+1}P_1'Q_1'\cdots P_s'Q_s')\\
&=\tr(P_1Q_1\cdots P_r Q_rP_{r+1}\cdots Q_m P_{m+1})\tr(P_1'Q_1'\cdots P_s'Q_s'P_{s+1}'\cdots Q'_{m'}P'_{m'+1})\\
&=W_{l_1}W_{l'}\, .
\end{align*}
Finally, if $(Q_r, Q_s')=(Q^*_e, Q_e)$, then
\begin{align*}
&\sum_{u,k} (t_{uk}t_{vk'}+q_{uk}q_{vk'}-it_{uk}q_{vk'}+iq_{uk}t_{vk'})\\
&\qquad\qquad\qquad(P_{r+1}\cdots Q_m P_{m+1}P_1Q_1\cdots P_r )_{uk}( P_{s+1}'\cdots Q'_{m'}P'_{m'+1}P_1'Q_1'\cdots P_s')_{k'v}\\
&=\sum_{u,k} (q_{uk}-it_{uk})(q_{vk'}+it_{vk'})\\
&\qquad\qquad\qquad(P_{r+1}\cdots Q_m P_{m+1}P_1Q_1\cdots P_r )_{uk}( P_{s+1}'\cdots Q'_{m'}P'_{m'+1}P_1'Q_1'\cdots P_s')_{k'v}\\
&=\tr(P_{r+1}\cdots Q_m P_{m+1}P_1Q_1\cdots P_r Q_r)\tr(P_{s+1}'\cdots Q'_{m'}P'_{m'+1}P_1'Q_1'\cdots P_s'Q_s')\\
&=\tr(P_1Q_1\cdots P_r Q_rP_{r+1}\cdots Q_m P_{m+1})\tr(P_1'Q_1'\cdots P_s'Q_s'P_{s+1}'\cdots Q'_{m'}P'_{m'+1})\\
&=W_{l_1}W_{l'}\, .
\end{align*}
By combining all of above equations we get
\[
(-A_1A'-B_1B'+A_1B'+B_1A')W_{l_1}W_{l'}=-(A_1-B_1)(A'-B')W_{l_1}W_{l'}\, .
\]
\end{proof}
Next, we will combine above lemmas and prove Theorem \ref{mastern}. 
\begin{proof}[Proof of Theorem \ref{mastern}]
For an edge $e$ let $\mr^+(e)$ and $\mr^-(e)$ respectively be the sets of positively and negatively oriented plaquettes passing through $e$. Then $\mr^+(e)\cup \mr^+(e^{-1})=\cp^+(e)$ and $\mr^+(e)\cup \mr^-(e)=\cp(e)$. Let $Q = (Q_{e'})_{e'\in E^+_\Lambda}$ denote a collection of independent Haar-distributed random $SU(N)$ matrices and let $\cp^+_{\Lambda}$ denote the set of all positively oriented plaquettes contained in $\Lambda$.
Define
\[
f(Q) := W_{l_1}\, , 
\]
and
\[
g(Q):= Z_{\Lambda, N, \beta}^{-1} W_{l_2}W_{l_3}\cdots W_{l_n} \exp\biggl(N\beta\sum_{p\in \cp^+_\Lambda} \Re W_p\biggr)\, .
\]
Note that since $2\Re W_p=W_p+W_{p^{-1}}$ we can write
\[
g(Q)=Z_{\Lambda, N, \beta}^{-1}W_{l_2}W_{l_3}\cdots W_{l_n} \exp\biggl(\frac{N\beta}{2}\sum_{p\in \cp^+_{\Lambda}}(W_p+W_{p^{-1}})\biggr)\, .
\]
By Lemma \ref{lem2},
\begin{align*}
\ee\biggl[\sum_{u,k} \biggl(q_{uk}^e \fpar{f}{q_{uk}^e}+t_{uk}^e \fpar{f}{t_{uk}^e}\biggr) g \biggr] &= m\smallavg{W_{l_1}W_{l_2}\cdots W_{l_n}}\, .
\end{align*}
By Lemma \ref{lem3}, 
\begin{align*}
\ee\biggl[\sum_{i,k}\biggl(\spar{f}{{q^e_{uk}}} +\spar{f}{{t^e_{uk}}}\biggr)g\biggr] &= 2\sum_{x\in A_1,\, y\in B_1} \smallavg{W_{\times^1_{x,y} l_1} W_{\times^2_{x,y} l_1}W_{l_2}\cdots W_{l_n}}\\
&\qquad +2\sum_{x\in B_1,\, y\in A_1} \smallavg{W_{\times^1_{x,y} l_1} W_{\times^2_{x,y} l_1}W_{l_2}\cdots W_{l_n}}\, .
\end{align*}
By Lemma \ref{lem4}, 
\begin{align*}
\ee\biggl[\sum_{u,v,k,k'} \biggl((q_{vk}q_{uk'}&-t_{vk}t_{uk'}) \biggl(\mpar{f}{q_{uk}}{q_{vk'}}-\mpar{f}{t_{uk}}{t_{vk'}}\biggr) +2(q_{vk}t_{uk'}+t_{vk}q_{uk'})\mpar{f}{q_{uk}}{t_{vk'}}\biggr)g\biggr]\\
&=  2\sum_{\substack{x,y\in A_1\\ x\ne y}}\smallavg{W_{\times^1_{x,y} l_1} W_{\times^2_{x,y} l_1}W_{l_2}\cdots W_{l_n}} + 2\sum_{\substack{x,y\in B_1\\ x\ne y}}\smallavg{W_{\times^1_{x,y} l_1} W_{\times^2_{x,y} l_1}W_{l_2}\cdots W_{l_n}}\, .
\end{align*}
By Lemma \ref{lem5}
\begin{align*}
\ee\biggl[\sum_{u,v,k,k'} \biggl( t_{uk}t_{vk'} \mpar{f}{q_{uk}}{q_{vk'}}+q_{uk}q_{vk'} \mpar{f}{t_{uk}}{t_{vk'}}-2t_{uk}q_{vk'} &\mpar{f}{q_{uk}}{t_{vk'}} \biggr)g\biggr]\\
&=(m-t_1^2)\smallavg{W_{l_1}W_{l_2}\cdots W_{l_n}}\, .
\end{align*}
Next, note that
\begin{align}
\fpar{g}{q_{uk}^e} &= \sum_{r=2}^n Z_{\Lambda, N,\beta}^{-1} W_{l_2}\cdots W_{l_{r-1}} \fpar{W_{l_r}}{q_{uk}^e} W_{l_{r+1}}\cdots W_{l_n}\exp\biggl(\frac{N\beta}{2}\sum_{p\in \cp^+_{\Lambda}}(W_p+W_{p^{-1}})\biggr)\nonumber\\
&\quad + \sum_{p\in \cp^+(e)} Z_{\Lambda, N,\beta}^{-1} W_{l_2}\cdots W_{l_n}\frac{N\beta}{2} \fpar{W_p}{q_{uk}^e}\exp\biggl(\frac{N\beta}{2}\sum_{p'\in \cp^+_{\Lambda}}(W_{p'}+W_{p'^{-1}})\biggr)\notag\\
&\quad+ \sum_{p\in \cp^+(e)} Z_{\Lambda, N,\beta}^{-1} W_{l_2}\cdots W_{l_n}\frac{N\beta}{2} \fpar{W_{p^{-1}}}{q_{uk}^e}\exp\biggl(\frac{N\beta}{2}\sum_{p'\in \cp^+_{\Lambda}}(W_{p'}+W_{p'^{-1}})\biggr)\, .\label{gq}
\end{align}
A similar equality also holds for $\fpar{g}{t_{uk}^e}$. So, by Lemma \ref{lem6} and these identities
\begin{align*}
\ee\biggl[&\sum_{u,k} \biggl(\fpar{f}{q^e_{uk}} \fpar{g}{q^e_{uk}}+\fpar{f}{t^e_{uk}} \fpar{g}{t^e_{uk}}\biggr)\biggr]\\
&= 2\sum_{r=2}^n \sum_{x\in A_1, \, y\in B_r} \bigavg{W_{l_1\ominus_{x,y} l_r} \prod_{\substack{2\le t\le n\\t\ne r}} W_{l_t}}+2\sum_{r=2}^n \sum_{x\in B_1, \, y\in A_r} \bigavg{W_{l_1\ominus_{x,y} l_r} \prod_{\substack{2\le t\le n\\t\ne r}} W_{l_t}}\\
&\quad + N\beta \sum_{p\in \mr^+(e^{-1})} \sum_{x\in A_1} \smallavg{W_{l_1\ominus_x p}W_{l_2}\cdots W_{l_n}}+N\beta \sum_{p\in \mr^+(e)} \sum_{x\in B_1} \smallavg{W_{l_1\ominus_x p}W_{l_2}\cdots W_{l_n}}\\
&\quad + N\beta \sum_{p\in \mr^+(e)} \sum_{x\in A_1} \smallavg{W_{l_1\ominus_x p^{-1}}W_{l_2}\cdots W_{l_n}}+N\beta \sum_{p\in \mr^+(e^{-1})} \sum_{x\in B_1} \smallavg{W_{l_1\ominus_x p^{-1}}W_{l_2}\cdots W_{l_n}}\, .
\end{align*}
Note that $W_{l_1\ominus_x p^{-1}}=W_{l_1\ominus_x p}$ for any $x\in A_1$ and plaquette $p\in \mr^+(e)$, and for any $x\in B_1$ and $p\in \mr^+(e^{-1})$. Therefore the above equality can be written as 
\begin{align*}
\ee\biggl[&\sum_{u,k} \biggl(\fpar{f}{q^e_{uk}} \fpar{g}{q^e_{uk}}+\fpar{f}{t^e_{uk}} \fpar{g}{t^e_{uk}}\biggr)\biggr]\\
&= 2\sum_{r=2}^n \sum_{x\in A_1, \, y\in B_r} \bigavg{W_{l_1\ominus_{x,y} l_r} \prod_{\substack{2\le t\le n\\t\ne r}} W_{l_t}}+2\sum_{r=2}^n \sum_{x\in B_1, \, y\in A_r} \bigavg{W_{l_1\ominus_{x,y} l_r} \prod_{\substack{2\le t\le n\\t\ne r}} W_{l_t}}\\
&\qquad + N\beta \sum_{p\in \cp^+(e)} \sum_{x\in C_1} \smallavg{W_{l_1\ominus_x p}W_{l_2}\cdots W_{l_n}}\, .
\end{align*}
Similarly, by Lemma \ref{lem7} and equation \eqref{gq}
\begin{align*}
&\ee\biggl[\sum_{u,v,k,k'} \biggl((q_{vk}q_{uk'}-t_{vk}t_{uk'})\left(\fpar{f}{q_{uk}}\fpar{g}{q_{vk'}}- \fpar{f}{t_{uk}}\fpar{g}{t_{vk'}}\right)\\
&\qquad\qquad\qquad+(q_{vk}t_{uk'}+t_{vk}q_{uk'})\left(\fpar{f}{q_{uk}}\fpar{g}{t_{vk'}}+ \fpar{f}{t_{uk}}\fpar{g}{q_{vk'}} \right)\biggr)\biggr]\\
&=2\sum_{r=2}^n \sum_{x\in A_1, \, y\in A_r} \bigavg{W_{l_1\oplus_{x,y} l_r} \prod_{\substack{2\le t\le n\\t\ne r}} W_{l_t}}+2\sum_{r=2}^n \sum_{x\in B_1, \, y\in B_r} \bigavg{W_{l_1\oplus_{x,y} l_r} \prod_{\substack{2\le t\le n\\t\ne r}} W_{l_t}} 
\\
&\quad+ N\beta \sum_{p\in \mr^+(e)} \sum_{x\in A_1} \smallavg{W_{l_1\oplus_x p}W_{l_2}\cdots W_{l_n}}+ N\beta \sum_{p\in \mr^+(e^{-1})} \sum_{x\in B_1} \smallavg{W_{l_1\oplus_x p}W_{l_2}\cdots W_{l_n}}\\
&\quad+ N\beta \sum_{p\in \mr^+(e^{-1})} \sum_{x\in A_1} \smallavg{W_{l_1\oplus_x p^{-1}}W_{l_2}\cdots W_{l_n}}+ N\beta \sum_{p\in \mr^+(e)} \sum_{x\in B_1} \smallavg{W_{l_1\oplus_x p^{-1}}W_{l_2}\cdots W_{l_n}}\, .
\end{align*}
Since $W_{l_1\oplus_x p^{-1}}=W_{l_1\oplus_x p}$ for any $x\in A_1$ and $p\in\mr^+(e^{-1})$, and for any $x\in B_1$ and $p\in \mr^+(e)$, the above equality can be written as
\begin{align*}
&\ee\biggl[\sum_{u,v,k,k'} \biggl((q_{vk}q_{uk'}-t_{vk}t_{uk'})\left(\fpar{f}{q_{uk}}\fpar{g}{q_{vk'}}- \fpar{f}{t_{uk}}\fpar{g}{t_{vk'}}\right)\\
&\qquad\qquad\qquad+(q_{vk}t_{uk'}+t_{vk}q_{uk'})\left(\fpar{f}{q_{uk}}\fpar{g}{t_{vk'}}+ \fpar{f}{t_{uk}}\fpar{g}{q_{vk'}} \right)\biggr)\biggr]\\
&=2\sum_{r=2}^n \sum_{x\in A_1, \, y\in A_r} \bigavg{W_{l_1\oplus_{x,y} l_r} \prod_{\substack{2\le t\le n\\t\ne r}} W_{l_t}}+2\sum_{r=2}^n \sum_{x\in B_1, \, y\in B_r} \bigavg{W_{l_1\oplus_{x,y} l_r} \prod_{\substack{2\le t\le n\\t\ne r}} W_{l_t}} 
\\
&\qquad + N\beta \sum_{p\in \cp^+(e)} \sum_{x\in C_1} \smallavg{W_{l_1\oplus_x p}W_{l_2}\cdots W_{l_n}}\, .
\end{align*}
By Lemma \ref{lem8} and equation \eqref{gq}
\begin{align*}
&\ee\biggl[\sum_{u,v,k,k'} \biggl(t_{uk}t_{vk'}\fpar{f}{q_{uk}}\fpar{g}{q_{vk'}}+q_{uk}q_{vk'} \fpar{f}{t_{uk}}\fpar{g}{t_{vk'}}-t_{uk}q_{vk'} \fpar{f}{q_{uk}}\fpar{g}{t_{vk'}}-q_{uk}t_{vk'} \fpar{f}{t_{uk}}\fpar{g}{q_{vk'}} \biggr)\biggr]\\
&=-t_1\biggl(\sum_{r=2}^nt_r\biggr)\smallavg{W_{l_1}W_{l_2}\cdots W_{l_n}}\\
&\quad -\frac{N\beta t_1}{2}\sum_{p\in \mr^+(e)} \smallavg{W_{l_1}W_{p}W_{l_2}\cdots W_{l_n}}+\frac{N\beta t_1}{2} \sum_{p\in \mr^+(e^{-1})} \smallavg{W_{l_1}W_{p}W_{l_2}\cdots W_{l_n}}\\
&\quad +\frac{N\beta t_1}{2}\sum_{p\in \mr^+(e)} \smallavg{W_{l_1}W_{p^{-1}}W_{l_2}\cdots W_{l_n}}-\frac{N\beta t_1}{2} \sum_{p\in \mr^+(e^{-1})} \smallavg{W_{l_1}W_{p^{-1}}W_{l_2}\cdots W_{l_n}}\, .
\end{align*}
Note that if $p \in \mr^+(e)$, then $p^{-1}\in \mr^-(e^{-1})$, and if $p \in \mr^+(e^{-1})$, then $p^{-1}\in \mr^-(e)$. Thus
\begin{align*}
&\ee\biggl[\sum_{u,v,k,k'} \biggl(t_{uk}t_{vk'}\fpar{f}{q_{uk}}\fpar{g}{q_{vk'}}+q_{uk}q_{vk'} \fpar{f}{t_{uk}}\fpar{g}{t_{vk'}}-t_{uk}q_{vk'} \fpar{f}{q_{uk}}\fpar{g}{t_{vk'}}-q_{uk}t_{vk'} \fpar{f}{t_{uk}}\fpar{g}{q_{vk'}} \biggr)\biggr]\\
&=-t_1\biggl(\sum_{r=2}^nt_r\biggr)\smallavg{W_{l_1}W_{l_2}\cdots W_{l_n}}\\
&\quad -\frac{N\beta t_1}{2}\sum_{p\in \cp(e)} \smallavg{W_{l_1}W_pW_{l_2}\cdots W_{l_n}}+\frac{N\beta t_1}{2} \sum_{p\in \cp(e^{-1})} \smallavg{W_{l_1}W_{p}W_{l_2}\cdots W_{l_n}}\\
&=-t_1\biggl(\sum_{r=2}^nt_r\biggr)\smallavg{W_{l_1}W_{l_2}\cdots W_{l_n}}\\
&\quad-\frac{N\beta}{2}\sum_{x\in A_1}\sum_{p\in \cp(e)} \smallavg{W_{l_1}W_pW_{l_2}\cdots W_{l_n}}-\frac{N\beta}{2}\sum_{x\in B_1} \sum_{p\in \cp(e^{-1})} \smallavg{W_{l_1}W_{p}W_{l_2}\cdots W_{l_n}}\\
&\quad+\frac{N\beta}{2}\sum_{x\in B_1}\sum_{p\in \cp(e)} \smallavg{W_{l_1}W_pW_{l_2}\cdots W_{l_n}}+\frac{N\beta}{2}\sum_{x\in A_1} \sum_{p\in \cp(e^{-1})} \smallavg{W_{l_1}W_{p}W_{l_2}\cdots W_{l_n}}\, .
\end{align*}
To finish the proof first condition on $(Q_{e'})_{e'\neq e}$ and apply Theorem \ref{sdthm} with above calculations, and then take unconditional expectation.
\end{proof}
Now we are ready to prove Theorem~\ref{symfinmaster}. Recall that if $(l_1, \ldots, l_n)$ is the minimal representation of $s$, then
\[
\phi_N(s)=\frac{\smallavg{W_{l_1}\cdots W_{l_n}}}{N^n}\, .
\]
\begin{proof}[Proof of Theorem~\ref{symfinmaster}]
Recall that equation \eqref{unsym} was obtained by choosing the first edge $e$ of the first loop $l_1$ in minimal representation $(l_1, \ldots,l_n)$ of $s$. Since 
\[
\phi_N(l_{\pi(1)}, \ldots, l_{\pi(n)})=\phi_N(l_1, \ldots, l_n) 
\]
for any permutation $\pi$ of $\{1, 2, \ldots,n \}$, we can write the same equation for any $l_k$ and any edge $e$. By doing so and dividing both sides by $N^{n+1}$ we get
\begin{align}
\biggl(m_k(e)-\frac{t_k(e)t(e)}{N^2}\biggr)\phi_N(s) &=\text{split }l_k+\text{merge }l_k+\text{deform }l_k+\text{expand }l_k \label{c1}
\end{align}
where the split $l_k$ term is given by
\begin{align*}
&\sum_{x\in A_k, \, y\in B_k}\phi_N(l_1, \ldots, l_{k-1}, \times_{x,y}^1 l_k, \times_{x,y}^2 l_k, l_{k+1},\ldots, l_n) \\
&\quad  + \sum_{x\in B_k, \, y\in A_k}\phi_N(l_1, \ldots, l_{k-1}, \times_{x,y}^1 l_k, \times_{x,y}^2 l_k, l_{k+1},\ldots, l_n)\\
&\quad  - \sum_{\substack{x,y\in A_k\\ x\ne y}} \phi_N(l_1, \ldots, l_{k-1}, \times_{x,y}^1 l_k, \times_{x,y}^2 l_k, l_{k+1},\ldots, l_n)\\
&\quad  - \sum_{\substack{x,y\in B_k\\ x\ne y}} \phi_N(l_1, \ldots, l_{k-1}, \times_{x,y}^1 l_k, \times_{x,y}^2 l_k, l_{k+1},\ldots, l_n)\, ,
\end{align*}
the merge $l_k$ term is given by
\begin{align*}
&\frac{1}{N^2}\sum_{1\leq r<k}\sum_{\substack{x\in A_k, \, y\in B_r\\\textup{or }\\x\in B_k, \, y\in A_r}}\phi_N(l_1, \ldots, l_{r-1}, l_{r+1}, \ldots, l_{k-1}, l_k \ominus_{x,y} l_r, l_{k+1}, \ldots, l_n)\\
&+\frac{1}{N^2}\sum_{k<r\leq n}\sum_{\substack{x\in A_k, \, y\in B_r\\\textup{or }\\x\in B_k, \, y\in A_r}}\phi_N(l_1, \ldots, l_{k-1}, l_k \ominus_{x,y} l_r, l_{k+1}, \ldots, l_{r-1}, l_{r+1}, \ldots, l_n)\\
&-\frac{1}{N^2}\sum_{1\leq r<k}\sum_{\substack{x\in A_k, \, y\in A_r\\\textup{or }\\x\in B_k, \, y\in B_r}}\phi_N(l_1, \ldots, l_{r-1}, l_{r+1}, \ldots, l_{k-1}, l_k \oplus_{x,y} l_r, l_{k+1}, \ldots, l_n)\\
&-\frac{1}{N^2}\sum_{k<r\leq n}\sum_{\substack{x\in A_k, \, y\in A_r\\\textup{or }\\x\in B_k, \, y\in B_r}}\phi_N(l_1, \ldots, l_{k-1}, l_k \oplus_{x,y} l_r, l_{k+1}, \ldots, l_{r-1}, l_{r+1}, \ldots, l_n)\, ,
\end{align*}
the deform $l_k$ term is given by
\begin{align*}
&\frac{\beta}{2} \sum_{p\in \cp^+(e)}\sum_{x\in C_k}\phi_N(l_1, \ldots, l_{k-1}, l_k\ominus_{x}p, l_{k+1}, \ldots,l_n)\\
&- \frac{\beta}{2}\sum_{p\in \cp^+(e)}\sum_{x\in C_k}\phi_N(l_1, \ldots, l_{k-1}, l_k\oplus_{x}p, l_{k+1}, \ldots,l_n) \, ,
\end{align*}
and the expansion term is given by
\begin{align*}
&\frac{\beta}{2}\sum_{x\in A_k}\sum_{p\in \cp(e)}\phi_N(l_1,\ldots,l_{k-1}, l_k, p, l_{k+1}, \ldots, l_n)\\
&\quad + \frac{\beta}{2}\sum_{x\in B_k} \sum_{p\in \cp(e^{-1})}\phi_N(l_1,\ldots,l_{k-1}, l_k, p, l_{k+1}, \ldots, l_n)\\
&\quad - \frac{\beta}{2}\sum_{x\in B_k}\sum_{p\in \cp(e)}\phi_N(l_1,\ldots,l_{k-1}, l_k, p, l_{k+1}, \ldots, l_n)\\
&\quad - \frac{\beta}{2}\sum_{x\in A_k} \sum_{p\in \cp(e^{-1})}\phi_N(l_1,\ldots,l_{k-1}, l_k, p, l_{k+1}, \ldots, l_n) \, .
\end{align*}
Note that if the loop $l_k$ contains neither $e$ nor $e^{-1}$ then $m_k=0$ and $t_k=0$, and in this case both sides of equation \eqref{c1} are zero. Next, observe that
\[
\sum_{e\in E^+}m_{k}(e)=|l_k|\, .
\]
Then by writing equation \eqref{c1} for each element of $E^+$ and adding them side by side we get
\begin{align}
\biggl(|l_k|-\frac{1}{N^2}\sum_{e\in E^+}t_k(e)t(e)\biggr)\phi_N(s)&=\frac{1}{N^2}\sum_{s'\in \fst_k^-(s)}\phi_N(s')-\frac{1}{N^2}\sum_{s'\in \fst^+_k(s)}\phi_N(s')\notag\\
&\quad + \sum_{s'\in \fs^-_k(s)}\phi_N(s')-\sum_{s'\in \fs^+_k(s)}\phi_N(s')\notag\\
&\quad + \frac{\beta}{2}\sum_{s'\in \fd^-_k(s)}\phi_N(s')-\frac{\beta}{2}\sum_{s'\in \fd^+_k(s)}\phi_N(s')\notag \\
&\quad + \frac{\beta}{2}\sum_{s'\in \fex^-_k(s)}\phi_N(s')-\frac{\beta}{2}\sum_{s'\in \fex^+_k(s)}\phi_N(s')\, .\label{c2}
\end{align}
Since
\[
\sum_{k=1}^{n}|l_k|=|s| \, ,
\]
and
\[
\sum_{k=1}^{n}\sum_{e\in E^+}t_k(e)t(e)=\sum_{e\in E^+}t(e)\sum_{k=1}^{n}t_k(e)=\sum_{e \in E^+}t(e)^2=\ell(s)\, ,
\]
we finish the proof by summing both sides of equation \eqref{c2} over all $k=1, \ldots, n$.
\end{proof}
\section{Unsymmetrized master loop equation for $f_0$}\label{unsymf0}
Recall the definition of $A_r$, $B_r$, $m_r$ and $t_r$ from Section \ref{string}. For any loop function $h: \cs \to \cc$, let
\begin{align*}
\sum_{\text{merge}\,s}h&=\sum_{r=2}^n\sum_{x\in A_1, \, y\in B_r} h(l_1 \ominus_{x,y} l_r, \ldots, l_n)+\sum_{r=2}^n\sum_{x\in B_1, \, y\in A_r} h(l_1 \ominus_{x,y} l_r, \ldots, l_n)\\
&\quad - \sum_{r=2}^n\sum_{x\in A_1, \, y\in A_r} h(l_1 \oplus_{x,y} l_r, \ldots, l_n) - \sum_{r=2}^n\sum_{x\in B_1, \, y\in B_r}h(l_1 \oplus_{x,y} l_r, \ldots, l_n)\, .
\end{align*}
Similarly, let
\begin{align*}
\sum_{\text{split}\,s}h&=\sum_{x\in A_1, \, y\in B_1} h(\times_{x,y}^1 l_1, \times_{x,y}^2 l_1, \ldots, l_n) + \sum_{x\in B_1, \, y\in A_1} h(\times_{x,y}^1 l_1, \times_{x,y}^2 l_1, \ldots, l_n)\\
&\quad - \sum_{\substack{x,y\in A_1\\ x\ne y}} h(\times_{x,y}^1 l_1, \times_{x,y}^2 l_1, \ldots, l_n) - \sum_{\substack{x,y\in B_1\\ x\ne y}} h(\times_{x,y}^1 l_1, \times_{x,y}^2 l_1, \ldots, l_n)\, ,
\end{align*}
\begin{align*}
\sum_{\text{deform}\,s}h&=\sum_{p\in \cp^+(e)}\sum_{x\in C_1}h(l_1 \ominus_{x}p, \ldots, l_n)-  \sum_{p\in \cp^+(e)}\sum_{x\in C_1}h(l_1 \oplus_{x}p, \ldots, l_n)\, ,
\end{align*}
and
\begin{align*}
\sum_{\text{expand}\,s}h&=\sum_{x\in A_1}\sum_{p\in \cp(e)}h(l_1, p, l_2, \ldots, l_n)  + \sum_{x\in B_1} \sum_{p\in \cp(e^{-1})}h(l_1, p, l_2, \ldots, l_n)\\
&-\quad \sum_{x\in B_1}\sum_{p\in \cp(e)}h(l_1, p, l_2, \ldots, l_n)-\sum_{x\in A_1}\sum_{p\in \cp(e^{-1})}h(l_1, p, l_2, \ldots, l_n)  \, .
\end{align*}
Let 
\[
\sideset{}{^+}\sum\limits_{\text{merge}\,s}, \ \ \sideset{}{^+}\sum\limits_{\text{split}\,s}, \  \ \sideset{}{^+}\sum\limits_{\text{deform}\,s}, \ \ \sideset{}{^+}\sum\limits_{\text{expand}\,s}
\]
denote the sums as above with minus signs between terms replaced with plus signs. By dividing both sides of equation \eqref{unsym} by $N^{n+1}$ we can write it with above notation as
\begin{align}
\biggl(m-\frac{t_1t}{N^2}\biggr)\phi_N(s) &=\frac{1}{N^2}\sum_{\text{merge } s} \phi_N+\sum_{\text{split }s}\phi_N+\frac{\beta}{2}\sum_{\text{deform }s}\phi_N+\frac{\beta}{2}\sum_{\text{expand }s}\phi_N \label{finmastereq}
\end{align}
The goal of this section is to prove the following theorem.
\begin{thm}\label{f0}
There exists $\beta_0(d,0)>0$, depending only on dimension $d$, such that for any $|\beta|\leq \beta_0(d,0)$  and loop sequence $s$, the limit
\begin{align}
f_0(s)&=\lim_{N\to \infty} \phi_N(s) \label{limf0}
\end{align}
exists. Moreover, $f_0$ satisfies the equation
\begin{align}
mf_0(s)&=\sum_{\text{split }s}f_0+\frac{\beta}{2}\sum_{\text{deform }s}f_0+\frac{\beta}{2}\sum_{\text{expand }s}f_0 \label{f0mastereq}
\end{align}
\end{thm}
\begin{proof}[Proof of Theorem~\ref{f0}]
Note that since $|W_{l}|\leq N$ we have $|\phi_N(s)|\leq 1$. So by standard diagonal method the limit in \eqref{limf0} exists along some subsequence. By letting $N \to \infty$ in equation \eqref{finmastereq} we see that any subsequential limit $f_0$ satisfies equation \eqref{f0mastereq}. So it is enough to prove that equation \eqref{f0mastereq} has a unique solution for all sufficiently small $\beta$.
\begin{lmm}\label{uniquef0}
There exists $\beta_1(d, 0)>0$, depending only on dimension $d$, such that for all $|\beta|<\beta_1(d, 0)$, there exists a unique loop function $f: \cs \to\cc$ such that $(1)$ $f(\emptyset)=1 \,$, $(2)$ $|f(s)|\leq L^{|s|}$ for some $L\geq 1$, and $(3)$ $f$ satisfies equation \eqref{f0mastereq}.  
\end{lmm}
\begin{proof}
Since $\phi_N(\emptyset)=1$, we get $f_0(\emptyset)=1$. So, from this and above discussion we deduce that any subsequencial limit $f_0$ satisfies conditions $(1)$,  $(2)$ and $(3)$. So we just need to prove uniqueness. Let $f$ and $\tilde{f}$ be two such functions. Define $g:\cs \to \cc$ as
\[
g(s):=f(s)-\tilde{f}(s) \, .
\]
Let $\Delta$ be the set of all finite sequences of integers including the null sequence. Let $\Delta^+$ be the subset of $\Delta$ consisting of those sequences with each term $\geq 4$. If $(l_1, \ldots, l_n)$ is the minimal representation of a loop sequence $s$, then let $\delta(s)=(|l_1|, \ldots, |l_n|)$ be the degree sequence of $s$. Since each non-null loop has at least four edges, $\delta(s)\in \Delta^+$ for any non-null $s$. Given 
$\delta=(\delta_1, \ldots, \delta_n)$ define the length of $\delta$ as
\[
|\delta|:=|\delta_1|+\cdots+|\delta_n|\, ,
\]
the size as 
\[
\#\delta:= n \, ,
\]
and the index as 
\[\iota(\delta):=|\delta|-\#\delta \, .\]
If $\delta'=(\delta'_1, \ldots, \delta'_m)$ is another element of $\Delta$, then let $\delta'\leq \delta$ whenever $n=m$ and $\delta'_i\leq \delta_i$ for all $1\leq i \leq n$. For any $\delta \in \Delta^+$ let
\[
D(\delta)=\max_{s\in \cs:\, \delta(s)\leq \delta} |g(s)|,
\]
and extend this definition to $\delta \in \Delta\setminus\Delta^+$ by letting $D(\delta)=0$. For any $\lambda \in (0, 1)$ let
\[
F(\lambda):=\sum_{\delta \in \Delta}D(\delta)\lambda^{\iota(\delta)} \, .
\]
Note that the number of $\delta\in \Delta$ with $|\delta|=r$ and $\#\delta=n$ is at most $\binom{r-1}{n-1}$. Moreover, since
\[
|g(s)|=|f(s)-\tilde{f}(s)|\leq L^{|s|}+L^{|s|}=2L^{|s|} \, ,
\]
we have $D(\delta)\leq 2L^{|\delta|}$. Since $|\delta|\geq 4\#\delta$ for any $\delta \in \Delta^+$
\begin{align*}
F(\lambda)&=\sum_{r=4}^{\infty}\sum_{n=1}^{r}\sum_{\substack{\delta \in \Delta^+:\, |\delta|=r\\\#\delta=n}}D(\delta)\lambda^{\iota(\delta)}\\
&\leq \sum_{r=4}^{\infty}\sum_{n=1}^{r}2L^{r}\lambda^{r-n}\dbinom{r-1}{n-1}\\
&\leq \sum_{r=4}^{\infty}2L^{r}\lambda^{3r/4}2^{r}<\infty
\end{align*}
for all $\lambda< (2L)^{-4/3}$. Note that since both $f$ and $\tilde{f}$ satisfy equation \eqref{f0mastereq} so does $g$. Therefore,
\begin{align}
|g(s)|\leq \frac{1}{m}\sideset{}{^+}\sum_{\text{split }s}|g|+\frac{|\beta|}{2m}\sideset{}{^+}\sum_{\text{deform }s}|g|+\frac{|\beta|}{2m}\sideset{}{^+}\sum_{\text{expand }s}|g| \label{a0}
\end{align}
Now, take any $\delta \in \Delta^+$ and let $s$ be so that $\delta(s)\leq \delta$. Recall that
\begin{align*}
\frac{1}{m}\sideset{}{^+}\sum_{\text{split}\,s}|g|&=\frac{1}{m}\sum_{x\in A_1, \, y\in B_1} |g(\times_{x,y}^1 l_1, \times_{x,y}^2 l_1, \ldots, l_n)|+ \frac{1}{m}\sum_{x\in B_1, \, y\in A_1} |g(\times_{x,y}^1 l_1, \times_{x,y}^2 l_1, \ldots, l_n)|\\
&\quad + \frac{1}{m}\sum_{\substack{x,y\in A_1\\ x\ne y}} |g(\times_{x,y}^1 l_1, \times_{x,y}^2 l_1, \ldots, l_n)| + \frac{1}{m} \sum_{\substack{x,y\in B_1\\ x\ne y}} |g(\times_{x,y}^1 l_1, \times_{x,y}^2 l_1, \ldots, l_n)|\, .
\end{align*}
By Lemma \ref{split2}
\begin{align*}
\frac{1}{m}\sum_{x\in A_1, \, y\in B_1} |g(\times_{x,y}^1 l_1, \times_{x,y}^2 l_1, \ldots, l_n)|&\leq \frac{1}{m}\sum_{x\in A_1, \, y\in B_1}D(\delta_1-|y-x|-1, |y-x|-1, \ldots, \delta_n)\\
&\leq \frac{2}{m} \sum_{x\in A_1}\sum_{k=1}^{\infty} D(\delta_1-k-1, k-1, \ldots, \delta_n)\\
&\leq 2 \sum_{k=1}^{\infty} D(\delta_1-k-1, k-1, \ldots, \delta_n)\, .
\end{align*}
The same inequality also holds if $A_1$ and $B_1$ are swapped. Similarly by Lemma \ref{split1}
\begin{align*}
\frac{1}{m}\sum_{\substack{x,y\in A_1\\ x\ne y}} |g(\times_{x,y}^1 l_1, \times_{x,y}^2 l_1, \ldots, l_n)|&\leq \frac{1}{m}\sum_{\substack{x,y\in A_1\\ x\ne y}}D(\delta_1-|y-x|, |y-x|, \ldots, \delta_n)\\
&\leq \frac{2}{m} \sum_{x\in A_1}\sum_{k=1}^{\infty} D(\delta_1-k, k, \ldots, \delta_n)\\
&\leq 2\sum_{k=1}^{\infty} D(\delta_1-k, k, \ldots, \delta_n)\, .
\end{align*}
The same inequality also holds if we replace $A_1$ with $B_1$. Thus
\begin{align}
\frac{1}{m}\sideset{}{^+}\sum_{\text{split}\,s}|g|&\leq 4\sum_{k=1}^{\infty} D(\delta_1-k-1, k-1, \ldots, \delta_n)+4\sum_{k=1}^{\infty} D(\delta_1-k, k, \ldots, \delta_n)\, . \label{s0}
\end{align}
Next, note that
\begin{align*}
\frac{1}{2m}\sideset{}{^+}\sum_{\text{deform}\,s}|g|&=\frac{1}{2m}\sum_{p\in \cp^+(e)}\sum_{x\in C_1}|g(l_1 \ominus_{x}p, \ldots, l_n)|+ \frac{1}{2m}\sum_{p\in \cp^+(e)}\sum_{x\in C_1}|g(l_1 \ominus_{x}p, \ldots, l_n)|\, .
\end{align*}
By Lemma \ref{deform} $|l_1 \ominus_{x}p|\leq |l_1|+4$. Since a deformation of the loop $l_1$ may also result in a null loop we have
\begin{align*}
|g(l_1 \ominus_{x}p, \ldots, l_n)|&\leq D(\delta_1+4, \ldots, \delta_n)+D(\delta_2, \ldots, \delta_n) \, .
\end{align*}
Since there are $2(d-1)$ positively oriented plaquettes passing through $e$ or $e^{-1}$ we get
\begin{align*}
&\frac{1}{2m}\sum_{p\in \cp^+(e)}\sum_{x\in C_1}|g(l_1 \ominus_{x}p, \ldots, l_n)|\leq d D(\delta_1+4, \ldots, \delta_n)+d D(\delta_2, \ldots, \delta_n)\, .
\end{align*}
The same inequality also holds for the second term in the deformation sum. Therefore,
\begin{align}
\frac{1}{2m}\sideset{}{^+}\sum_{\text{deform}\,s}|g|&\leq 2d D(\delta_1+4, \ldots, \delta_n)+2d D(\delta_2, \ldots, \delta_n)\, . \label{d0}
\end{align}
Finally, note that
\begin{align*}
\frac{1}{2m}\sideset{}{^+}\sum_{\text{expand}\,s}|g|&=\frac{1}{2}\sum_{p\in \cp(e)}|g(l_1, p, l_2, \ldots, l_n)|+ \frac{1}{2}\sum_{p\in \cp(e^{-1})}|g(l_1, p, l_2, \ldots, l_n)| \, .
\end{align*}
Since $|p|=4$ for any plaquette $p$ and $|\cp(e)| = 2(d-1)$
\begin{align*}
\frac{1}{2}\sum_{p\in \cp(e)}|g(l_1, p, l_2, \ldots, l_n)|\leq dD(\delta_1, 4, \ldots, \delta_n) \, .
\end{align*}
The same inequality also holds for the second term in the expansion sum. Hence
\begin{align}
\frac{1}{2m}\sideset{}{^+}\sum_{\text{expand}\,s}|g|&\leq 2dD(\delta_1, 4, \delta_2, \ldots, \delta_n)\, . \label{e0}
\end{align}
By combining inequalities  \eqref{a0}, \eqref{s0}, \eqref{d0} and \eqref{e0} we get
\begin{align*}
|g(s)|&\leq 4\sum_{k=1}^{\infty} D(\delta_1-k-1, k-1, \ldots, \delta_n)+4\sum_{k=1}^{\infty} D(\delta_1-k, k, \ldots, \delta_n)\\
&\quad+2d|\beta| D(\delta_1+4, \ldots, \delta_n)+2d|\beta| D(\delta_2, \ldots, \delta_n)+2d|\beta| D(\delta_1, 4, \delta_2 \ldots, \delta_n)\, .
\end{align*}
Note that this inequality is valid for all $s$ such that $\delta(s)\leq \delta$. Hence by taking the maximum of both sides over all such $s$ we obtain
\begin{align}
D(\delta)&\leq 4\sum_{k=1}^{\infty} D(\delta_1-k-1, k-1, \ldots, \delta_n)+4\sum_{k=1}^{\infty} D(\delta_1-k, k, \ldots, \delta_n)\notag\\
&\quad+2d|\beta| D(\delta_1+4, \ldots, \delta_n)+2d|\beta| D(\delta_2, \ldots, \delta_n)+2d|\beta| D(\delta_1, 4, \delta, \ldots, \delta_n)\, . \notag
\end{align}
Therefore,
\begin{align}
F(\lambda)&\leq \sum_{\delta \in \Delta^+}\lambda^{\iota(\delta)}\biggl(4\sum_{k=1}^{\infty} D(\delta_1-k-1, k-1, \ldots, \delta_n)+4\sum_{k=1}^{\infty} D(\delta_1-k, k, \ldots, \delta_n)\notag\\
&\quad+2d|\beta| D(\delta_1+4, \ldots, \delta_n)+2d|\beta| D(\delta_2, \ldots, \delta_n)+2d|\beta| D(\delta_1, 4, \delta_2, \ldots, \delta_n) \biggr)\,. \label{a1}
\end{align}
For $k\geq 1$ define 
\begin{align*}
\rho_k(\delta_1, \delta_2, \ldots, \delta_n)&:=(\delta_1-k-1, k-1, \delta_2, \ldots, \delta_n)\, ,\\
\sigma_k(\delta_1, \delta_2, \ldots, \delta_n)&:=(\delta_1-k, k, \delta_2, \ldots, \delta_n)\, .
\end{align*}
Also let 
\begin{align*}
\alpha(\delta_1, \delta_2, \ldots, \delta_n)&:=(\delta_1+4, \delta_2, \ldots, \delta_n)\, ,\\
\gamma(\delta_1, \delta_2, \ldots, \delta_n)&:=(\delta_1, 4, \delta_2, \ldots, \delta_n)\, ,\\
\theta(\delta_1, \delta_2, \ldots, \delta_n)&:=(\delta_2, \ldots, \delta_n)
\end{align*}
with $\theta(\emptyset)=\emptyset$. Note that map $\rho_k$ is injective for all $k$. Moreover images of these maps are disjoint since the second component of $\rho_k(\delta)$ is $k-1$. Since $\iota(\delta)=\iota(\rho_k(\delta))+3$, we have
\begin{align}
\sum_{\delta \in \Delta^+}\sum_{k=1}^{\infty} D(\delta_1-k-1, k-1, \delta_2, \ldots, \delta_n) \lambda^{\iota(\delta)}&=\sum_{\delta \in \Delta^+}\sum_{k=1}^{\infty} D(\rho_k(\delta)) \lambda^{\iota(\rho_k(\delta))+3}\notag\\
&\leq \sum_{\delta \in \Delta^+}D(\delta) \lambda^{\delta+3}=\lambda^3F(\lambda)\, . \label{a2}
\end{align}
The map $\sigma_k$ is also injective for all $k$ and images of these maps are also disjoint since the second component of $\sigma_k(\delta)$ is $k$. Since $\iota(\delta)=\iota(\sigma_k(\delta))+1$, we get
\begin{align}
\sum_{\delta \in \Delta^+}\sum_{k=1}^{\infty} D(\delta_1-k, k, \delta_2, \ldots, \delta_n) \lambda^{\iota(\delta)}&=\sum_{\delta \in \Delta^+}\sum_{k=1}^{\infty} D(\sigma_k(\delta)) \lambda^{\iota(\sigma_k(\delta))+1}\notag\\
&\leq \sum_{\delta \in \Delta^+}D(\delta) \lambda^{\delta+1}=\lambda F(\lambda)\, . \label{a3}
\end{align}
Next, note that the maps $\alpha$ and $\gamma$ are also injective, $\iota(\delta)=\iota(\alpha(\delta))-4$ and $\iota(\delta)=\iota(\gamma(\delta))-3$. Thus
\begin{align}
\sum_{\delta \in \Delta^+} D(\delta_1+4, \delta_2, \ldots, \delta_n) \lambda^{\iota(\delta)}&=\sum_{\delta \in \Delta^+} D(\alpha(\delta)) \lambda^{\iota(\alpha(\delta))-4}\notag\\
&\leq \sum_{\delta \in \Delta^+}D(\delta) \lambda^{\delta-4}=\frac{F(\lambda)}{\lambda^{4}} \label{a4}\, ,
\end{align}
and
\begin{align}
\sum_{\delta \in \Delta^+} D(\delta_1, 4, \delta_2, \ldots, \delta_n) \lambda^{\iota(\delta)}&=\sum_{\delta \in \Delta^+} D(\gamma(\delta)) \lambda^{\iota(\gamma(\delta))-3}\notag\\
&\leq \sum_{\delta \in \Delta^+}D(\delta) \lambda^{\delta-3}=\frac{F(\lambda)}{\lambda^{3}}  \label{a5}\, .
\end{align}
Finally, note that $\theta^{-1}(\delta_2, \ldots, \delta_n)\cap \Delta^+=\{(k, \delta_2, \ldots, \delta_n ):  k= 4, 5, \ldots\}$ and
\begin{align}
\sum_{\delta \in \Delta^+} D(\delta_2, \ldots, \delta_n) \lambda^{\iota(\delta)}&=\sum_{k=4}^{\infty}\sum_{\delta \in \Delta^+:\,\delta_1=k} D(\theta(\delta)) \lambda^{\iota(\theta(\delta))+k-1}\notag\\
&\leq \sum_{k=1}^{\infty}\sum_{\delta' \in \Delta^+} D(\delta') \lambda^{\iota(\delta')+k-1}=\frac{F(\lambda)}{1-\lambda} \label{a6}\, .
\end{align}
By combining inequalities \eqref{a2}, \eqref{a3}, \eqref{a4}, \eqref{a5}, \eqref{a6} with \eqref{a1} we obtain
\[
F(\lambda)\leq \biggl(4\lambda^3+4\lambda+\frac{2|\beta|d}{\lambda^4}+\frac{2|\beta|d}{\lambda^3}+\frac{2|\beta|d}{1-\lambda}\biggr)F(\lambda)\, .
\]
If we choose $\lambda$ so that $4\lambda^3+4\lambda<1/4$, then for all $\beta$ such that 
\[
\frac{2|\beta|d}{\lambda^4}+\frac{2|\beta|d}{\lambda^3}+\frac{2|\beta|d}{1-\lambda}<\frac{1}{4}
\]
we get $F(\lambda)\leq  0$ which in turn means that $f(s)=\tilde{f}(s)$ for all $s \in \cs$. Therefore for all sufficiently small $\beta$, depending on $d$, there exists a unique function satisfying conditions $(a)$, $(b)$ and $(c)$.
\end{proof}
This completes the proof of Theorem~\ref{f0}.
\end{proof}
\section{Unsymmetrized master loop equation for $f_{2k}$}
Let $e$, $m$, $t_1$ and $t$ be as in Section \ref{unsymf0}. The goal of this section is to prove the following theorem.
\begin{thm}\label{fk}
There exists sequence of positive numbers $\{\beta_0(d, k)\}_{k\geq 0}$, depending only on dimension $d$, and loop functions $\{f_{2k}\}_{k\geq 0}$ depending on $\beta$ such that for any $|\beta|\leq \beta_0(d, k)$
\begin{align}
\lim_{N\to \infty} N^{2k}\biggl(\phi_N(s)-f_0(s)-\frac{1}{N^2}f_2(s)-\cdots - \frac{1}{N^{2k}}f_{2k}(s)\biggr)=0\, . \label{limfk}
\end{align}
Moreover, for any non-null loop sequence $s$
\begin{align}
mf_{2k}(s)=t_1tf_{2k-2}(s)+\sum_{\text{merge}\,s}f_{2k-2}+\sum_{\text{split}\,s}f_{2k}+\frac{\beta}{2}\sum_{\text{deform}\,s}f_{2k}+\frac{\beta}{2}\sum_{\text{expand}\,s}f_{2k}\label{fkmastereq}
\end{align}
where $f_j(s)$ is assumed to be zero for all $s$ when $j<0$.
\end{thm}
The proof will be by induction on $k$. We already proved the case $k=0$ in Theorem \ref{f0}. Now, fix $k\geq 1$ and assume the claim holds for all $k'<k$. Equation \eqref{finmastereq} will be the main ingredient of the proof. Define sequence of increasing positive real numbers $\{L_{q}\}_{q\geq 0}$ as $L_0=1$ and
\[
L_{q}=(20L_{q-1})^{4/3}
\]
for $q\geq 1$.
\begin{lmm}\label{ineqfk}
Fix $N$ and $\Lambda_N$, and let $s$ be a loop sequence such that any edge with distance $\leq 1$ from $s$ lies inside $\Lambda_N$. Then, for any $0\leq q\leq k$ 
\begin{align}
\left| N^{2q}\biggl(\phi_N(s)-f_0(s)-\frac{1}{N^2}f_2(s)-\cdots - \frac{1}{N^{2q-2}}f_{2q-2}(s) \biggr)\right|\leq L_q^{|s|}\, . \label{ineqfk1}
\end{align}
\end{lmm}
\begin{proof}
We will use induction on $q$. Let $H_{0, N}(s)=\phi_N(s)$ and
\begin{align}
H_{q, N}(s)&=N^{2q}\biggl(\phi_N(s)-f_0(s)-\frac{1}{N^2}f_2(s)-\cdots - \frac{1}{N^{2q-2}}f_{2q-2}(s) \biggr)\, . \label{hk}
\end{align}
for any $q\geq 1$. If $q=0$, then $|H_{0, N}(s)|=|\phi_N(s)|\leq 1$ since $|W_{l}|\leq N$ for any loop $l$. So the claim holds for $q=0$. Since $\phi_N(\emptyset)=f_0(\emptyset)=1$, we have $f_{2k}(\emptyset)=0$ for all $k\geq 1$. So inequality \eqref{hk} also holds for $s=\emptyset$ and $k\geq 1$. From now on we will assume that $s\neq \emptyset$. Next, suppose that the claim is true for all $q'<q$. From equation \eqref{limfk} we get
\begin{align*}
|f_{q'}(s)| = \lim_{N\to \infty}|H_{q', N}(s)| \leq L_{q'}^{|s|} \,.
\end{align*}
Since equation \eqref{fkmastereq} holds for all $k'<k$ we can write
\begin{align}
m\biggl(f_0(s)+\frac{1}{N^2}f_2(s)+\cdots+\frac{1}{N^{2q-2}}f_{2q-2}&(s)\biggr)=\frac{t_1t}{N^2}\biggl(f_0(s)+\frac{1}{N^2}f_2(s)\cdots+\frac{1}{N^{2q-4}}f_{2q-4}(s)\biggr) \notag \\
&\quad + \frac{1}{N^2}\sum_{\text{merge}\,s}\biggl(f_0+\frac{1}{N^2}f_2+\cdots+\frac{1}{N^{2q-4}}f_{2q-4}\biggr) \notag \\
&\quad + \sum_{\text{split}\,s}\biggl(f_0+\frac{1}{N^2}f_2+\cdots+\frac{1}{N^{2q-2}}f_{2q-2}\biggr) \notag \\
&\quad + \frac{\beta}{2}\sum_{\text{deform}\,s}\biggl(f_0+\frac{1}{N^2}f_2+\cdots+\frac{1}{N^{2q-2}}f_{2q-2}\biggr) \notag \\
&\quad + \frac{\beta}{2}\sum_{\text{expand}\,s}\biggl(f_0+\frac{1}{N^2}f_2+\cdots+\frac{1}{N^{2q-2}}f_{2q-2}\biggr)\, . \label{sum}
\end{align}
On the other hand, by equation \eqref{finmastereq}
\begin{align}
m\phi_N(s) &=\frac{t_1t}{N^2}\phi_N(s)+\frac{1}{N^2}\sum_{\text{merge}\,s} \phi_N+\sum_{\text{split}\,s}\phi_N+\frac{\beta}{2}\sum_{\text{deform}\,s}\phi_N+\frac{\beta}{2}\sum_{\text{expand}\,s}\phi_N\, . \label{finmastereq1}
\end{align}
Now, subtract equation \eqref{sum} from equation \eqref{finmastereq1} and multiply both sides by $N^{2q}$ to get 
\begin{align}
mH_{q, N}(s) &=t_1tH_{q-1, N}(s)+\sum_{\text{merge}\, s} H_{q-1,N}+\sum_{\text{split}\,s}H_{q, N}+\frac{\beta}{2}\sum_{\text{deform}\,s}H_{q, N}+\frac{\beta}{2}\sum_{\text{expand}\,s}H_{q, N} \, . \label{equationhk}
\end{align}
By Lemma \ref{merger} $|s'|\leq |s|$ for any $s'\in\fst(s)$. So, by induction hypothesis
\[
H_{q-1,N}(s') \leq L_{q-1}^{|s'|}\leq L_{q-1}^{|s|}\, .
\]
Note that if $(l_1, \ldots, l_n)$ is the minimal representation of a loop sequence $s$, then for each $2\leq i\leq n$ the loop $l_i$ can be merged to the loop $l_1$ at locations $e$ and $e^{-1}$ at most in $m|l_i|$ different ways. So there are at most $m|s|$ mergers of $l_1$ with other loop components of $s$. Hence
\[
\sum_{\text{merge}\,s} H_{q-1,N} \leq m|s|L_{q-1}^{|s|}\,.
\]
Since $|t_1|\leq m$ and $|t|\leq |s|$, by equation \eqref{equationhk}, induction hypothesis and above inequality 
\begin{align}
|H_{q, N}(s)|&\leq 2|s|L_{q-1}^{|s|}+\frac{1}{m}\sideset{}{^+}\sum_{\text{split}\,s}|H_{q, N}|+\frac{\beta}{2m}\sideset{}{^+}\sum_{\text{deform}\,s}|H_{q, N}|+\frac{\beta}{2m}\sideset{}{^+}\sum_{\text{expand}\,s}|H_{q, N}| \, . \label{g0}
\end{align}
Let $\Delta$ and $\Delta^+$ be as in the proof of Lemma~\ref{uniquef0}. For any $\delta \in\Delta^+$ let
\[
D(\delta)=\max_{s\in \cs: \delta(s)\leq \delta}|H_{q, N}(s)|\, .
\]
and let $D(\delta)=0$ for all $\delta\in \Delta\setminus\Delta^+$. For $\lambda\in (0, 1)$ define
\[
F(\lambda)=\sum_{\delta\in \Delta^+}\lambda^{\iota(\delta)}D(\delta)=\sum_{\delta\in \Delta}\lambda^{\iota(\delta)}D(\delta)\, .
\] 
First we show that $F(\lambda)$ is finite for all sufficiently small $\lambda$. By induction hypothesis
\[
|H_{q, N}(s)|=|N^2H_{q-1}(s)-N^2f_{2q-2}(s)|\leq N^2|H_{q-1, N}(s)|+N^2|f_{2q-2}(s)|\leq 2N^2L_{q-1}^{|s|}\, .
\]
Therefore
\[
|D(\delta)| \leq 2N^2L_{q-1}^{|\delta|}\, .
\]
Since the number of $\delta \in \Delta^+$ such that $|\delta|=r$ and $\#\delta=n$ is bounded by $\dbinom{r-1}{n-1}$ and $|\delta|\geq 4\#\delta$ for any such $\delta$, we get
\begin{align}
F(\lambda)&=\sum_{r=4}^{\infty}\sum_{n=1}^{r}\sum_{\substack{\delta \in \Delta^+:\, |\delta|=r\\\#\delta=n}}D(\delta)\lambda^{\iota(\delta)}\notag \\
&\leq 2N^2\sum_{r=4}^{\infty}\sum_{n=1}^{r}L_{q-1}^{r}\lambda^{r-n}\dbinom{r-1}{n-1}\notag\\
&\leq 2N^2\sum_{r=4}^{\infty}2L_{q-1}^{r}\lambda^{3r/4}2^{r}<\infty \label{g10}
\end{align}
for all $\lambda< (2L_{q-1})^{-4/3}$. Next, note that
\begin{align*}
\frac{1}{m}\sideset{}{^+}\sum_{\text{split}\,s}|H_{q, N}|&=\frac{1}{m}\sum_{x\in A_1, \, y\in B_1} |H_{q, N}(\times_{x,y}^1 l_1, \times_{x,y}^2 l_1, \ldots, l_n)|+ \frac{1}{m}\sum_{x\in B_1, \, y\in A_1} |H_{q, N}(\times_{x,y}^1 l_1, \times_{x,y}^2 l_1, \ldots, l_n)|\\
&\quad + \frac{1}{m}\sum_{\substack{x,y\in A_1\\ x\ne y}} |H_{q, N}(\times_{x,y}^1 l_1, \times_{x,y}^2 l_1, \ldots, l_n)| + \frac{1}{m} \sum_{\substack{x,y\in B_1\\ x\ne y}} |H_{q, N}(\times_{x,y}^1 l_1, \times_{x,y}^2 l_1, \ldots, l_n)|\, .
\end{align*}
By Lemma \ref{split2}
\begin{align*}
\frac{1}{m}\sum_{x\in A_1, \, y\in B_1} |H_{q, N}(\times_{x,y}^1 l_1, \times_{x,y}^2 l_1, \ldots, l_n)|&\leq \frac{1}{m}\sum_{x\in A_1, \, y\in B_1}D(\delta_1-|y-x|-1, |y-x|-1, \ldots, \delta_n)\\
&\leq \frac{2}{m} \sum_{x\in A_1}\sum_{k=1}^{\infty} D(\delta_1-k-1, k-1, \ldots, \delta_n)\\
&\leq 2 \sum_{k=1}^{\infty} D(\delta_1-k-1, k-1, \ldots, \delta_n)\, .
\end{align*}
The same inequality still holds if we swap $A_1$ and $B_1$. Similarly, by Lemma \ref{split1}
\begin{align*}
\frac{1}{m}\sum_{\substack{x,y\in A_1\\ x\ne y}} |H_{q, N}(\times_{x,y}^1 l_1, \times_{x,y}^2 l_1, \ldots, l_n)|&\leq \frac{1}{m}\sum_{\substack{x,y\in A_1\\ x\ne y}}D(\delta_1-|y-x|, |y-x|, \ldots, \delta_n)\\
&\leq \frac{2}{m} \sum_{x\in A_1}\sum_{k=1}^{\infty} D(\delta_1-k, k, \ldots, \delta_n)\\
&\leq 2\sum_{k=1}^{\infty} D(\delta_1-k, k, \ldots, \delta_n)\, .
\end{align*}
The same inequality is still valid if $A_1$ is replaced by $B_1$. Therefore,
\begin{align}
\frac{1}{m}\sideset{}{^+}\sum_{\text{split}\,s}|H_{q, N}|&\leq 4\sum_{k=1}^{\infty} D(\delta_1-k-1, k-1, \ldots, \delta_n)+4\sum_{k=1}^{\infty} D(\delta_1-k, k, \ldots, \delta_n)\, . \label{g1}
\end{align}
Next, recall that
\begin{align*}
\frac{1}{2m}\sideset{}{^+}\sum_{\text{deform}\,s}|H_{q, N}|&=\frac{1}{2m}\sum_{p\in \cp^+(e)}\sum_{x\in C_1}|H_{q, N}(l_1 \ominus_{x}p, \ldots, l_n)|+ \frac{1}{2m}\sum_{p\in \cp^+(e)}\sum_{x\in C_1}|H_{q, N}(l_1 \oplus_{x}p, \ldots, l_n)|\, .
\end{align*}
Note that by Lemma \ref{deform} $|l_1 \ominus_{x}p|\leq |l_1|+4$. Moreover, a deformation of the loop $l_1$ can also yield a null loop. Therefore,
\begin{align*}
|H_{q, N}(l_1 \ominus_{x}p, \ldots, l_n)|&\leq D(\delta_1+4, \ldots, \delta_n)+D(\delta_2, \ldots, \delta_n)\, .
\end{align*}
Since there are less than $2d$ positively oriented plaquettes passing through an edge or its inverse
\begin{align*}
&\frac{1}{2m}\sum_{p\in \cp^+(e)}\sum_{x\in C_1}|H_{q, N}(l_1 \ominus_{x}p, \ldots, l_n)|\leq d D(\delta_1+4, \ldots, \delta_n)+d D(\delta_2, \ldots, \delta_n)\, .
\end{align*}
The same inequality also holds for the sum over positive deformations. Therefore
\begin{align}
\frac{1}{2m}\sideset{}{^+}\sum_{\text{deform}\,s}|H_{q, N}|&\leq 2d D(\delta_1+4, \ldots, \delta_n)+2d D(\delta_2, \ldots, \delta_n)\, . \label{g2}
\end{align}
Recall that
\begin{align*}
\frac{1}{2m}\sideset{}{^+}\sum_{\text{expand}\,s}|H_{q, N}|&=\frac{1}{2}\sum_{p\in \cp(e)}|H_{q, N}(l_1, p, l_2, \ldots, l_n)|+ \frac{1}{2}\sum_{p\in \cp(e^{-1})}|H_{q, N}(l_1, p, l_2, \ldots, l_n)| \, .
\end{align*}
Since $|p|=4$ for any plaquette $p$ and $|\cp(e)|=2(d-1)$
\begin{align*}
\frac{1}{2}\sum_{p\in \cp(e)}|H_{q, N}(l_1, p, l_2, \ldots, l_n)|\leq dD(\delta_1, 4, \ldots, \delta_n)\, .
\end{align*}
The above inequality is also valid for the second term in the expansion sum. So
\begin{align}
\frac{1}{2m}\sideset{}{^+}\sum_{\text{expand}\,s}|H_{q, N}|&\leq 2dD(\delta_1, 4, \delta_2, \ldots, \delta_n)\, . \label{g3}
\end{align}
By combining inequalities  \eqref{g0}, \eqref{g1}, \eqref{g2} and \eqref{g3} we get
\begin{align*}
|H_{q, N}(s)|&\leq 2|s|L_{q-1}^{|s|}+4\sum_{k=1}^{\infty} D(\delta_1-k-1, k-1, \ldots, \delta_n)+4\sum_{k=1}^{\infty} D(\delta_1-k, k, \ldots, \delta_n)\\
&\quad+2d|\beta| D(\delta_1+4, \ldots, \delta_n)+2d|\beta| D(\delta_2, \ldots, \delta_n)+2d|\beta| D(\delta_1, 4, \delta_2 \ldots, \delta_n)\, .
\end{align*}
By taking maximum of both sides over all $s$ such that $\delta(s)\leq \delta$ we obtain
\begin{align}
D(\delta)&\leq 2|\delta|L_{q-1}^{|\delta|}+4\sum_{k=1}^{\infty} D(\delta_1-k-1, k-1, \ldots, \delta_n)+4\sum_{k=1}^{\infty} D(\delta_1-k, k, \ldots, \delta_n)\notag\\
&\quad+2d|\beta| D(\delta_1+4, \ldots, \delta_n)+2d|\beta| D(\delta_2, \ldots, \delta_n)+2d|\beta| D(\delta_1, 4, \delta, \ldots, \delta_n)\, . \notag
\end{align}
Therefore
\begin{align}
F(\lambda)\leq 2\sum_{\delta\in \Delta^+}\lambda^{\iota(\delta)}|\delta|L_{q-1}^{|\delta|}+\sum_{\delta\in \Delta^+}\lambda^{\iota(\delta)}&\biggl(4\sum_{k=1}^{\infty} D(\delta_1-k-1, k-1, \ldots, \delta_n)\notag\\
&+4\sum_{k=1}^{\infty} D(\delta_1-k, k, \ldots, \delta_n)\notag+2d|\beta| D(\delta_1+4, \ldots, \delta_n)\notag\\
&+2d|\beta| D(\delta_2, \ldots, \delta_n)+2d|\beta| D(\delta_1, 4, \delta_2, \ldots, \delta_n) \biggr)\,. \label{g4}
\end{align}
Proceeding as in the derivation of inequality \eqref{g10}, for $\lambda<(4L_{q-1})^{-4/3}$ we have 
\begin{align}
\sum_{\delta\in \Delta^+}\lambda^{\iota(\delta)}|\delta|L_{q-1}^{|\delta|}&=\sum_{r=4}^{\infty}\sum_{n=1}^{r}\sum_{\substack{\delta \in \Delta^+:\, |\delta|=r\\\#\delta=n}}\lambda^{r-n}rL_{q-1}^r\notag \\
&\leq \sum_{r=4}^{\infty}\sum_{n=1}^{r}\lambda^{r-n}rL_{q-1}^{r}\dbinom{r-1}{n-1}\notag \\
&\leq \sum_{r=4}^{\infty}L_{q-1}^{r}\lambda^{3r/4}2^{2r}\leq \frac{4L_{q-1}\lambda^{3/4}}{1-4L_{q-1}\lambda^{3/4}}  \, . \label{g11}
\end{align}
Next, let $\{\rho_k\}_{k\geq 1}$, $\{\sigma_k\}_{k\geq 1}$, $\alpha$, $\gamma$ and $\theta$ be the maps defined as in the proof of Lemma~\ref{uniquef0}. As noted in the proof of Lemma \ref{uniquef0} the maps $\{\rho_k\}_{k\geq 1}$ are injective and their images of are disjoint. Since $\iota(\delta)=\iota(\rho_k(\delta))+3$, we get
\begin{align}
\sum_{\delta \in \Delta^+}\sum_{k=1}^{\infty} D(\delta_1-k-1, k-1, \delta_2, \ldots, \delta_n) \lambda^{\iota(\delta)}&=\sum_{\delta \in \Delta^+}\sum_{k=1}^{\infty} D(\rho_k(\delta)) \lambda^{\iota(\rho_k(\delta))+3}\notag\\
&\leq \sum_{\delta \in \Delta^+}D(\delta) \lambda^{\delta+3}=\lambda^3F(\lambda)\, . \label{g5}
\end{align}
Similarly, the maps $\{\sigma_k\}_{k\geq 1}$ are also injective and have disjoint images. Since $\iota(\delta)=\iota(\sigma_k(\delta))+1$, we have
\begin{align}
\sum_{\delta \in \Delta^+}\sum_{k=1}^{\infty} D(\delta_1-k, k, \delta_2, \ldots, \delta_n) \lambda^{\iota(\delta)}&=\sum_{\delta \in \Delta^+}\sum_{k=1}^{\infty} D(\sigma_k(\delta)) \lambda^{\iota(\sigma_k(\delta))+1}\notag\\
&\leq \sum_{\delta \in \Delta^+}D(\delta) \lambda^{\delta+1}=\lambda F(\lambda)\, . \label{g6}
\end{align}
The maps $\alpha$ and $\gamma$ are also injective, and $\iota(\delta)=\iota(\alpha(\delta))-4$ and $\iota(\delta)=\iota(\gamma(\delta))-3$. Therefore
\begin{align}
\sum_{\delta \in \Delta^+} D(\delta_1+4, \delta_2, \ldots, \delta_n) \lambda^{\iota(\delta)}&=\sum_{\delta \in \Delta^+} D(\alpha(\delta)) \lambda^{\iota(\alpha(\delta))-4}\notag\\
&\leq \sum_{\delta \in \Delta^+}D(\delta) \lambda^{\delta-4}=\frac{F(\lambda)}{\lambda^{4}} \label{g7}\, ,
\end{align}
and
\begin{align}
\sum_{\delta \in \Delta^+} D(\delta_1, 4, \delta_2, \ldots, \delta_n) \lambda^{\iota(\delta)}&=\sum_{\delta \in \Delta^+} D(\gamma(\delta)) \lambda^{\iota(\gamma(\delta))-3}\notag\\
&\leq \sum_{\delta \in \Delta^+}D(\delta) \lambda^{\delta-3}=\frac{F(\lambda)}{\lambda^{3}}  \label{g8}\, .
\end{align}
Finally, observe that $\theta^{-1}(\delta_2, \ldots, \delta_n)\cap \Delta^+=\{(k, \delta_2, \ldots, \delta_n ):  k=4, 5, \ldots\}$. Hence
\begin{align}
\sum_{\delta \in \Delta^+} D(\delta_2, \ldots, \delta_n) \lambda^{\iota(\delta)}&=\sum_{k=4}^{\infty}\sum_{\delta \in \Delta^+:\,\delta_1=k} D(\theta(\delta)) \lambda^{\iota(\theta(\delta))+k-1}\notag\\
&\leq \sum_{k=1}^{\infty}\sum_{\delta' \in \Delta^+} D(\delta')
\lambda^{\iota(\delta')+k-1}=\frac{F(\lambda)}{1-\lambda} \label{g9}\, .
\end{align}
By combining inequalities \eqref{g11}, \eqref{g5}, \eqref{g6}, \eqref{g7}, \eqref{g8}, \eqref{g9} with \eqref{g4} we obtain
\[
F(\lambda)\leq \frac{8L_{q-1}\lambda^{3/4}}{1-4L_{q-1}\lambda^{3/4}}+\biggl(4\lambda^3+4\lambda+\frac{2|\beta|d}{\lambda^4}+\frac{2|\beta|d}{\lambda^3}+\frac{2|\beta|d}{1-\lambda}\biggr)F(\lambda)\, .
\]
If we let $\lambda=L_q^{-1}=(20L_{q-1})^{-4/3}$, then $4\lambda^3+4\lambda<1/4$ and
\[
\frac{8L_{q-1}\lambda^{3/4}}{1-4L_{q-1}\lambda^{3/4}}=\frac{1}{2}\, .
\]
So by choosing $\beta$ such that 
\[
\frac{2|\beta|d}{\lambda^4}+\frac{2|\beta|d}{\lambda^3}+\frac{2|\beta|d}{1-\lambda}<\frac{1}{4}
\]
we get $F(\lambda)\leq  1$. Hence
\[
|H_{q, N}(s)|\leq D(\delta(s))\leq \lambda^{-\iota(\delta(s))}=L_q^{\iota(s)} \leq L_q^{|s|}\, .
\]
This finishes the proof of the lemma.
\end{proof}
By inequality \eqref{ineqfk1} we deduce that $\lim_{N\to \infty} H_{k, N}(s)$ exists along some subsequence. If $f_{2k}(s)$ denotes this subsequential limit then from equality \eqref{equationhk} and induction hypothesis we see that $f_{2k}$ satisfies equation \eqref{fkmastereq}. So, it is enough to prove that this equation has a unique solution for all sufficiently small $\beta$.
\begin{lmm}\label{uniquefk}
There exists $\beta_1(d, k)>0$, depending only on dimension $d$, such that for all $|\beta|<\beta_1(d, k)$, there exists a unique loop function $f: \cs \to\cc$ such that $(1)$ $f(\emptyset)=0 \,$, $(2)$ $|f(s)|\leq L^{|s|}$ for some $L\geq 1$, and $(3)$ $f$ satisfies equation \eqref{fkmastereq}. 
\end{lmm}
\begin{proof}
We already showed that there exists a solution for all sufficiently small $\beta$. Now, suppose both $f$ and $\tilde{f}$ satisfy conditions of the lemma. Consider the loop function
\[
g(s):=f(s)-\tilde{f}(s)+f_0(s)\, .
\]
By subtracting the equation \eqref{fkmastereq} for $f$ and $\tilde{f}$ and adding the equation \eqref{f0mastereq} we see that $g$ also satisfies equation \eqref{f0mastereq}. Moreover, for any non-null loop sequence $s$
\[
|g(s)|\leq |f(s)|+|\tilde{f}(s)|+|f_0(s)|\leq L^{|s|}+L^{|s|}+1\leq (2L+1)^{|s|}\, .
\] 
Since $f(\emptyset)=\tilde{f}(\emptyset)=0$ and $f_0(\emptyset)=1$, the above inequality is still valid when $s$ is the null loop sequence and $g(\emptyset)=1$.  Thus $g$ satisfies all conditions of Lemma \ref{uniquef0}. Therefore $g(s)=f_0(s)$ for all $s$ for all sufficiently small $\beta$, which means $f=\tilde{f}$.
\end{proof}
Lemma \ref{uniquefk} and previous calculations finish the proof of Theorem \ref{fk}.
\section{Symmetrized master loop equation}\label{symmastereqsection}
The following theorem is the symmetrized version of Theorem \ref{fk}.
\begin{thm}\label{symfk}
The function $f_{2k}$ satisfies equation
\begin{align}
|s|f_{2k}(s)&=\ell(s)f_{2k-2}(s)+\sum_{s'\in \fst^{-}(s)}f_{2k-2}(s')-\sum_{s'\in \fst^{+}(s)}f_{2k-2}(s')+\sum_{s'\in \fs^{-}(s)}f_{2k}(s')-\sum_{s'\in \fs^{+}(s)}f_{2k}(s')\notag\\
&\quad +\frac{\beta}{2}\sum_{s'\in \fd^-(s)}f_{2k}(s')-\frac{\beta}{2}\sum_{s'\in \fd^+(s)}f_{2k}(s')+\frac{\beta}{2}\sum_{s'\in \fex^-(s)}f_{2k}(s')-\frac{\beta}{2}\sum_{s'\in \fex^+(s)}f_{2k}(s') \, .\label{symfkmastereq}
\end{align}
where $f_j$ is understood to be zero when $j<0$. 
\end{thm}
\begin{proof}
The proof is by induction on $k$. The case $k=0$ follows from equation \eqref{symfinmastereq} by letting $N\to \infty$ on both sides. Now, let $k\geq 1$ and suppose the claim holds for all $k'<k$. By induction hypothesis
\begin{align}
|s|\sum_{j=0}^{k-1}\frac{1}{N^{2j}}f_{2j}(s)=\frac{\ell(s)}{N^2}\sum_{j=0}^{k-2}\frac{1}{N^{2j}}f_{2j}(s)& + \frac{1}{N^2}\sum_{s\in \fst^-(s)}\sum_{j=0}^{k-2} \frac{1}{N^{2j}}f_{2j}-\frac{1}{N^2}\sum_{s\in \fst^+(s)}\sum_{j=0}^{k-2} \frac{1}{N^{2j}}f_{2j} \notag \\
&\quad + \sum_{s\in \fs^-(s)}\sum_{j=0}^{k-1} \frac{1}{N^{2j}}f_{2j}-\sum_{s\in \fs^+(s)}\sum_{j=0}^{k-1} \frac{1}{N^{2j}}f_{2j} \notag \\
&\quad + \frac{\beta}{2} \sum_{s\in \fd^-(s)}\sum_{j=0}^{k-1} \frac{1}{N^{2j}}f_{2j}- \frac{\beta}{2} \sum_{s\in \fd^+(s)}\sum_{j=0}^{k-1} \frac{1}{N^{2j}}f_{2j} \notag \\
&\quad + \frac{\beta}{2} \sum_{s\in \fex^-(s)}\sum_{j=0}^{k-1} \frac{1}{N^{2j}}f_{2j}- \frac{\beta}{2} \sum_{s\in \fex^+(s)}\sum_{j=0}^{k-1} \frac{1}{N^{2j}}f_{2j} \, .  \notag
\end{align}
Also we can write equation \eqref{symfinmastereq} as
\begin{align}
|s|\phi_N(s)=\frac{\ell(s)}{N^2}\phi_N(s)&+\frac{1}{N^2}\sum_{s'\in \fst^-(s)}\phi_N(s')-\frac{1}{N^2}\sum_{s'\in \fst^+(s)}\phi_N(s')\\
&\quad + \sum_{s'\in \fs^-(s)}\phi_N(s') - \sum_{s'\in \fs^+(s)}\phi_N(s')\notag\\
&\quad + \frac{\beta}{2}\sum_{s'\in \fd^-(s)}\phi_N(s') - \frac{\beta}{2}\sum_{s'\in \fd^+(s)}\phi_N(s')\notag\\
&\quad + \frac{\beta}{2}\sum_{s'\in \fex^-(s)}\phi_N(s') - \frac{\beta}{2}\sum_{s'\in \fex^+(s)}\phi_N(s') \, .\notag
\end{align}
By subtracting the first equation above from the second one and multiplying both sides by $N^{2k}$ we obtain
\begin{align}
|s|H_{k, N}(s)=\ell(s)H_{k-1, N}(s)&+\sum_{s'\in \fst^-(s)}H_{k-1, N}(s')-\sum_{s'\in \fst^+(s)}H_{k-1, N}(s')\notag\\
&\quad + \sum_{s'\in \fs^-(s)}H_{k, N}(s') - \sum_{s'\in \fs^+(s)}H_{k, N}(s')\notag\\
&\quad + \frac{\beta}{2}\sum_{s'\in \fd^-(s)}H_{k, N}(s') - \frac{\beta}{2}\sum_{s'\in \fd^+(s)}H_{k, N}(s')\notag\\
&\quad + \frac{\beta}{2}\sum_{s'\in \fex^-(s)}H_{k, N}(s') - \frac{\beta}{2}\sum_{s'\in \fex^+(s)}H_{k, N}(s')\, . \notag
\end{align}
By letting $N\to \infty$ in the last equation we get the desired result.
\end{proof}
\section{Series expansion}
Define collection of numbers $\{a_{i, k}(s): i, k\geq 0, s\in \cs\}$ inductively as follows. Let $a_{0, 0}(\emptyset)=1$, $a_{0,0}(s)=0$ for all $s\neq \emptyset$, and let $a_{i, 0}(\emptyset)=0$ for all $i\geq 1$. If $k\geq 1$, then let $a_{i, k}(\emptyset)=0$ for all $i\geq 0$ and $a_{0, k}(s)=0$ for all $s\in \cs$. Now, fix triple $(k, i, s)$ and suppose $a_{i', k'}(s')$ is defined for all $k'<k$ or $k'=k$ and $i'<i$ or $k'=k$, $i'=i$ and $\iota(s')<\iota(s)$. Then, let
\begin{align}
ma_{i,k}(s)=t_1t a_{i, k-1}(s)+\sum_{\text{merge}\,s}a_{i, k-1}+\sum_{\text{split}\,s}a_{i, k}+\frac{1}{2}\sum_{\text{deform}\,s}a_{i-1, k}+\frac{1}{2}\sum_{\text{expand}\,s}a_{i-1, k} \label{aik}
\end{align}
where $a_{i, k}(s)$ is understood to be zero for all $s$ if $i<0$ or $k<0$. Note that since $\iota(s')<\iota(s)$ for all $s'\in \fs(s)$ (Lemma \ref{iota2}) the third term on the right side of above equality was already defined. 
The following theorem is the main result of this section.
\begin{thm}\label{seriesfk}
Let $f_k$ be as in the previous section and let $a_{i,k}(s)$ be as in \eqref{aik}. There exists sequence of positive numbers $\{\beta'_0(d,k)\}_{k\geq 0}$, depending only on $d$, such that for any $|\beta|\leq \beta'_0(d,k)$ and loop sequence $s$
\[
f_k(s)=\sum_{i=0}^{\infty}a_{i,k}(s)\beta^i\, .
\]
Moreover, the above series is absolutely convergent.
\end{thm}
We need the following lemma to prove Theorem~\ref{seriesfk}.
\begin{lmm}\label{boundaik}
There exists $K(d)\geq 3$, depending only on dimension $d$, such that for any $s\in \cs$ with degree sequence $\delta=(\delta_1, \ldots, \delta_n)$ and $i\geq 0$
\begin{align}
|a_{i, k}(s)|\leq K(d)^{(5+2k)i+\iota(\delta)}|\delta|^{3k}C_{\delta_1-1}\cdots C_{\delta_n-1} \label{ineqaik}
\end{align}
where $C_i$ is the $i^{\text{th}}$ Catalan number. The product of Catalan numbers is assumed to be $1$ when $s=\emptyset$.
\end{lmm}
\begin{proof}
We will use three fold induction: first on $k$, then on $i$ and finally on $\iota(s)$. The number $K=K(d)$ will be chosen later in the proof.  Since $a_{0, 0}(s)$ is either $0$ or $1$, inequality obviously holds for triples $(0, 0, s)$. Next let $i\geq 1$ and assume inequality holds for all triples $(i', 0, s)$  with $i'<i$. We will now use induction on $\iota(s)$ to prove inequality for the triples $(i, 0, s)$. Since $a_{i,0}(\emptyset)=0$, the claim is true for $(i, 0, \emptyset)$. So, let $s\neq \emptyset$ and assume the claim is also true for all triples $(i, 0, s')$ with $\iota(s')<\iota(s)$. Note that by equation \eqref{aik}
\begin{align}
|a_{i,0}(s)|\leq \frac{1}{m}\sideset{}{^+}\sum_{\text{split}\,s}|a_{i, 0}|+\frac{1}{2m}\sideset{}{^+}\sum_{\text{deform}\,s}|a_{i-1, 0}|+\frac{1}{2m}\sideset{}{^+}\sum_{\text{expand}\,s}|a_{i-1, 0}| \,. \label{b1}
\end{align}
Recall that
\begin{align*}
\frac{1}{m}\sideset{}{^+}\sum_{\text{split}\,s}|a_{i,0}|&=\frac{1}{m}\sum_{x\in A_1, \, y\in B_1} |a_{i,0}(\times_{x,y}^1 l_1, \times_{x,y}^2 l_1, \ldots, l_n)|+\frac{1}{m}\sum_{x\in B_1, \, y\in A_1} |a_{i,0}(\times_{x,y}^1 l_1, \times_{x,y}^2 l_1, \ldots, l_n)|\\
&\quad + \frac{1}{m}\sum_{\substack{x,y\in A_1\\ x\ne y}} |a_{i,0}(\times_{x,y}^1 l_1, \times_{x,y}^2 l_1, \ldots, l_n)|+\frac{1}{m} \sum_{\substack{x,y\in B_1\\ x\ne y}} |a_{i,0}(\times_{x,y}^1 l_1, \times_{x,y}^2 l_1, \ldots, l_n)|\, .
\end{align*}
By Lemma \ref{split2} $\iota(s')\leq \iota(s)-3$ for all $s'\in \fs^-(s)$. So by induction hypothesis and Lemma \ref{catalan}
\begin{align*}
\frac{1}{m}\sum_{x\in A_1, \, y\in B_1} |a_{i,0}(\times_{x,y}^1 l_1, \times_{x,y}^2 l_1, \ldots, l_n)|&\leq \frac{1}{m}\sum_{x\in A_1, \, y\in B_1}K^{5i+\iota(\delta)-3}C_{\delta_1-|y-x|-2}C_{|y-x|-2}\prod_{j=2}^{n} C_{\delta_j-1}\\
&\leq \frac{2}{m} \sum_{x\in A_1}\sum_{k=1}^{\delta_1-1} K^{5i+\iota(\delta)-3}C_{\delta_1-k-1}C_{k-1}\prod_{j=2}^{n} C_{\delta_j-1}\\
&\leq  \frac{2}{m} \sum_{x\in A_1}K^{5i+\iota(\delta)-3}C_{\delta_1-1}\prod_{j=2}^{n} C_{\delta_j-1}\\
&\leq 2K^{5i+\iota(\delta)-3}C_{\delta_1-1}\cdots C_{\delta_n-1}\, .
\end{align*}
The same bound is also valid if we swap $A_1$ and $B_1$. Similarly, by Lemma \ref{split1} $\iota(s')\leq \iota(s)-1$ for all $s'\in \fs^+(s)$. So, by induction hypothesis and Lemma \ref{catalan}
\begin{align*}
\frac{1}{m}\sum_{\substack{x,y\in A_1\\ x\ne y}} |a_{i,0}(\times_{x,y}^1 l_1, \times_{x,y}^2 l_1, \ldots, l_n)|&\leq \frac{1}{m}\sum_{\substack{x,y\in A_1\\ x\ne y}}K^{5i+\iota(\delta)-1}C_{\delta_1-|y-x|-1}C_{|y-x|-1}\prod_{j=2}^{n} C_{\delta_j-1}\\
&\leq \frac{2}{m} \sum_{x\in A_1}\sum_{k=1}^{\delta_1-1} K^{5i+\iota(\delta)-1}C_{\delta_1-k-1}C_{k-1}\prod_{j=2}^{n} C_{\delta_j-1}\\
&\leq \frac{2}{m} \sum_{x\in A_1} K^{5i+\iota(\delta)-1}C_{\delta_1-1}\prod_{j=2}^{n} C_{\delta_j-1}\\
&\leq 2K^{5i+\iota(\delta)-1}C_{\delta_1-1}\cdots C_{\delta_n-1}\, .
\end{align*}
The same inequality is also true if we replace $A_1$ with $B_1$. Thus
\begin{align}
\frac{1}{m}\sideset{}{^+}\sum_{\text{split}\,s}|a_{i,0}|&\leq (4K^{-3}+4K^{-1})K^{5i+\iota(\delta)}C_{\delta_1-1}\cdots C_{\delta_n-1}\, . \label{b2}
\end{align}
Next, recall that
\begin{align*}
\frac{1}{2m}\sideset{}{^+}\sum_{\text{deform}\,s}|a_{i-1,0}|&=\frac{1}{2m}\sum_{p\in \cp^+(e)}\sum_{x\in C_1}|a_{i-1,0}(l_1 \ominus_{x}p, \ldots, l_n)|\\
&\quad +\frac{1}{2m}\sum_{p\in \cp^+(e)}\sum_{x\in C_1}|a_{i-1,0}(l_1 \oplus_{x}p, \ldots, l_n)| \, .
\end{align*}
By Lemma \ref{deform} $\iota(s')\leq \iota(s)+4$ for any $s'\in \fd(s)$. Thus by induction hypothesis and inequality \eqref{catalan0}, for any plaquette $p$
\begin{align*}
|a_{i-1,0}(l_1 \ominus_{x}p, \ldots, l_n)|&\leq K^{5(i-1)+\iota(\delta)+4}C_{\delta_1+3}C_{\delta_2-1} \cdots C_{\delta_n-1}\\
&\leq 4^4K^{5i+\iota(\delta)-1}C_{\delta_1-1}C_{\delta_2-1} \cdots C_{\delta_n-1}\,.
\end{align*}
Since $|\cp^+(e)|=2(d-1)$, we get
\begin{align*}
\frac{1}{2m}\sum_{p\in \cp^+(e)}\sum_{x\in C_1}|a_{i-1,0}(l_1 \ominus_{x}p, \ldots, l_n)|& \leq 256d K^{5i+\iota(\delta)-1}C_{\delta_1-1}C_{\delta_2-1} \cdots C_{\delta_n-1}\, .
\end{align*}
The same inequality is also true for the second term in the deformation sum. Hence
\begin{align}
\frac{1}{2m}\sideset{}{^+}\sum_{\text{deform}\,s}|a_{i-1,0}|&\leq 512dK^{5i+\iota(\delta)-1}C_{\delta_1-1}C_{\delta_2-1} \cdots C_{\delta_n-1}\,. \label{b3}
\end{align}
Finally, note that
\begin{align*}
\frac{1}{2m}\sideset{}{^+}\sum_{\text{expand}\,s}|a_{i-1,0}|&=\frac{1}{2}\sum_{p\in \cp(e)}|a_{i-1,0}(l_1, p, l_2, \ldots, l_n)|+ 
\frac{1}{2}\sum_{p\in \cp(e^{-1})}|a_{i-1,0}(l_1, p, l_2, \ldots, l_n)| \, .
\end{align*}
By Lemma \ref{expand} $\iota(s')\leq \iota(s)+3$ for any $s\in\fex(s)$. Since $|p|=4$ and $|\cp(e)|=2(d-1)$, by induction hypothesis
\begin{align*}
\frac{1}{2}\sum_{p\in \cp(e)}|a_{i-1,0}(l_1, p, l_2, \ldots, l_n)|&\leq  dK^{5(i-1)+\iota(\delta)+3}C_{\delta_1-1}C_3C_{\delta_2-1} \cdots C_{\delta_n-1}\\
&=5dK^{5i+\iota(\delta)-2}C_{\delta_1-1}C_{\delta_2-1} \cdots C_{\delta_n-1} \, .
\end{align*}
The same inequality is also true for the second term in the expansion sum. Thus
\begin{align}
\frac{1}{2m}\sideset{}{^+}\sum_{\text{expand}\,s}|a_{i-1,0}|&\leq 10dK^{5i+\iota(\delta)-2}C_{\delta_1-1}C_{\delta_2-1} \cdots C_{\delta_n-1}\, . \label{b4}
\end{align}
By combining \eqref{b2}, \eqref{b3}, \eqref{b4} with \eqref{b1} we get
\begin{align*}
|a_{i,0}(s)|\leq (4K^{-3}+4K^{-1}+512dK^{-1}+10dK^{-2})K^{5i+\iota(\delta)}C_{\delta_1-1}C_{\delta_2-1} \cdots C_{\delta_n-1}\,.
\end{align*}
By choosing $K$ so that
\begin{align}
4K^{-3}+4K^{-1}+512dK^{-1}+10dK^{-2}&\leq 1 \label{ineqK1}
\end{align}
we conclude that inequality \eqref{ineqaik} is valid for triples $(i, 0, s)$. 

Next, let $k\geq 1$ and assume that inequality \eqref{ineqaik} is valid for triples $(i, k', s)$ with $k'<k$. Since $a_{0, k}(s)=0$ for all $s\in \cs$, the inequality also holds for the triples $(0, k, s)$. So, let $i\geq 1$ and assume the claim is true for all triples $(i', k, s)$ with $i'<i$. Now we will use induction on $\iota(s)$ to prove for the triples $(i, k, s)$. Since $a_{i, k}(\emptyset)=0$, triple $(i, k, \emptyset)$ satisfies inequality \eqref{ineqaik}. So let $s\neq \emptyset$ and assume that inequality \eqref{ineqaik} is also true for triples $(i, k, s')$ with $\iota(s')<\iota(s)$. By triangle inequality
\begin{align}
|a_{i,k}(s)|\leq \frac{t_1t}{m}|a_{i, k-1}(s)|+\frac{1}{m}\sideset{}{^+}\sum_{\text{merge}\,s}|a_{i, k-1}|+\frac{1}{m}\sideset{}{^+}\sum_{\text{split}\,s}|a_{i, k}|+\frac{1}{2m}\sideset{}{^+}\sum_{\text{deform}\,s}|a_{i-1, k}|+\frac{1}{2m}\sideset{}{^+}\sum_{\text{expand}\,s}|a_{i-1, k}|\,. \label{h0}
\end{align}
Since $t_1\leq m$ and $t\leq |s|$, by induction hypothesis
\begin{align}
\frac{t_1t}{m}|a_{i, k-1}(s)|&\leq |\delta|K^{(5+2k-2)i+\iota(\delta)}|\delta|^{3k-3}C_{\delta_1-1}\cdots C_{\delta_n-1}\notag\\
&\leq K^{(5+k)i+\iota(\delta)-2}|\delta|^{3k}C_{\delta_1-1}\cdots C_{\delta_n-1}\, .\label{h1}
\end{align}
Recall that 
\begin{align*}
\frac{1}{m}\sum_{\text{merge}\,s}|a_{i, k-1}|&=\frac{1}{m}\sum_{r=2}^{n}\sum_{x\in A_1, y\in B_r}|a_{i, k-1}(l_1\ominus_{x, y}l_r, l_2, \ldots, l_{r-1}, l_{r+1}, \ldots, l_n)|\\
&+\frac{1}{m}\sum_{r=2}^{n}\sum_{x\in B_1, y\in A_r}|a_{i, k-1}(l_1\ominus_{x, y}l_r, l_2, \ldots, l_{r-1}, l_{r+1}, \ldots, l_n)|\\
&+\frac{1}{m}\sum_{r=2}^{n}\sum_{x\in A_1, y\in A_r}|a_{i, k-1}(l_1\oplus_{x, y}l_r, l_2, \ldots, l_{r-1}, l_{r+1}, \ldots, l_n)|\\
&+\frac{1}{m}\sum_{r=2}^{n}\sum_{x\in B_1, y\in B_r}|a_{i, k-1}(l_1\oplus_{x, y}l_r, l_2, \ldots, l_{r-1}, l_{r+1}, \ldots, l_n)|\, .
\end{align*}
By Lemma \ref{merger} $|\delta(s')|\leq |\delta|$ and $\iota(\delta(s'))\leq \iota(\delta)+1$ for any $s'\in \fst(s)$. So, by induction hypothesis and Lemma \ref{catalan1}
\begin{align}
&\frac{1}{m}\sum_{r=2}^{n}\sum_{x\in A_1, y\in B_r}|a_{i, k-1}(l_1\ominus_{x, y}l_r, l_2, \ldots, l_{r-1}, l_{r+1}, \ldots, l_n)|\notag\\
&\quad \leq \frac{1}{m}\sum_{r=2}^{n}\sum_{x\in A_1, y\in B_r}K^{(5+2k-2)i+\iota(\delta)+1}|\delta|^{3k-3}C_{\delta_1+\delta_r-1}C_{\delta_2-1}\cdots C_{\delta_{r-1}-1}C_{\delta_{r+1}-1}\cdots C_{\delta_n-1}\notag\\
&\quad \leq \frac{1}{m}\sum_{r=2}^{n}\sum_{x\in A_1, y\in B_r}K^{(5+2k)i+\iota(\delta)-1}|\delta|^{3k-3}(\delta_1+\delta_r)^2C_{\delta_1-1}\cdots C_{\delta_n-1}\notag\\
&\quad \leq \sum_{r=2}^{n}\delta_r K^{(5+2k)i+\iota(\delta)-1}|\delta|^{3k-3}(\delta_1+\delta_r)^2C_{\delta_1-1}\cdots C_{\delta_n-1}\notag\\
&\quad \leq K^{(5+2k)i+\iota(\delta)-1}|\delta|^{3k}C_{\delta_1-1}\cdots C_{\delta_n-1}\, . \notag
\end{align}
The same bound also holds for the rest of the terms in the merging sum. Therefore
\begin{align}
\frac{1}{m}\sum_{\text{merge}\,s}|a_{i, k-1}|&\leq 4K^{(5+2k)i+\iota(\delta)-1}|\delta|^{3k}C_{\delta_1-1}\cdots C_{\delta_n-1} \, .\label{h2}
\end{align}
We will proceed as in case $k=0$ to bound the last three terms on the right of inequality \eqref{h0}. By Lemma \ref{split2} $|\delta(s')|\leq |\delta|$ and $\iota(s')\leq \iota(s)-3$ for all $s'\in \fs^-(s)$. So by induction hypothesis and Lemma \ref{catalan}
\begin{align*}
&\frac{1}{m}\sum_{x\in A_1, \, y\in B_1} |a_{i,k}(\times_{x,y}^1 l_1, \times_{x,y}^2 l_1, \ldots, l_n)|\\
&\quad \leq \frac{1}{m}\sum_{x\in A_1, \, y\in B_1}K^{(5+2k)i+\iota(\delta)-3}|\delta|^{3k}C_{\delta_1-|y-x|-2}C_{|y-x|-2}\prod_{j=2}^{n} C_{\delta_j-1}\\
&\quad \leq \frac{2}{m} \sum_{x\in A_1}\sum_{k=1}^{\delta_1-1} K^{(5+2k)i+\iota(\delta)-3}|\delta|^{3k} C_{\delta_1-k-1}C_{k-1}\prod_{j=2}^{n} C_{\delta_j-1}\\
&\quad \leq  \frac{2}{m} \sum_{x\in A_1}K^{(5+2k)i+\iota(\delta)-3}|\delta|^{3k}C_{\delta_1-1}\prod_{j=2}^{n} C_{\delta_j-1}\\
&\quad \leq 2K^{(5+2k)i+\iota(\delta)-3}|\delta|^{3k}C_{\delta_1-1}\cdots C_{\delta_n-1}\, .
\end{align*}
Similarly, by Lemma \ref{split1} if  $s'\in \fs^+(s)$, then $|\delta(s')|\leq |\delta|$ and $\iota(s')\leq \iota(s)-1$. So by induction hypothesis and Lemma \ref{catalan}
\begin{align*}
&\frac{1}{m}\sum_{\substack{x,y\in A_1\\ x\ne y}} |a_{i,k}(\times_{x,y}^1 l_1, \times_{x,y}^2 l_1, \ldots, l_n)|\\
&\quad \leq \frac{1}{m}\sum_{\substack{x,y\in A_1\\ x\ne y}}K^{(5+2k)i+\iota(\delta)-1}|\delta|^{3k}C_{\delta_1-|y-x|-1}C_{|y-x|-1}\prod_{j=2}^{n} C_{\delta_j-1}\\
&\quad \leq \frac{2}{m} \sum_{x\in A_1}\sum_{k=1}^{\delta_1-1} K^{(5+2k)i+\iota(\delta)-1}|\delta|^{3k}C_{\delta_1-k-1}C_{k-1}\prod_{j=2}^{n} C_{\delta_j-1}\\
&\quad \leq \frac{2}{m} \sum_{x\in A_1} K^{(5+2k)i+\iota(\delta)-1}|\delta|^{3k}C_{\delta_1-1}\prod_{j=2}^{n} C_{\delta_j-1}\\
&\quad \leq 2K^{(5+2k)i+\iota(\delta)-1}|\delta|^{3k}C_{\delta_1-1}\cdots C_{\delta_n-1}\, .
\end{align*}
The same bounds are still valid if we replace $A_1$ with $B_1$ in above two inequalities. Therefore
\begin{align}
\frac{1}{m}\sideset{}{^+}\sum_{\text{split}\,s}|a_{i,k}|&\leq (4K^{-3}+4K^{-1})K^{(5+2k)i+\iota(\delta)}|\delta|^{3k}C_{\delta_1-1}\cdots C_{\delta_n-1}\, . \label{h3}
\end{align}
By Lemma \ref{deform} $s'\in \fd(s)$, then $|\delta(s')|\leq |\delta|+4$ and $\iota(s')\leq \iota(s)+4$. Since $|\delta|\geq 4$ and $K\geq 3$, by induction hypothesis and inequality \eqref{catalan0}
\begin{align*}
|a_{i-1,k}(l_1 \ominus_{x}p, \ldots, l_n)|&\leq K^{(5+2k)(i-1)+\iota(\delta)+4}(|\delta|+4)^{3k}C_{\delta_1+3}C_{\delta_2-1} \cdots C_{\delta_n-1}\\
&\leq K^{(5+2k)i+\iota(\delta)-1-2k}(2|\delta|)^{3k}4^4C_{\delta_1-1}C_{\delta_2-1} \cdots C_{\delta_n-1}\\
&\leq 256K^{(5+2k)i+\iota(\delta)-1}|\delta|^{3k}C_{\delta_1-1}C_{\delta_2-1} \cdots C_{\delta_n-1}\,.
\end{align*}
Since there are less than  $2d$ positively oriented plaquettes passing through each edge we get
\begin{align*}
\frac{1}{2m}\sum_{p\in \cp^+(e)}\sum_{x\in C_1}|a_{i-1,k}(l_1 \ominus_{x}p, \ldots, l_n)|& \leq 256d K^{(5+2k)i+\iota(\delta)-1}|\delta|^{3k}C_{\delta_1-1}C_{\delta_2-1} \cdots C_{\delta_n-1}\, .
\end{align*}
The same bound also holds for the second term in the deformation sum. So
\begin{align}
\frac{1}{2m}\sideset{}{^+}\sum_{\text{deform}\,s}|a_{i-1,k}|&\leq 512dK^{(5+2k)i+\iota(\delta)-1}|\delta|^{3k}C_{\delta_1-1}C_{\delta_2-1} \cdots C_{\delta_n-1}\,. \label{h4}
\end{align}
By Lemma \ref{expand} if $s\in\fex(s)$, then $|\delta(s')|=|\delta|+4$ and $\iota(s')=\iota(s)+3$. Since $|p|=4$ for any plaquette $p$ and $|\cp(e)|=2(d-1)$ for any edge $e$ by induction hypothesis
\begin{align*}
\frac{1}{2}\sum_{p\in \cp(e)}|a_{i-1,k}(l_1, p, l_2, \ldots, l_n)|&\leq dK^{(5+2k)(i-1)+\iota(\delta)+3}(|\delta|+4)^{3k}C_{\delta_1-1}C_3C_{\delta_2-1} \cdots C_{\delta_n-1}\\
&\leq  5dK^{(5+2k)i+\iota(\delta)-2-2k}(2|\delta|)^{3k}C_{\delta_1-1}C_{\delta_2-1} \cdots C_{\delta_n-1}\\
&\leq 5dK^{(5+2k)i+\iota(\delta)-2}|\delta|^{3k}C_{\delta_1-1}C_{\delta_2-1} \cdots C_{\delta_n-1}
\end{align*}
The same inequality is also true for the second term in the expansion sum. Hence
\begin{align}
\frac{1}{2m}\sideset{}{^+}\sum_{\text{expand}\,s}|a_{i-1,k}|&\leq 10dK^{(5+2k)i+\iota(\delta)-2}|\delta|^{3k}C_{\delta_1-1}C_{\delta_2-1} \cdots C_{\delta_n-1}\, . \label{h5}
\end{align}
By combining \eqref{h1}, \eqref{h2}, \eqref{h3}, \eqref{h4}, \eqref{h5} and \eqref{h0} we get
\begin{align*}
|a_{i, k}(s)|&\leq (8K^{-1}+K^{-2}+4K^{-3}+512dK^{-1}+10dK^{-2})K^{(5+2k)i+\iota(\delta)-2}|\delta|^{3k}C_{\delta_1-1}C_{\delta_2-1} \cdots C_{\delta_n-1}\, . 
\end{align*}
We finish the proof by choosing $K$ so that 
\begin{align}
8K^{-1}+K^{-2}+4K^{-3}+512dK^{-1}+10dK^{-2}\leq 1\, .  \label{ineqK2}
\end{align}
\end{proof}
\begin{proof}[Proof of Theorem~\ref{seriesfk}]
The proof is by induction on $k$. Let
\[
\psi_{2k}(s):=\sum_{i=0}^{\infty}a_{i,k}(s)\beta^i \, ,
\]
and note that by Lemma~\ref{boundaik} the series on the right side is absolutely convergent for all $|\beta|<K^{-(5+2k)}$. Moreover, if $\delta=(\delta_1, \ldots, \delta_n)$ is the degree sequence of a non-null loop sequence $s$, then by Lemma~\ref{boundaik} and the fact that  $\iota(\delta)=|\delta|-\#\delta \leq |\delta|-1$ we have
\begin{align*}
|\psi_0(s)|&\leq \sum_{i=0}^{\infty}|a_{i,0}(s)||\beta|^{i}\leq \sum_{i=0}^{\infty}K^{5i+\iota(\delta)}C_{\delta_1-1}\ldots C_{\delta_n-1}|\beta|^{i}\\
&\leq \sum_{i=0}^{\infty}K^{5i+\iota(\delta)}4^{\delta_1-1}\cdots 4^{\delta_n-1}|\beta|^{i}=(4K)^{\iota(\delta)}\sum_{i=0}^{\infty}\left(|\beta|K^{5}\right)^{i}\\
&\leq \frac{(4K)^{|\delta|-1}}{1-|\beta|K^5} \leq \frac{(4K)^{|s|}}{4(1-|\beta|K^5)}\,. 
\end{align*}
So $|\psi_0(s)|\leq (4K)^{|s|}$ for all $|\beta|$ such that $2|\beta|K^5<1$ and for all non-null $s\in \cs$. Since $\psi_0(\emptyset)=1$ this bound is also valid when $s=\emptyset$. Finally, by the definition of $a_{i, 0}(s)$ for any non-null loop sequence $s$ we have
\begin{align*}
\psi_0(s)&=\sum_{i=0}^{\infty}a_{i,0}(s)\beta^i\\
&=\sum_{i=0}^{\infty}\biggl(\frac{1}{m}\sum_{\text{split}\,s}a_{i, 0}+\frac{1}{2m}\sum_{\text{deform}\,s}a_{i-1, 0}+\frac{1}{2m}\sum_{\text{expand}\,s}a_{i-1, 0}\biggr)\beta^i\\
&=\frac{1}{m}\sum_{\text{split}\,s}\sum_{i=0}^{\infty}a_{i, 0}\beta^{i}+\frac{\beta}{2m}\sum_{\text{deform}\,s}\sum_{i=1}^{\infty}a_{i-1, 0}\beta^{i-1}+\frac{\beta}{2m}\sum_{\text{expand}\,s}\sum_{i=1}^{\infty}a_{i-1, 0}\beta^{i-1}\\
&=\frac{1}{m}\sum_{\text{split}\,s}\psi_0+\frac{\beta}{2m}\sum_{\text{deform}\,s}\psi_0+\frac{\beta}{2m}\sum_{\text{expand}\,s}\psi_0\,.
\end{align*}
This means $\psi_0$ satisfies equation \eqref{f0mastereq}. So we showed that for all sufficiently small $\beta$, depending only on $d$, $\psi_0$ satisfies conditions $(a)$, $(b)$, $(c)$ of Lemma~\ref{uniquef0}. Therefore $\psi_0(s)=f_0(s)$ for all $s\in \cs$. This completes the base case of induction. 

Now, let $k\geq 1$ and suppose the claim holds for all $k'<k$. Since $a_{0, k}(s)=0$ for all $s$
\begin{align*}
|\psi_{2k}(s)|&\leq \sum_{i=1}^{\infty}|a_{i,k}(s)||\beta|^{i}\leq \sum_{i=1}^{\infty}K^{(5+2k)i+\iota(\delta)}|\delta|^{3k} C_{\delta_1-1}\ldots C_{\delta_n-1}|\beta|^{i}\\
&\leq \sum_{i=1}^{\infty}K^{(5+2k)i+|\delta|}2^{3k|\delta|}4^{\delta_1}\cdots 4^{\delta_n}|\beta|^{i}=(2^{3k+2}K)^{|\delta|}\sum_{i=1}^{\infty}\left(|\beta|K^{5+2k}\right)^{i}\\
&=(2^{3k+2}K)^{|s|}\frac{|\beta|K^{5+2k}}{1-|\beta|K^{5+2k}}\,. 
\end{align*}
So 
\begin{align}
|\psi_{2k}(s)| \leq (2^{3k+2}K)^{|s|} \label{ineqK3}
\end{align}
for all $\beta$ such that $2|\beta|K^{5+2k}\leq 1$. Since $\psi_{2k}(\emptyset)=0$ this bound is still valid when $s=\emptyset$. Finally, observe that by definition of $a_{i, k}(s)$
\begin{align*}
\psi_{2k}(s)&=\sum_{i=0}^{\infty}a_{i,k}(s)\beta^i\\
&=\sum_{i=0}^{\infty}\biggl(\frac{t_1t}{m}a_{i, k-1}(s)+\frac{1}{m}\sum_{\text{merge}\,s}a_{i, k-1}+\frac{1}{m}\sum_{\text{split}\,s}a_{i, k}\\
&\qquad \qquad \qquad \qquad + \frac{1}{2m}\sum_{\text{deform}\,s}a_{i-1, k}+\frac{1}{2m}\sum_{\text{expand}\,s}a_{i-1, k}\biggr)\beta^i\\
&=\frac{t_1t}{m}\sum_{i=0}^{\infty}a_{i, k-1}\beta^{i}+\frac{1}{m}\sum_{\text{merge}\,s}\sum_{i=0}^{\infty}a_{i, k-1}\beta^{i}+\frac{1}{m}\sum_{\text{split}\,s}\sum_{i=0}^{\infty}a_{i, k}\beta^{i}\\
&\qquad \qquad \qquad \qquad + \frac{\beta}{2m}\sum_{\text{deform}\,s}\sum_{i=1}^{\infty}a_{i-1, k}\beta^{i-1}+\frac{\beta}{2m}\sum_{\text{expand}\,s}\sum_{i=1}^{\infty}a_{i-1, k}\beta^{i-1}\\
&=\frac{t_1t}{m}\psi_{2k-2}(s) + \frac{1}{m}\sum_{\text{merge}\,s}\psi_{2k-2} + \frac{1}{m}\sum_{\text{split}\,s}\psi_{2k} + \frac{\beta}{2m}\sum_{\text{deform}\,s}\psi_{2k} + \frac{\beta}{2m}\sum_{\text{expand}\,s}\psi_{2k} \, .
\end{align*}
So by induction hypothesis
\begin{align*}
\psi_{2k}(s)&=\frac{t_1t}{m}f_{2k-2}(s) + \frac{1}{m}\sum_{\text{merge}\,s}f_{2k-2} + \frac{1}{m}\sum_{\text{split}\,s}\psi_{2k} + \frac{\beta}{2m}\sum_{\text{deform}\,s}\psi_{2k} + \frac{\beta}{2m}\sum_{\text{expand}\,s}\psi_{2k} \, .
\end{align*}
Therefore, $\psi_{2k}$ satisfies conditions $(a)$, $(b)$, $(c)$ of Lemma~\ref{uniquefk} for all sufficiently small $|\beta|$. Hence $\psi_{2k}(s)=f_{2k}(s)$ for all $s\in \cs$.
\end{proof}
One immediate consequence of Theorem \ref{seriesfk} is the following corollary which gives the symmetrized version of equation \eqref{aik}.
\begin{cor}
The numbers $\{a_{i,k}(s):i, k\geq 0, s\in \cs\}$ satisfy equation
\begin{align}
|s|a_{i,k}(s)&=\ell(s)a_{i, k-1} + \sum_{s'\in \fst^-(s)}a_{i,k-1}(s') - \sum_{s'\in \fst^+(s)}a_{i,k-1}(s') + \sum_{s'\in \fs^-(s)}a_{i,k}(s') - \sum_{s'\in \fs^+(s)}a_{i,k}(s') \notag \\
&+ \frac{\beta}{2}\sum_{s'\in \fd^-(s)}a_{i-1,k}(s') - \frac{\beta}{2}\sum_{s'\in \fd^+(s)}a_{i-1,k}(s') + \frac{\beta}{2}\sum_{s'\in \fex^+(s)}a_{i-1,k}(s') - \frac{\beta}{2}\sum_{s'\in \fex^-(s)}a_{i-1,k}(s')\, .\label{symaik}
\end{align}
\end{cor}
\begin{proof}
The proof follows by equating coefficients of $\beta^i$ on both sides of equation \eqref{symfkmastereq}.
\end{proof}
We finish this section by proving part (iii) of Theorem \ref{maintheorem}.
\begin{proof}[Proof of part \text{(iii)} of Theorem \ref{maintheorem}]
Note that $K=1024d$ satisfies both of inequalities \eqref{ineqK1} and \eqref{ineqK2}. Therefore from inequality \eqref{ineqK3} we get
\[
|f_{2k}(s)|\leq (2^{3k+12}d)^{|s|}\, .
\]
\end{proof}

\section{Absolute convergence of the sum over trajectories}\label{absconv}
The goal of this section is to prove the following theorem.
\begin{thm}\label{absconvmxk}
There exists a sequence of positive real numbers $\{\beta_0''(d, k)\}_{k\geq 0}$ such that for any $|\beta|<\beta_0''(d, k)$ the sum
\begin{align}
\sum_{X\in \mx_k(s)}w_{\beta}(X) \label{sumxk}
\end{align}
is absolutely convergent.
\end{thm}
Define the sequence of numbers $\{b_{i, k}(s): i, k\geq 0, s \in \cs\}$ inductively as follows. Let $b_{0,0}(\emptyset)=1$, $b_{0, 0}(s)=0$ for all non-null loop sequence $s$ and let $b_{i, 0}(\emptyset)=0$ for all $i\geq 1$. If $k\geq 1$, let $b_{i, k}(\emptyset)=0$ for all $i\geq 0$ and $b_{0, k}(s)=0$ for all $s\in \cs$. Now suppose $b_{i', k'}(s')$ has been defined for all triples $(i', k', s')$ such that either $k'<k$, or $k'=k$ and $i'<i$ or $k'=k$, $i'=i$ and $\iota(s')<\iota(s)$. Then, let 
\begin{align}
b_{i,k}(s)&=\frac{\ell(s)}{|s|}b_{i, k-1}(s) + \frac{1}{|s|}\sum_{s'\in \fst(s)}b_{i,k-1}(s') + \frac{1}{|s|}\sum_{s'\in \fs(s)}b_{i,k}(s')\notag \\
&\qquad  + \frac{1}{2|s|}\sum_{s'\in \fd(s)}b_{i-1,k}(s') + \frac{1}{2|s|}\sum_{s'\in \fex(s)}b_{i-1,k}(s')\,. \label{bik}
\end{align}
where $b_{i, k}(s)$ is understood to be zero for all $s$ if $i<0$ or $k<0$.  Since $\iota(s')<\iota(s)$ for all $s'\in \fs(s)$, by the time of definition of $b_{i,k}(s)$ the third term on the right has already been defined.
\begin{lmm}\label{boundbik}
There exists $K(d)\geq 3$ such that for any loop sequence $s$ with degree sequence $\delta=(\delta_1, \ldots, \delta_n)$
\begin{align}
0\leq b_{i,k}(s)\leq K(d)^{(5+2k)i+\iota(\delta)}|\delta|^{3k}C_{\delta_1-1}\cdots C_{\delta_n-1} \label{ineqbik}
\end{align}
where $C_i$ is the $i^{\text{th}}$ Catalan number. The product of Catalan numbers is assumed to be $1$ when $s=\emptyset$.
\end{lmm}
\begin{proof}
The lower bound trivially follows from the recursive definition \eqref{bik}. The proof of the upper bound is by our usual three fold induction: first on $k$, then on $i$ and then on $\iota(s)$. The number $K=K(d)$ will be chosen later in the proof. First suppose that $k=0$. Since $b_{0,0}(s)$ is either $0$ or $1$ inequality is true for triples $(0, 0, s)$. Fix $i\geq 1$ and suppose inequality holds for triples $(i', 0, s)$ with $i'<i$ and $s\in \cs$. We will use induction on $\iota(s)$ to prove that it also holds for triples $(i, 0, s)$. Since $b_{i, 0}(\emptyset)=0$, it is true for $(i, 0, \emptyset)$. Now, let $s\neq \emptyset$ and suppose the claim is also true for triples $(i, 0, s')$ with $\iota(s')<\iota(s)$. Let $\fs^-_r(s)$ be the set of all loop sequences obtained by applying negative splitting operation to $r^{\text{th}}$ component of $s$. Similarly, define $\fs^+_r(s)$, $\fd^-_r(s)$, $\fd^+_r(s)$, $\fex^-_r(s)$ and $\fex^+_r(s)$. By Lemma \ref{split2} $\iota(s')\leq \iota(s)-3$ for all $s'\in \fs^-(s)$. So by induction hypothesis
\begin{align*}
\frac{1}{|s|}\sum_{s'\in \fs^-(s)}b_{i,0}(s')&=\frac{1}{|s|}\sum_{r=1}^{n}\sum_{s'\in \fs^-_r(s)}b_{i,0}(s')\\
&\leq \frac{1}{|s|}\sum_{r=1}^{n}\sum_{1\leq x\neq y \leq \delta_r}b_{i,0}(l_1, \ldots, l_{r-1},\times_{x,y}^1 l_r, \times_{x,y}^2 l_r, l_{r+1}, \ldots, l_n)\\
&\leq \frac{1}{|s|}\sum_{r=1}^{n}\sum_{1\leq x\neq y \leq \delta_r}K^{5i+\iota(\delta)-3}C_{\delta_r-|x-y|-2}C_{|x-y|-2}\prod\limits_{\substack{1\leq i \leq n\\i\neq r}}C_{\delta_i-1}\\
&\leq \frac{2}{|s|}\sum_{r=1}^{n}\sum_{x=1}^{\delta_r}\sum_{u=2}^{\delta_r-2}K^{5i+\iota(\delta)-3}C_{\delta_r-u-2}C_{u-2}\prod\limits_{\substack{1\leq i \leq n\\i\neq r}}C_{\delta_i-1}\, .
\end{align*}
So, by Lemma~\ref{catalan}
\begin{align*}
\frac{1}{|s|}\sum_{s'\in \fs^-(s)}b_{i,0}(s')&\leq  \frac{2}{|s|}\sum_{r=1}^{n}\sum_{x=1}^{\delta_r}K^{5i+\iota(\delta)-3}C_{\delta_r-3}\prod\limits_{\substack{1\leq i \leq n\\i\neq r}}C_{\delta_i-1}\\
&\leq \frac{2}{|s|}\sum_{r=1}^{n}\delta_rK^{5i+\iota(\delta)-3}C_{\delta_1-1}\cdots C_{\delta_n-1}\\
&=2K^{5i+\iota(\delta)-3}C_{\delta_1-1}\cdots C_{\delta_n-1} \, .
\end{align*}
Similarly, by Lemma \ref{split1} $\iota(s')\leq \iota(s)-1$ for all $s'\in \fs^+(s)$. So, by induction hypothesis
\begin{align*}
\frac{1}{|s|}\sum_{s'\in \fs^+(s)}b_{i,0}(s')&=\frac{1}{|s|}\sum_{r=1}^{n}\sum_{s'\in \fs^+_r(s)}b_{i,0}(s')\\
&\leq \frac{1}{|s|}\sum_{r=1}^{n}\sum_{1\leq x\neq y \leq \delta_r}b_{i,0}(l_1, \ldots, l_{r-1},\times_{x,y}^1 l_r, \times_{x,y}^2 l_r, l_{r+1}, \ldots, l_n)\\
&\leq \frac{1}{|s|}\sum_{r=1}^{n}\sum_{1\leq x\neq y \leq \delta_r}K^{5i+\iota(\delta)-1}C_{\delta_r-|x-y|-1}C_{|x-y|-1}\prod\limits_{\substack{1\leq i \leq n\\i\neq r}}C_{\delta_i-1}\\
&\leq \frac{2}{|s|}\sum_{r=1}^{n}\sum_{x=1}^{\delta_r}\sum_{u=1}^{\delta_r-1}K^{5i+\iota(\delta)-1}C_{\delta_r-u-1}C_{u-1}\prod\limits_{\substack{1\leq i \leq n\\i\neq r}}C_{\delta_i-1}\, .
\end{align*}
By  Lemma~\ref{catalan} 
\begin{align*}
\frac{1}{|s|}\sum_{s'\in \fs^+(s)}b_{i,0}(s')&\leq  \frac{2}{|s|}\sum_{r=1}^{n}\sum_{x=1}^{\delta_r}K^{5i+\iota(\delta)-1}C_{\delta_r-1}\prod\limits_{\substack{1\leq i \leq n\\i\neq r}}C_{\delta_i-1}\\
&\leq \frac{2}{|s|}\sum_{r=1}^{n}\delta_rK^{5i+\iota(\delta)-1}C_{\delta_1-1}\cdots C_{\delta_n-1}\\
&=2K^{5i+\iota(\delta)-1}C_{\delta_1-1}\cdots C_{\delta_n-1}\, .
\end{align*}
By combining above inequalities obtain
\begin{align}
\frac{1}{|s|}\sum_{s'\in \fs(s)}b_{i,0}(s')&\leq (2K^{-3}+2K^{-1})K^{5i+\iota(\delta)}C_{\delta_1-1}\cdots C_{\delta_n-1}\, .\label{d1}
\end{align}
Next, by Lemma \ref{deform} $\iota(s')\leq \iota(s)+4$ for all $s'\in \fd(s)$. Since $|\cp^+(e)|=2(d-1)$ for any edge $e$, each edge of $s$ can be deformed by less than $2d$ plaquettes. This means $|\fd^+_r(s)|\leq 2d\delta_r$. So, by induction hypothesis and inequality \eqref{catalan0}
\begin{align*}
\frac{1}{2|s|}\sum_{s'\in \fd^+(s)}b_{i-1,0}(s')&=\frac{1}{2|s|}\sum_{r=1}^{n}\sum_{s'\in \fd^+_r(s)}b_{i-1,0}(s')\\
&\leq \frac{d}{|s|}\sum_{r=1}^{n} \delta_r K^{5(i-1)+\iota(\delta)+4}C_{\delta_r+3}\prod\limits_{\substack{1\leq i \leq n\\i\neq r}}C_{\delta_i-1}\\
&\leq d K^{5i+\iota(\delta)-1}4^4C_{\delta_r-1}\prod\limits_{\substack{1\leq i \leq n\\i\neq r}}C_{\delta_i-1}\\
&= 256d K^{5i+\iota(\delta)-1}C_{\delta_1-1}\cdots C_{\delta_n-1}\, .
\end{align*}
The same inequality is also true for the sum over negative deformations. Thus
\begin{align}
\frac{1}{2|s|}\sum_{s'\in \fd(s)}b_{i-1,0}(s')&\leq 512d K^{5i+\iota(\delta)-1}C_{\delta_1-1}\cdots C_{\delta_n-1}\, . \label{d2}
\end{align}
By Lemma \ref{expand} $\iota(s')=\iota(s)+3$ for all $s'\in \fex(s)$. Since $|\cp(e)|=2(d-1)$ for all $e$, each edge of $s$ can be expanded by no more than $2d$ plaquettes. Therefore, $|\fex^+_r(s)| \leq 2d\delta_r$. Then, by induction hypothesis and inequality \eqref{catalan0}
\begin{align*}
\frac{1}{2|s|}\sum_{s'\in \fex^+(s)}b_{i-1,0}(s')&=\frac{1}{2|s|}\sum_{r=1}^{n}\sum_{s'\in \fex^+_r(s)}b_{i-1,0}(s')\\
&\leq \frac{d}{|s|}\sum_{r=1}^{n} \delta_r K^{5(i-1)+\iota(\delta)+3}C_{\delta_1-1}\cdots C_{\delta_{r-1}}C_{\delta_r-1}C_3C_{\delta_{r+1}-1}\cdots C_{\delta_n-1}\\
&= 5d K^{5i+\iota(\delta)-2}C_{\delta_1-1}\cdots C_{\delta_n-1}\, .
\end{align*}
The same inequality still holds for the sum over negative expansions. Hence
\begin{align}
\frac{1}{2|s|}\sum_{s'\in \fex(s)}b_{i-1,0}(s')&\leq 10d K^{5i+\iota(\delta)-2}C_{\delta_1-1}\cdots C_{\delta_n-1}\, . \label{d3}
\end{align}
Combining inequalities \eqref{d1}, \eqref{d2}, \eqref{d3} with equation \eqref{bik} we obtain
\begin{align}
|b_{i,0}(s)|&\leq(2K^{-3}+2K^{-1}+512dK^{-1}+ 10dK^{-2}) K^{5i+\iota(\delta)}C_{\delta_1-1}\cdots C_{\delta_n-1}\, .\notag
\end{align}
By choosing $K$ so that the number in the brackets is $\leq 1$ we finish the proof for the base case $k=0$.

Next, fix $k\geq 1$ and suppose inequality \eqref{ineqbik} is true for all triples $(i, k', s)$ with $k'<k$. Since $b_{0, k}(s)=0$ for all $s$, the triples $(0, k, s)$ also satisfy this inequality. Now, fix $i\geq 1$ and suppose further that triples $(i', k, s)$ with $i'<i$ also satisfy \eqref{ineqbik}. We will use induction on $\iota(s)$ to prove the claim for the triples $(i, k, s)$. Since $b_{i, k}(\emptyset)=0$, the base case $s=\emptyset$ is trivially true. So let $s\neq \emptyset$ and assume that the claim holds for the triples $(i, k, s')$ with $\iota(s')<\iota(s)$. By induction hypothesis and since $\ell(s) \leq |s|^2$ 
\begin{align}
\frac{\ell(s)}{|s|}b_{i, k-1}(s) &\leq |s| K^{(5+2k-2)i+\iota(\delta)}|\delta|^{3k-3} C_{\delta_1-1}\cdots C_{\delta_n-1} \notag \\
&\leq K^{(5+2k)i+\iota(\delta)-2}|\delta|^{3k} C_{\delta_1-1}\cdots C_{\delta_n-1} \, . \label{j1}
\end{align}
By Lemma \ref{merger} if $s'\in \fst(s)$, then $|\delta(s')|\leq |\delta|$ and $\iota(s')\leq \iota(s)+1$. So by induction hypothesis
\begin{align*}
&\frac{1}{|s|}\sum_{s'\in \fst^-(s)}b_{i,k-1}(s')\\
&\quad = \frac{1}{|s|}\sum_{1\leq u < v\leq n}\sum_{\substack{x\in A_u,\,y \in B_v\\\textup{or}\\x\in B_u\, , y\in A_v}} b_{i,k-1}(l_1, \ldots, l_{u-1}, l_{u}\ominus_{x, y}l_{v}, l_{u+1}, \ldots, l_{v-1}, l_{v+1}, \ldots, l_n)\\
&\qquad + \frac{1}{|s|}\sum_{1\leq u < v\leq n}\sum_{\substack{x\in A_u,\, y \in B_v\\\textup{or}\\x\in B_u\, , y\in A_v}} b_{i,k-1}(l_1, \ldots, l_{u-1}, l_{u+1}, \ldots, l_{v-1}, l_{v}\ominus_{x, y}l_{u}, l_{v+1}, \ldots, l_n)\\
&\quad \leq \frac{2}{|s|}\sum_{1\leq u < v\leq n}\sum_{\substack{x\in A_u,\, y \in B_v\\\textup{or}\\x\in B_u\, , y\in A_v}} K^{(5+2k-2)i+\iota(\delta)+1}|\delta|^{3k-3} C_{\delta_u+\delta_v-1}\prod_{\substack{1\leq i\leq n\\i\notin\{u, v\}}} C_{\delta_i-1}\\
\end{align*}
By Lemma \ref{catalan1}
\begin{align*}
\frac{1}{|s|}\sum_{s'\in \fst^-(s)}b_{i,k-1}(s') &\leq \frac{2}{|s|}\sum_{1\leq u < v\leq n}\sum_{\substack{x\in A_u,\, y \in B_v\\\textup{or}\\x\in B_u\, , y\in A_v}} K^{(5+2k-2)i+\iota(\delta)+1}|\delta|^{3k-3} (\delta_u+\delta_v)^{2}C_{\delta_1-1}\cdots C_{\delta_n-1}\\
&\quad \leq \frac{2}{|s|}\sum_{1\leq u < v\leq n}\delta_u\delta_v (\delta_u+\delta_v)^{2} K^{(5+2k-2)i+\iota(\delta)+1}|\delta|^{3k-3} C_{\delta_1-1}\cdots C_{\delta_n-1}\\
&\quad \leq K^{(5+2k)i+\iota(\delta)-1}|\delta|^{3k}C_{\delta_1-1}\cdots C_{\delta_n-1}\, .
\end{align*}
The same bound is also valid for the sum over positive mergers. Therefore
\begin{align}
\frac{1}{|s|}\sum_{s'\in \fst(s)}b_{i, k-1}(s')&\leq 2K^{(5+2k)i+\iota(\delta)-1}|\delta|^{3k} C_{\delta_1-1}\cdots C_{\delta_n-1} \, . \label{j2}
\end{align}
By Lemma \ref{split2} if $s'\in \fs^-(s)$, then $|\delta(s')|\leq |\delta|$ and $\iota(s')\leq \iota(s)-3$. Hence by induction hypothesis
\begin{align*}
\frac{1}{|s|}\sum_{s'\in \fs^-(s)}b_{i,k}(s')&=\frac{1}{|s|}\sum_{r=1}^{n}\sum_{s'\in \fs^-_r(s)}b_{i,k}(s')\\
&\leq \frac{1}{|s|}\sum_{r=1}^{n}\sum_{1\leq x\neq y \leq \delta_r}b_{i,k}(l_1, \ldots, l_{r-1},\times_{x,y}^1 l_r, \times_{x,y}^2 l_r, l_{r+1}, \ldots, l_n)\\
&\leq \frac{1}{|s|}\sum_{r=1}^{n}\sum_{1\leq x\neq y \leq \delta_r}K^{(5+2k)i+\iota(\delta)-3}|\delta|^{3k}C_{\delta_r-|x-y|-2}C_{|x-y|-2}\prod\limits_{\substack{1\leq i \leq n\\i\neq r}}C_{\delta_i-1}\\
&\leq \frac{2}{|s|}\sum_{r=1}^{n}\sum_{x=1}^{\delta_r}\sum_{u=2}^{\delta_r-2}K^{(5+2k)i+\iota(\delta)-3}|\delta|^{3k}C_{\delta_r-u-2}C_{u-2}\prod\limits_{\substack{1\leq i \leq n\\i\neq r}}C_{\delta_i-1}\, .
\end{align*}
Applying Lemma~\ref{catalan} we get
\begin{align*}
\frac{1}{|s|}\sum_{s'\in \fs^-(s)}b_{i,k}(s')&\leq  \frac{2}{|s|}\sum_{r=1}^{n}\sum_{x=1}^{\delta_r}K^{(5+2k)i+\iota(\delta)-3}|\delta|^{3k}C_{\delta_r-3}\prod\limits_{\substack{1\leq i \leq n\\i\neq r}}C_{\delta_i-1}\\
&\leq \frac{2}{|s|}\sum_{r=1}^{n}\delta_rK^{(5+2k)i+\iota(\delta)-3}|\delta|^{3k}C_{\delta_1-1}\cdots C_{\delta_n-1}\\
&=2K^{(5+2k)i+\iota(\delta)-3}|\delta|^{3k}C_{\delta_1-1}\cdots C_{\delta_n-1}\, .
\end{align*}
Similarly, by Lemma \ref{split1} if $s'\in \fs^+(s)$, then $|\delta(s')|\leq |\delta|$ and $\iota(s')\leq \iota(s)-1$. So by induction hypothesis
\begin{align*}
\frac{1}{|s|}\sum_{s'\in \fs^+(s)}b_{i,k}(s')&=\frac{1}{|s|}\sum_{r=1}^{n}\sum_{s'\in \fs^+_r(s)}b_{i,k}(s')\\
&\leq \frac{1}{|s|}\sum_{r=1}^{n}\sum_{1\leq x\neq y \leq \delta_r}b_{i,k}(l_1, \ldots, l_{r-1},\times_{x,y}^1 l_r, \times_{x,y}^2 l_r, l_{r+1}, \ldots, l_n)\\
&\leq \frac{1}{|s|}\sum_{r=1}^{n}\sum_{1\leq x\neq y \leq \delta_r}K^{(5+2k)i+\iota(\delta)-1}|\delta|^{3k}C_{\delta_r-|x-y|-1}C_{|x-y|-1}\prod\limits_{\substack{1\leq i \leq n\\i\neq r}}C_{\delta_i-1}\\
&\leq \frac{2}{|s|}\sum_{r=1}^{n}\sum_{x=1}^{\delta_r}\sum_{u=1}^{\delta_r-1}K^{(5+2k)i+\iota(\delta)-1}|\delta|^{3k}C_{\delta_r-u-1}C_{u-1}\prod\limits_{\substack{1\leq i \leq n\\i\neq r}}C_{\delta_i-1}\, .
\end{align*}
By Lemma~\ref{catalan} we obtain 
\begin{align*}
\frac{1}{|s|}\sum_{s'\in \fs^+(s)}b_{i,k}(s')&\leq  \frac{2}{|s|}\sum_{r=1}^{n}\sum_{x=1}^{\delta_r}K^{(5+2k)i+\iota(\delta)-1}|\delta|^{3k}C_{\delta_r-1}\prod\limits_{\substack{1\leq i \leq n\\i\neq r}}C_{\delta_i-1}\\
&=\frac{2}{|s|}\sum_{r=1}^{n}\delta_rK^{5i+\iota(\delta)-1}|\delta|^{3k}C_{\delta_1-1}\cdots C_{\delta_n-1}\\
&=2K^{5i+\iota(\delta)-1}|\delta|^{3k}C_{\delta_1-1}\cdots C_{\delta_n-1}\, .
\end{align*}
By combining above two displays we get
\begin{align}
\frac{1}{|s|}\sum_{s'\in \fs(s)}b_{i,k}(s')&\leq (2K^{-3}+2K^{-1})K^{(5+2k)i+\iota(\delta)}|\delta|^{3k}C_{\delta_1-1}\cdots C_{\delta_n-1}\, .\label{j3}
\end{align}
By Lemma \ref{deform} if  $s'\in \fd(s)$, then $|\delta(s')|\leq |\delta|+4$ and $\iota(s')\leq \iota(s)+4$. Since $|\fd^+_r(s)|\leq 2d\delta_r$, by induction hypothesis and inequality \eqref{catalan0} we get
\begin{align*}
\frac{1}{2|s|}\sum_{s'\in \fd^+(s)}b_{i-1,k}(s')&=\frac{1}{2|s|}\sum_{r=1}^{n}\sum_{s'\in \fd^+_r(s)}b_{i-1,k}(s')\\
&\leq \frac{d}{|s|}\sum_{r=1}^{n} \delta_r K^{(5+2k)(i-1)+\iota(\delta)+4}(|\delta|+4)^{3k}C_{\delta_r+3}\prod\limits_{\substack{1\leq i \leq n\\i\neq r}}C_{\delta_i-1}\\
&\leq  \frac{d}{|s|}\sum_{r=1}^{n} \delta_r K^{(5+2k)i+\iota(\delta)-2k-1}(2|\delta|)^{3k}4^4C_{\delta_r-1}\prod\limits_{\substack{1\leq i \leq n\\i\neq r}}C_{\delta_i-1}\\
&\leq 256d K^{(5+2k)i+\iota(\delta)-1}|\delta|^{3k}C_{\delta_1-1}\cdots C_{\delta_n-1}\, .
\end{align*}
The same bound also holds for the sum over negative deformations. Hence
\begin{align}
\frac{1}{2|s|}\sum_{s'\in \fd(s)}b_{i-1,0}(s')&\leq 512d K^{(5+2k)i+\iota(\delta)-1}|\delta|^{3k}C_{\delta_1-1}\cdots C_{\delta_n-1}\, . \label{j4}
\end{align}
By Lemma \ref{expand} if $s'\in \fex(s)$, then $|\delta(s')|=|\delta|+4$ and $\iota(s')=\iota(s)+3$. Since $|\fex^+_r(s)|\leq 2d\delta_r$, by induction hypothesis and inequality \eqref{catalan0}
\begin{align*}
&\frac{1}{2|s|}\sum_{s'\in \fex^+(s)}b_{i-1,k}(s')=\frac{1}{2|s|}\sum_{r=1}^{n}\sum_{s'\in \fex^+_r(s)}b_{i-1,k}(s')\\
&\qquad\leq \frac{d}{|s|}\sum_{r=1}^{n}\delta_r K^{(5+2k)(i-1)+\iota(\delta)+3}(|\delta|+4)^{3k}C_{\delta_1-1}\cdots C_{\delta_{r-1}}C_{\delta_r-1}C_3C_{\delta_{r+1}-1}\cdots C_{\delta_n-1}\\
&\qquad = 5d K^{(5+2k)i+\iota(\delta)-2k-2}(2|\delta|)^{3k}C_{\delta_1-1}\cdots C_{\delta_n-1}\\
&\qquad \leq 5d K^{(5+2k)i+\iota(\delta)-2}|\delta|^{3k}C_{\delta_1-1}\cdots C_{\delta_n-1}\, .
\end{align*}
The same inequality still holds for the sum over negative expansions. Hence
\begin{align}
\frac{1}{2|s|}\sum_{s'\in \fex(s)}b_{i-1,k}(s')&\leq 10d K^{(5+2k)i+\iota(\delta)-2}|\delta|^{3k}C_{\delta_1-1}\ldots C_{\delta_n-1}\, . \label{j5}
\end{align}
Combining inequalities \eqref{j1}, \eqref{j2}, \eqref{j3}, \eqref{j4} and \eqref{j5} with equation \eqref{bik} we get
\begin{align}
|b_{i,k}(s)|&\leq(2K^{-3}+K^{-2}+4K^{-1}+512dK^{-1}+ 10dK^{-2}) K^{(5+2k)i+\iota(\delta)}|\delta|^{3k}C_{\delta_1-1}\cdots C_{\delta_n-1}\, .\notag
\end{align}
We finish the proof by choosing $K$ so that the number in the brackets is $\leq 1$.
\end{proof}
\begin{lmm}\label{siklmm}
For any loop sequence $s$ and non-negative integers $i$ and $k$ let $\mx_{i, k}(s)$ be as in Section \ref{string} and $b_{i,k}(s)$ as in \eqref{bik}. Then
\begin{align}
\sum_{X\in \mx_{i,k}(s)}|w_{\beta}(X)|=b_{i,k}(s)|\beta|^{i} \, . \label{sik}
\end{align}
\end{lmm}
\begin{proof}
Define 
\begin{align*}
S_{\beta, i, k}(s):=\sum_{X\in \mx_{i,k}(s)}|w_{\beta}(X)|\, . 
\end{align*}
Note that by Lemma \ref{fintraj} the set $\mx_{i,k}(s)$ is finite and therefore $S_{\beta, i, k}(s)$ is well-defined.

The proof is again by three fold induction: first on $k$, then on $i$ and then on $\iota(s)$. For the base case of induction assume $k=0$. Note that any trajectory in $\mx_{0,0}(s)$ can contain only splitting operations. By Lemma \ref{iota2} we know that a trajectory of non-null loop sequence that contains only splitting operations does not vanish. Thus $\mx_{0, 0}(s)$ is empty if $s\neq \emptyset$ and $\mx_{0, 0}(\emptyset)$ contains only the null trajectory. This means both sides of \eqref{sik} are zero when $s\neq \emptyset$) and one when $s=\emptyset$. Now, fix $i\geq 1$ and suppose the claim holds for all triples $(i', 0, s)$ with $i'<i$. We will use induction on $\iota(s)$ to show that it also holds for the triples $(i, 0, s)$. Since $\mx_{i,0}(\emptyset)$ is empty and $\beta_{i, 0}(\emptyset)=0$, when $s=\emptyset$ both sides of \eqref{sik} are zero. So we can assume that $s\neq \emptyset$ and the claim also holds for the triples $(i, 0, s')$ with $\iota(s')<\iota(s)$. Note that for any $X\in \mx_{i, 0}(s)$ we have $X=(s, X')$ where either $X'\in \mx_{i, 0}(s')$ for some $s'\in\fs(s)$ or $X'\in \mx_{i-1, 0}(s')$ for some $s'\in\fd(s)$ or $X'\in \mx_{i-1, 0}(s')$ for some $s'\in\fex(s)$. So
\begin{align*}
S_{\beta, i, 0}(s)=\sum_{s'\in \fs(s)}\sum_{\substack{X=(s, X')\\X'\in \mx_{i, 0}(s')}}|w_{\beta}(X)|+\sum_{s'\in \fd(s)}\sum_{\substack{X=(s, X')\\X'\in \mx_{i-1, 0}(s')}}|w_{\beta}(X)|+\sum_{s'\in \fex(s)}\sum_{\substack{X=(s, X')\\X'\in \mx_{i-1, 0}(s')}}|w_{\beta}(X)|\, .
\end{align*}
By definition of the weight of a trajectory, induction hypothesis and Lemma \ref{iota2}
\begin{align*}
\sum_{s'\in \fs(s)}\sum_{\substack{X=(s, X')\\X'\in \mx_{i, 0}(s')}}|w_{\beta}(X)|=\frac{1}{|s|}\sum_{s'\in \fs(s)}\sum_{X'\in \mx_{i, 0}(s')}|w_{\beta}(X')|=\frac{1}{|s|}\sum_{s'\in \fs(s)}b_{i,0}(s')|\beta|^{i}\, .
\end{align*}
Similarly, 
\begin{align*}
\sum_{s'\in \fd(s)}\sum_{\substack{X=(s, X')\\X'\in \mx_{i, 0}(s')}}|w_{\beta}(X)|=\frac{|\beta|}{2|s|}\sum_{s'\in \fd(s)}\sum_{X'\in \mx_{i-1, 0}(s')}|w_{\beta}(X')|=\frac{|\beta|}{2|s|}\sum_{s'\in \fd(s)}b_{i-1,0}(s')|\beta|^{i-1}\, ,
\end{align*}
and
\begin{align*}
\sum_{s'\in \fex(s)}\sum_{\substack{X=(s, X')\\X'\in \mx_{i, 0}(s')}}|w_{\beta}(X)|=\frac{|\beta|}{2|s|}\sum_{s'\in \fex(s)}\sum_{X'\in \mx_{i-1, 0}(s')}|w_{\beta}(X')|=\frac{|\beta|}{2|s|}\sum_{s'\in \fex(s)}b_{i-1,0}(s')|\beta|^{i-1}\, .
\end{align*}
Putting together above four displays we get
\begin{align*}
S_{\beta, i, 0}(s)=\biggl(\frac{1}{|s|}\sum_{s'\in \fs(s)}b_{i,0}(s')+\frac{1}{2|s|}\sum_{s'\in \fd(s)}b_{i-1,0}(s')+\frac{1}{2|s|}\sum_{s'\in \fex(s)}b_{i-1,0}(s')\biggr)|\beta|^{i}=b_{i, 0}(s)|\beta|^{i}\, .
\end{align*}
This finishes the proof of the base case $k=0$.

Next, fix $k\geq 1$ and assume equality \eqref{sik} is true for all triples $(i, k', s)$ with $k'<k$. First suppose that $i=0$. We will use induction on $\iota(s)$ to prove that the claim holds for all triples $(0, k, s)$. Since $\mx_{0, k}(\emptyset)$ is empty and $b_{0, k}(\emptyset)=0$, both sides of \eqref{sik} are zero when $i=0$ and $s = \emptyset$. This proves equality \eqref{sik} for the triple $(0, k, \emptyset)$. Now, let $s\neq \emptyset$ and  assume that the claim holds for all triples $(0, k, s')$ with $\iota(s')<\iota(s)$. Note that any $X\in \mx_{0, k}(s)$ can be written as $X=(s, X')$ where either $X'\in \mx_{0, k-1}(s)$ or $X'\in \mx_{0, k-1}(s')$ for some $s'\in\fst(s)$ or $X'\in \mx_{0, k}(s')$ for some $s'\in\fs(s)$. So
\begin{align*}
S_{\beta, 0, k}(s)=\sum_{\substack{X=(s, X')\\X'\in \mx_{0, k-1}(s)}}|w_{\beta}(X)|+\sum_{s'\in \fst(s)}\sum_{\substack{X=(s, X')\\X'\in \mx_{0, k-1}(s')}}|w_{\beta}(X)|+\sum_{s'\in \fs(s)}\sum_{\substack{X=(s, X')\\X'\in \mx_{0, k}(s')}}|w_{\beta}(X)|\, .
\end{align*}
By definition of the weight of a trajectory and induction hypothesis
\begin{align*}
\sum_{\substack{X=(s, X')\\X'\in \mx_{0, k-1}(s)}}|w_{\beta}(X)|=\frac{\ell(s)}{|s|}\sum_{X'\in \mx_{0, k-1}(s)}|w_{\beta}(X')|=\frac{\ell(s)}{|s|}b_{0, k-1}(s) \, .
\end{align*}
Similarly, 
\begin{align*}
\sum_{s'\in \fst(s)}\sum_{\substack{X=(s, X')\\X'\in \mx_{0, k-1}(s')}}|w_{\beta}(X)|=\frac{1}{|s|}\sum_{s'\in \fst(s)}\sum_{X'\in \mx_{0, k-1}(s')}|w_{\beta}(X')|=\frac{1}{|s|}\sum_{s'\in \fst(s)}b_{0,k-1}(s')\, ,
\end{align*}
and
\begin{align*}
\sum_{s'\in \fs(s)}\sum_{\substack{X=(s, X')\\X'\in \mx_{0, k}(s')}}|w_{\beta}(X)|=\frac{1}{|s|}\sum_{s'\in \fs(s)}\sum_{X'\in \mx_{0, k}(s')}|w_{\beta}(X')|=\frac{1}{|s|}\sum_{s'\in \fs(s)}b_{0,k}(s')\, .
\end{align*}
Putting together above four displays we get
\begin{align*}
S_{\beta, 0, k}(s)=\frac{\ell(s)}{|s|}b_{0, k-1}(s)+\frac{1}{|s|}\sum_{s'\in \fst(s)}b_{0,k-1}(s')+\frac{1}{|s|}\sum_{s'\in \fs(s)}b_{0,k}(s')=b_{0, k}(s)|\, .
\end{align*}
This finishes the proof for the triples $(0, k, s)$. 

Now, fix $i\geq 1$ and assume that equality \eqref{sik} is also valid for all triples $(i', k, s)$ where $i'<i$. Now we will use induction on $\iota(s)$ to prove the claim for triples $(i, k, s)$. If $s=\emptyset$, then both sides of \eqref{sik} are zero. So let $s\neq \emptyset$ and assume that the claim also holds for triples $(i, k, s')$ with $\iota(s')<\iota(s)$. Note that any $X\in \mx_{i, k}(s)$ can be written as $X=(s, X')$ where either $X'\in \mx_{i, k-1}(s)$ or $X'\in \mx_{i, k-1}(s')$ for some $s'\in\fst(s)$ or $X'\in \mx_{i, k}(s')$ for some $s'\in\fs(s)$ or $X'\in \mx_{i-1, k}(s')$ for some $s'\in\fd(s)$ or $X'\in \mx_{i-1, k}(s')$ for some $s'\in\fex(s)$. So
\begin{align*}
S_{\beta, i, k}(s)&=\sum_{\substack{X=(s, X')\\X'\in \mx_{i, k-1}(s)}}|w_{\beta}(X)|+\sum_{s'\in \fst(s)}\sum_{\substack{X=(s, X')\\X'\in \mx_{i, k-1}(s')}}|w_{\beta}(X)|+\sum_{s'\in \fs(s)}\sum_{\substack{X=(s, X')\\X'\in \mx_{i, k}(s')}}|w_{\beta}(X)|\\
&\qquad \qquad +\sum_{s'\in \fd(s)}\sum_{\substack{X=(s, X')\\X'\in \mx_{i-1, k}(s')}}|w_{\beta}(X)|+\sum_{s'\in \fex(s)}\sum_{\substack{X=(s, X')\\X'\in \mx_{i-1, k}(s')}}|w_{\beta}(X)|\, .
\end{align*}
By definition of the weight of a trajectory and induction hypothesis
\begin{align*}
\sum_{\substack{X=(s, X')\\X'\in \mx_{i, k-1}(s)}}|w_{\beta}(X)|=\frac{\ell(s)}{|s|}\sum_{X'\in \mx_{i, k-1}(s)}|w_{\beta}(X')|=\frac{\ell(s)}{|s|}b_{i, k-1}(s)|\beta|^{i} \, .
\end{align*}
Similarly, 
\begin{align*}
\sum_{s'\in \fst(s)}\sum_{\substack{X=(s, X')\\X'\in \mx_{i, k-1}(s')}}|w_{\beta}(X)|&=\frac{1}{|s|}\sum_{s'\in \fst(s)}\sum_{X'\in \mx_{i, k-1}(s')}|w_{\beta}(X')|=\frac{1}{|s|}\sum_{s'\in \fs(s)}b_{i,k-1}(s')|\beta|^{i}\, ,\\
\sum_{s'\in \fs(s)}\sum_{\substack{X=(s, X')\\X'\in \mx_{i, k}(s')}}|w_{\beta}(X)|&=\frac{1}{|s|}\sum_{s'\in \fs(s)}\sum_{X'\in \mx_{i, k}(s')}|w_{\beta}(X')|=\frac{1}{|s|}\sum_{s'\in \fs(s)}b_{i,k}(s')|\beta|^{i}\, ,\\
\sum_{s'\in \fd(s)}\sum_{\substack{X=(s, X')\\X'\in \mx_{i-1, k}(s')}}|w_{\beta}(X)|&=\frac{|\beta|}{2|s|}\sum_{s'\in \fd(s)}\sum_{X'\in \mx_{i-1, k}(s')}|w_{\beta}(X')|=\frac{|\beta|}{2|s|}\sum_{s'\in \fd(s)}b_{i-1,k}(s')|\beta|^{i-1}\, ,
\end{align*}
and
\begin{align*}
\sum_{s'\in \fex(s)}\sum_{\substack{X=(s, X')\\X'\in \mx_{i-1, k}(s')}}|w_{\beta}(X)|&=\frac{|\beta|}{2|s|}\sum_{s'\in \fex(s)}\sum_{X'\in \mx_{i-1, k}(s')}|w_{\beta}(X')|=\frac{|\beta|}{2|s|}\sum_{s'\in \fex(s)}b_{i-1,k}(s')|\beta|^{i-1}\, .
\end{align*}
Putting together above six displays we get
\begin{align*}
S_{\beta, i, k}(s)&=\biggl(\frac{\ell(s)}{|s|}b_{i, k-1}(s)+\frac{1}{|s|}\sum_{s'\in \fst(s)}b_{i,k-1}(s')+\frac{1}{|s|}\sum_{s'\in \fs(s)}b_{i,k}(s')\\
&\qquad\qquad +\frac{1}{2|s|}\sum_{s'\in \fd(s)}b_{i-1,k}(s')+\frac{1}{2|s|}\sum_{s'\in \fex(s)}b_{i-1,k}(s')\biggr)|\beta|^{i}\\
&=b_{i, k}(s)|\beta|^{i}\, .
\end{align*}
This completes the proof of this lemma.
\end{proof}
Now we are ready to prove absolute convergence of the sum \eqref{sumxk}.
\begin{proof}[Proof of Theorem \ref{absconvmxk}]
Note that for any $k$ any vanishing trajectory $X \in \mx_k(s)$ is in exactly one of $\mx_{i,k}(s)$ for $i=0, 1, \ldots$. Hence
\[
\sum_{X\in \mx_{k}(s)}|w_{\beta}(X)|=\sum_{i=0}^{\infty}\sum_{X\in \mx_{i,k}(s)}|w_{\beta}(X)|=\sum_{i=0}^{\infty}b_{i, k}(s)|\beta^{i}|\,.
\]
By Lemma \ref{boundbik} this sum is convergent for all sufficiently small $|\beta|$.
\end{proof}
\section{Gauge-string duality}
The goal of this section is to provide representation of loop functions $f_{2k}(s)$ as a sum of weights of vanishing trajectories.
\begin{lmm}\label{tiklmm}
For any loop sequence $s$ and non-negative integers $i$ and $k$ let $\mx_{i, k}(s)$ be as in Section \ref{string} and $a_{i,k}(s)$ as in \eqref{symaik}. Then
\begin{align}
\sum_{X\in \mx_{i,k}(s)}w_{\beta}(X)=a_{i,k}(s)\beta^{i}\, . \label{tik}
\end{align}
\end{lmm}
\begin{proof}
Define 
\begin{align*}
T_{\beta, i, k}(s):=\sum_{X\in \mx_{i,k}(s)}w_{\beta}(X)\, . 
\end{align*}
$T_{\beta, i, k}(s)$ is well-defined because by Lemma \ref{fintraj} the set $\mx_{i,k}(s)$ is finite.

The proof is by our usual three fold induction. For the base case of induction assume $k=0$. As noted in the proof of Lemma \ref{siklmm} the set $\mx_{0, 0}(s)$ is empty if $s\neq \emptyset$ and $\mx_{0, 0}(\emptyset)$ contains only the null trajectory. Hence both sides of \eqref{tik} are either zero when $s \neq \emptyset$ and one when $s = \emptyset$. Next, fix $i\geq 1$ and suppose the claim holds for all triples $(i', 0, s)$ with $i'<i$. We will use induction on $\iota(s)$ to prove that equality \eqref{tik} also holds for the triples $(i, 0, s)$. Since $\mx_{i,0}(\emptyset)$ is an empty set and $a_{i, 0}(\emptyset)=0$, if $k=0$ and $s=\emptyset$, then both sides of \eqref{tik} are zero. Thus, the claim is true for triple $(i, 0, \emptyset)$. Now, let $s\neq \emptyset$ and suppose that the claim also holds for the triples $(i, 0, s')$ with $\iota(s')<\iota(s)$. Note that any $X\in \mx_{i, 0}(s)$ can be written as $X=(s, X')$ where either $X'\in \mx_{i, 0}(s')$ for some $s'\in\fs(s)$ or $X'\in \mx_{i-1, 0}(s')$ for some $s'\in\fd(s)$ or $X'\in \mx_{i-1, 0}(s')$ for some $s'\in\fex(s)$. So
\begin{align*}
T_{\beta, i, 0}(s)=\sum_{s'\in \fs(s)}\sum_{\substack{X=(s, X')\\X'\in \mx_{i, 0}(s')}}w_{\beta}(X)+\sum_{s'\in \fd(s)}\sum_{\substack{X=(s, X')\\X'\in \mx_{i-1, 0}(s')}}w_{\beta}(X)+\sum_{s'\in \fex(s)}\sum_{\substack{X=(s, X')\\X'\in \mx_{i-1, 0}(s')}}w_{\beta}(X)\, .
\end{align*}
By definition of the weight of a trajectory, induction hypothesis and Lemma \ref{iota2}
\begin{align*}
\sum_{s'\in \fs(s)}\sum_{\substack{X=(s, X')\\X'\in \mx_{i, 0}(s')}}w_{\beta}(X)&=\frac{1}{|s|}\sum_{s'\in \fs^-(s)}\sum_{X'\in \mx_{i, 0}(s')}w_{\beta}(X')-\frac{1}{|s|}\sum_{s'\in \fs^+(s)}\sum_{X'\in \mx_{i, 0}(s')}w_{\beta}(X')\\
&=\frac{1}{|s|}\sum_{s'\in \fs^-(s)}a_{i,0}(s')\beta^{i}-\frac{1}{|s|}\sum_{s'\in \fs^+(s)}a_{i,0}(s')\beta^{i}\, .
\end{align*}
Similarly, 
\begin{align*}
\sum_{s'\in \fd(s)}\sum_{\substack{X=(s, X')\\X'\in \mx_{i, 0}(s')}}w_{\beta}(X)&=\frac{\beta}{2|s|}\sum_{s'\in \fd^-(s)}\sum_{X'\in \mx_{i-1, 0}(s')}w_{\beta}(X')-\frac{\beta}{2|s|}\sum_{s'\in \fd^+(s)}\sum_{X'\in \mx_{i-1, 0}(s')}w_{\beta}(X')\\
&=\frac{\beta}{2|s|}\sum_{s'\in \fd^-(s)}a_{i-1,0}(s')\beta^{i-1}-\frac{\beta}{2|s|}\sum_{s'\in \fd^+(s)}a_{i-1,0}(s')\beta^{i-1}\, ,
\end{align*}
and
\begin{align*}
\sum_{s'\in \fex(s)}\sum_{\substack{X=(s, X')\\X'\in \mx_{i, 0}(s')}}w_{\beta}(X)&=\frac{\beta}{2|s|}\sum_{s'\in \fex^-(s)}\sum_{X'\in \mx_{i-1, 0}(s')}w_{\beta}(X')-\frac{\beta}{2|s|}\sum_{s'\in \fex^+(s)}\sum_{X'\in \mx_{i-1, 0}(s')}w_{\beta}(X')\\
&=\frac{\beta}{2|s|}\sum_{s'\in \fex^-(s)}a_{i-1,0}(s')\beta^{i-1}-\frac{\beta}{2|s|}\sum_{s'\in \fex^+(s)}a_{i-1,0}(s')\beta^{i-1}\, .
\end{align*}
Putting together above four displays we get
\begin{align*}
T_{\beta, i, 0}(s)&=\biggl(\frac{1}{|s|}\sum_{s'\in \fs^-(s)}a_{i,0}(s')-\frac{1}{|s|}\sum_{s'\in \fs^-(s)}a_{i,0}(s')\\
&\quad + \frac{1}{2|s|}\sum_{s'\in \fd^-(s)}a_{i-1,0}(s')-\frac{1}{2|s|}\sum_{s'\in \fd^-(s)}a_{i-1,0}(s')\\
&\quad + \frac{1}{2|s|}\sum_{s'\in \fex^-(s)}a_{i-1,0}(s')-\frac{1}{2|s|}\sum_{s'\in \fex^+(s)}a_{i-1,0}(s')\biggr)\beta^{i}\\
&=a_{i, 0}(s)\beta^{i}\, .
\end{align*}
This completes the proof of the base case $k=0$.

Now, fix $k\geq 1$ and suppose that the claim is true for all triples $(i, k', s)$ with $k'<k$.  To prove for the triples $(i, k, s)$ first suppose that $i=0$. We will use induction on $\iota(s)$ to show that the claim is also true for triples $(0, k, s)$. Since $X_{0, k}(\emptyset)$ is an empty set and $a_{0, k}(\emptyset)=0$, both sides of \eqref{tik} are zero. This proves the claim for the triple $(0, k, \emptyset)$. Next, fix $s\neq \emptyset$ and assume that the claim holds for all triples $(0, k, s')$ with $\iota(s')<\iota(s)$. Note that any $X\in \mx_{0, k}(s)$ can be written as $X=(s, X')$ where either $X'\in \mx_{0, k-1}(s)$ or $X'\in \mx_{0, k-1}(s')$ for some $s'\in\fst(s)$ or $X'\in \mx_{0, k}(s')$ for some $s'\in\fs(s)$. Therefore
\begin{align*}
T_{\beta, 0, k}(s)=\sum_{\substack{X=(s, X')\\X'\in \mx_{0, k-1}(s)}}w_{\beta}(X)+\sum_{s'\in \fst(s)}\sum_{\substack{X=(s, X')\\X'\in \mx_{0, k-1}(s')}}w_{\beta}(X)+\sum_{s'\in \fs(s)}\sum_{\substack{X=(s, X')\\X'\in \mx_{0, k}(s')}}w_{\beta}(X)\, .
\end{align*}
By definition of the weight of a trajectory and induction hypothesis
\begin{align*}
\sum_{\substack{X=(s, X')\\X'\in \mx_{0, k-1}(s)}}w_{\beta}(X)&=\frac{\ell(s)}{|s|}\sum_{X'\in \mx_{0, k-1}(s)}w_{\beta}(X')=\frac{\ell(s)}{|s|}a_{0, k-1}(s) \, .
\end{align*}
Similarly, 
\begin{align*}
\sum_{s'\in \fst(s)}\sum_{\substack{X=(s, X')\\X'\in \mx_{0, k-1}(s')}}w_{\beta}(X)&=\frac{1}{|s|}\sum_{s'\in \fst^-(s)}\sum_{X'\in \mx_{0, k-1}(s')}w_{\beta}(X')-\frac{1}{|s|}\sum_{s'\in \fst^+(s)}\sum_{X'\in \mx_{0, k-1}(s')}w_{\beta}(X')\\
&=\frac{1}{|s|}\sum_{s'\in \fst^-(s)}a_{0,k-1}(s')-\frac{1}{|s|}\sum_{s'\in \fst^+(s)}a_{0,k-1}(s')\, ,
\end{align*}
and
\begin{align*}
\sum_{s'\in \fs(s)}\sum_{\substack{X=(s, X')\\X'\in \mx_{0, k}(s')}}w_{\beta}(X)&=\frac{1}{|s|}\sum_{s'\in \fs^-(s)}\sum_{X'\in \mx_{0, k}(s')}w_{\beta}(X')-\frac{1}{|s|}\sum_{s'\in \fs^+(s)}\sum_{X'\in \mx_{0, k}(s')}w_{\beta}(X')\\
&=\frac{1}{|s|}\sum_{s'\in \fs^-(s)}a_{0,k}(s')-\frac{1}{|s|}\sum_{s'\in \fs^+(s)}a_{0,k}(s')\, .
\end{align*}
By combining above four displays we get
\begin{align*}
T_{\beta, 0, k}(s)&=\frac{\ell(s)}{|s|}a_{0, k-1}(s)+\frac{1}{|s|}\sum_{s'\in \fst^-(s)}a_{0,k-1}(s')-\frac{1}{|s|}\sum_{s'\in \fst^+(s)}a_{0,k-1}(s')\\
&\qquad \qquad +\frac{1}{|s|}\sum_{s'\in \fs^-(s)}a_{0,k}(s')-\frac{1}{|s|}\sum_{s'\in \fs^+(s)}a_{0,k}(s')\\
&=a_{0, k}(s)\, .
\end{align*}
So the claim is true for the triples $(0, k, s)$. 

Next, fix $i\geq 1$ and assume that the claim is also true for all triples $(i', k, s)$ with $i'<i$. We will use induction on $\iota(s)$ to prove the claim for triples $(i, k, s)$. If $s=\emptyset$, then both sides of \eqref{tik} are zero. Fix $s\neq \emptyset$ and assume that the claim also holds for triples $(i, k, s')$ with $\iota(s')<\iota(s)$. Note that any $X\in \mx_{i, k}(s)$ can be written as $X=(s, X')$ where either $X'\in \mx_{i, k-1}(s)$ or $X'\in \mx_{i, k-1}(s')$ for some $s'\in\fst(s)$ or $X'\in \mx_{i, k}(s')$ for some $s'\in\fs(s)$ or $X'\in \mx_{i-1, k}(s')$ for some $s'\in\fd(s)$ or $X'\in \mx_{i-1, k}(s')$ for some $s'\in\fex(s)$. Thus
\begin{align*}
T_{\beta, i, k}(s)&=\sum_{\substack{X=(s, X')\\X'\in \mx_{i, k-1}(s)}}w_{\beta}(X)+\sum_{s'\in \fst(s)}\sum_{\substack{X=(s, X')\\X'\in \mx_{i, k-1}(s')}}w_{\beta}(X)+\sum_{s'\in \fs(s)}\sum_{\substack{X=(s, X')\\X'\in \mx_{i, k}(s')}}w_{\beta}(X)\\
&\qquad \qquad +\sum_{s'\in \fd(s)}\sum_{\substack{X=(s, X')\\X'\in \mx_{i-1, k}(s')}}w_{\beta}(X)+\sum_{s'\in \fex(s)}\sum_{\substack{X=(s, X')\\X'\in \mx_{i-1, k}(s')}}w_{\beta}(X)\, .
\end{align*}
By definition of the weight of a trajectory and induction hypothesis
\begin{align*}
\sum_{\substack{X=(s, X')\\X'\in \mx_{i, k-1}(s)}}w_{\beta}(X)=\frac{\ell(s)}{|s|}\sum_{X'\in \mx_{i, k-1}(s)}w_{\beta}(X')=\frac{\ell(s)}{|s|}a_{i, k-1}(s)\beta^{i} \, .
\end{align*}
Similarly, 
\begin{align*}
\sum_{s'\in \fst(s)}\sum_{\substack{X=(s, X')\\X'\in \mx_{i, k-1}(s')}}w_{\beta}(X)&=\frac{1}{|s|}\sum_{s'\in \fst^-(s)}\sum_{X'\in \mx_{i, k-1}(s')}w_{\beta}(X')-\frac{1}{|s|}\sum_{s'\in \fst^+(s)}\sum_{X'\in \mx_{i, k-1}(s')}w_{\beta}(X')\\
&=\frac{1}{|s|}\sum_{s'\in \fst^-(s)}a_{i,k-1}(s')\beta^{i}-\frac{1}{|s|}\sum_{s'\in \fst^+(s)}a_{i,k-1}(s')\beta^{i}\, ,
\end{align*}
\begin{align*}
\sum_{s'\in \fs(s)}\sum_{\substack{X=(s, X')\\X'\in \mx_{i, k}(s')}}w_{\beta}(X)&=\frac{1}{|s|}\sum_{s'\in \fs^-(s)}\sum_{X'\in \mx_{i, k}(s')}w_{\beta}(X')-\frac{1}{|s|}\sum_{s'\in \fs^+(s)}\sum_{X'\in \mx_{i, k}(s')}w_{\beta}(X')\\
&=\frac{1}{|s|}\sum_{s'\in \fs^-(s)}a_{i,k}(s')\beta^{i}-\frac{1}{|s|}\sum_{s'\in \fs^+(s)}a_{i,k}(s')\beta^{i}\, ,
\end{align*}
\begin{align*}
\sum_{s'\in \fd(s)}\sum_{\substack{X=(s, X')\\X'\in \mx_{i-1, k}(s')}}w_{\beta}(X)&=\frac{\beta}{2|s|}\sum_{s'\in \fd^-(s)}\sum_{X'\in \mx_{i-1, k}(s')}w_{\beta}(X')-\frac{\beta}{2|s|}\sum_{s'\in \fd^+(s)}\sum_{X'\in \mx_{i-1, k}(s')}w_{\beta}(X')\\
&=\frac{\beta}{2|s|}\sum_{s'\in \fd^-(s)}a_{i-1,k}(s')\beta^{i-1}-\frac{\beta}{2|s|}\sum_{s'\in \fd^+(s)}a_{i-1,k}(s')\beta^{i-1}\, ,
\end{align*}
and
\begin{align*}
\sum_{s'\in \fex(s)}\sum_{\substack{X=(s, X')\\X'\in \mx_{i-1, k}(s')}}w_{\beta}(X)&=\frac{\beta}{2|s|}\sum_{s'\in \fex^-(s)}\sum_{X'\in \mx_{i-1, k}(s')}w_{\beta}(X')-\frac{\beta}{2|s|}\sum_{s'\in \fex^+(s)}\sum_{X'\in \mx_{i-1, k}(s')}w_{\beta}(X')\\
&=\frac{\beta}{2|s|}\sum_{s'\in \fex^-(s)}a_{i-1,k}(s')\beta^{i-1}-\frac{\beta}{2|s|}\sum_{s'\in \fex^+(s)}a_{i-1,k}(s')\beta^{i-1}\, .
\end{align*}
By combining above six displays we obtain
\begin{align*}
T_{\beta, i, k}(s)&=\biggl(\frac{\ell(s)}{|s|}a_{i, k-1}(s)+\frac{1}{|s|}\sum_{s'\in \fst^-(s)}a_{i,k-1}(s')-\frac{1}{|s|}\sum_{s'\in \fst^+(s)}a_{i,k-1}(s')\\
&\qquad + \frac{1}{|s|}\sum_{s'\in \fs^-(s)}a_{i,k}(s') - \frac{1}{|s|}\sum_{s'\in \fs^+(s)}a_{i,k}(s')\\
&\qquad + \frac{1}{2|s|}\sum_{s'\in \fd^-(s)}a_{i-1,k}(s') - \frac{1}{2|s|}\sum_{s'\in \fd^+(s)}a_{i-1,k}(s')\\
&\qquad + \frac{1}{2|s|}\sum_{s'\in \fex^-(s)}a_{i-1,k}(s') - \frac{1}{2|s|}\sum_{s'\in \fex^+(s)}a_{i-1,k}(s')\biggr)\beta^{i}\\
&=a_{i, k}(s)\beta^{i}\, .
\end{align*}
This finishes the induction and the proof of this lemma.
\end{proof}
Now we are ready to prove part (i) of Theorem \ref{maintheorem}.
\begin{proof}[Proof of part (i) of Theorem \ref{maintheorem}]
Fix $k$. Since any vanishing trajectory $X \in \mx_k(s)$ is in exactly one of $\mx_{i,k}(s)$ for $i=0, 1, \ldots$. So, by Theorem \ref{seriesfk} and Lemma \ref{tiklmm}
\[
\sum_{X\in \mx_{k}(s)}w_{\beta}(X)=\sum_{i=0}^{\infty}\sum_{X\in \mx_{i,k}(s)}w_{\beta}(X)=\sum_{i=0}^{\infty}a_{i, k}(s)\beta^{i}=f_k(s)\,.
\]
\end{proof}
\section{Proof of Corollary \ref{factor}}
The following two lemmas, taken from \cite{chatterjee15}, will be the main ingredients of the proof. Throughout this section let $\ma(n,m)$ be the set of all nondecreasing functions $\alpha:\{0,1,\ldots, n+m\}\ra\{0,1,\ldots, n\}$ such that $\alpha(0)=0$, $\alpha(n+m)=n$ and $\alpha(i+1)-\alpha(i)\le 1$ for each $i<n$.
\begin{lmm}\label{general}
Let $n$ and $m$ be two positive integers. Let $a_0,a_1,\ldots, a_{n-1}$ and $b_0,\ldots, b_{m-1}$ be positive real numbers, and $a_n=b_m=0$. Then 
\begin{align*}
\sum_{\alpha \in \ma(n,m)}\prod_{i=0}^{n+m-1}\frac{1}{a_{\alpha(i)}+b_{i-\alpha(i)}} = \frac{1}{a_0a_1\cdots a_{n-1} b_0b_1\cdots b_{m-1}}\, .
\end{align*}
\end{lmm}

Given two loop sequences $s=(l_1, l_2, \ldots, l_n)$ and $s'=(l'_1, l'_2, \ldots, l'_m)$, define their concatenation as $(s, s'):=(l_1, l_2, \ldots, l_n,l'_1, l'_2, \ldots, l'_m )$. Let $X=(s_0, s_1, \ldots, s_n)$ and $X'=(s'_0, s'_1, \ldots, s'_m)$ be two trajectories. Define the concatenation of these trajectories by a map $\alpha \in \ma (n, m)$ to be a trajectory whose $i$th component is $(s_{\alpha(i)}, s'_{i-\alpha(i)})$ and denote it by $\alpha(X, X')$. Finally, denote the set of all such concatenations by $\ma(X, X')$. 

\begin{lmm}\label{bijec}
Take any two non-null loops $l$ and $l'$. For any $Y\in \mx_0(l,l')$, there exist unique $X\in \mx_0(l)$, $X'\in \mx_0(l')$ and $\alpha \in \ma(X,X')$ such that $Y = \alpha(X, X')$. Conversely, for any $X\in \mx_0(l)$, $X'\in \mx_0(l')$ and $\alpha \in \ma(X,X')$, $\alpha(X,X')\in \mx_0(l,l')$. 
\end{lmm}
Although in \cite{chatterjee15} the second lemma was proven for $SO(N)$ lattice gauge theory, where trajectories in $\mx_0(s)$ contain only splitting and deformation, the proof is still valid in $SU(N)$ setting, where trajectories may also contain expansion operation.
\begin{proof}[Proof of Corollary \ref{factor}]
For any trajectory $X$, let $\delta^+(X)$ and $\delta^-(X)$ be the number of positive and negative deformations of $X$, let $\xi^+(X)$ and $\xi^-(X)$ be the number of positive and negative expansions of $X$, and let $\chi^+(X)$ and $\chi^-(X)$ be the number of positive and negative splittings of $X$. Then, for any vanishing trajectory $X=(s_0, s_1, \ldots, s_n)$
\begin{align*}
w_{\beta}(X)=\frac{C(X)}{|s_0||s_1|\cdots |s_n|}\, .
\end{align*}
where 
$$C(X):=\left(-1\right)^{\chi^+(X)}\left(-\frac{\beta}{2}\right)^{\delta^+(X)}\left(\frac{\beta}{2}\right)^{\delta^-(X)}\left(-\frac{\beta}{2}\right)^{\xi^+(X)}\left(\frac{\beta}{2}\right)^{\xi^-(X)}\, .$$
Now, let $l$ be a loop and consider a trajectory $Y \in \mx_0(l,l)$. By Lemma \ref{bijec} we can write $Y=\alpha(X,X')$ with $X, X' \in \mx_0(l)$ and $\alpha \in \ma(X,X')$ in a unique way. Note that by definition 
$$\delta^+(Y)=\delta^+(X)+\delta^+(X')\, .$$ 
The similar equality also holds for $\delta^-(X)$, $\chi^+(X)$, $\chi^-(X)$, $\xi^+(X)$ and $\xi^-(X)$. Therefore these quantities do not depend on the map $\alpha$, and
$$C(Y)=C(X)C(X')\, .$$
Thus,
\begin{align}
w_{\beta}(Y)=C(X)C(X')\prod_{i=0}^{n+m} \frac{1}{|s_{\alpha(i)}+s'_{i-\alpha(i)}|}\, . \label{i0}
\end{align}
Also note that by Lemma \ref{general}
\begin{align}
\sum_{\alpha \in \ma(X, X')}\prod_{i=0}^{n+m} \frac{1}{|s_{\alpha(i)}+s'_{i-\alpha(i)}|}&=\frac{1}{|s_{0}||s_1|\cdots |s_n||s'_{0}||s'_{1}|\cdots |s'_m|} \label{i1}
\end{align}
By combining equations $\eqref{i0}$ and \eqref{i1} we get
\begin{align*}
\sum_{Y\in \mx_0(l,l)}w_\beta(Y)&=\biggr(\sum_{X \in \mx_0(l)}w_{\beta}(X)\biggr)^2\, .
\end{align*}
So, by Theorem \ref{maintheorem}
\begin{align*}
\lim_{N\to \infty}\frac{\smallavg{W_{l}^2}}{N^2}&=\biggl(\lim_{N\to \infty}\frac{\smallavg{W_l}}{N}\biggr)^2\,.
\end{align*}
Since $|W_{l}|\leq N$ for all $l$,
\begin{align*}
&\left|\lim_{N\to \infty}\frac{\smallavg{W_{l_1}W_{l_2} \cdots W_{l_n}}}{N^n}-\lim_{N\to \infty}\frac{\smallavg{W_{l_1}}}{N}\lim_{N\to \infty}\frac{\smallavg{W_{l_2} \cdots W_{l_n}}}{N^{n-1}}\right|=\lim_{N\to \infty}\frac{|\smallavg{(W_{l_1}-\smallavg{W_{l_1}}) W_{l_2} \cdots W_{l_n}}|}{N^n}\\
&\qquad \leq \lim_{N\to \infty}\frac{\smallavg{|W_{l_1}-\smallavg{W_{l_1}}|}}{N}\leq \lim_{N\to \infty}\frac{\smallavg{(W_{l_1}-\smallavg{W_{l_1}})^2}^{1/2}}{N}=\biggl(\lim_{N\to \infty}\frac{\smallavg{W_{l_1}^2}}{N^2}-\lim_{N\to \infty}\frac{\smallavg{W_{l_1}}^2}{N^2}\biggr)^{1/2}=0\, .
\end{align*}
We finish the proof by applying induction on $n$.
\end{proof}
\section{Proof of Corollary \ref{corresp}}
If $s'$ is an expansion of $s$ by a plaquette $p$, then by Corollary \ref{factor} 
$$f_0(s')=f_0(s)f_0(p)\, .$$
Therefore, for any fixed edge $e$
\begin{align*}
\sum_{\text{expand}\,s} f_0&=\sum_{x\in A_1}\sum_{p\in \cp(e)}f_0(s)f_0(p)+\sum_{x\in B_1}\sum_{p\in \cp(e^{-1})}f_0(s)f_0(p)\\
&\qquad -\sum_{x\in B_1}\sum_{p\in \cp(e)}f_0(s)f_0(p)-\sum_{x\in A_1}\sum_{p\in \cp(e^{-1})}f_0(s)f_0(p)\\
&=t_1f_0(s)\biggl(\sum_{p\in \cp(e)}f_0(p)-\sum_{p\in \cp(e^{-1})}f_0(p)\biggr)=0\, .
\end{align*}
So, equation \eqref{f0mastereq} for the inverse coupling strength $2\beta$ can be written as
\begin{align}
mf_0(s)&=\sum_{\text{split }s}f_0+\beta \sum_{\text{deform }s}f_0\, . \label{f0simpleeq}
\end{align}
This is the functional equation, presented in \cite{chatterjee15}, for the leading term of the asymptotic $1/N$ expansion of $SO(N)$ Wilson loop expectation at coupling strength $\beta$. Furthermore, it was proved in \cite{chatterjee15} that the equation \eqref{f0simpleeq} has a unique solution for all sufficiently small $|\beta|$, depending only on dimension $d$. This finishes the proof of the corollary.
\section*{Acknowledgments}
I thank Sourav Chatterjee for his helpful comments and constant encouragement.

\bibliographystyle{amsplain}

\end{document}